\documentclass[12pt, twoside, leqno]{article}
\usepackage{amsfonts}
\usepackage{mathrsfs}

\usepackage{amsthm,amsmath,amssymb,mathrsfs,amsfonts,graphicx,graphics,latexsym,exscale,cmmib57,dsfont,bbm,amscd}

\usepackage[T1]{fontenc}

\swapnumbers

\newtheorem{thm}{Theorem}[section]
\newtheorem{sthm}{Theorem}[subsection]
\newtheorem*{thm*}{Theorem}
\newtheorem*{cor*}{Corollary}

\newtheorem{lem}[thm]{Lemma}
\newtheorem{slem}[sthm]{Lemma}
\newtheorem*{2.4.3}{\ref{2.4}.3~Lemma}
\newtheorem*{2.8.6}{\ref{2.8}.6~Embedding Lemma}
\newtheorem*{2.9B}{\ref{2.9}B~Lemma}

\newtheorem*{6.1.10}{\ref{6.1.10}~Lemma}

\newtheorem*{5.5.1}{\ref{5.5}.1~Corollary}
\newtheorem*{5.5.2}{\ref{5.5}.2~Corollary}
\newtheorem*{5.5.3}{\ref{5.5}.3~Theorem}
\newtheorem*{5.6a}{\ref{5.6}a~Corollary}
\newtheorem*{5.6b}{\ref{5.6}b~Theorem}
\newtheorem*{5.7.1}{\ref{5.7}.1~Corollary}
\newtheorem*{5.8.3}{\ref{5.8}.3~Proposition}

\newtheorem*{6.2.7B}{\ref{6.2.7}B~Lemma}
\newtheorem*{6.2.7C}{\ref{6.2.7}C~Lemma}
\newtheorem*{6.2.7D}{\ref{6.2.7}D~Theorem}
\newtheorem*{6.2.7E}{\ref{6.2.7}E~Theorem}
\newtheorem*{6.2.7G}{\ref{6.2.7}G~Theorem}

\newtheorem*{6.3.8A}{\ref{6.3.8}A~Corollary}
\newtheorem*{6.3.8B}{\ref{6.3.8}B~Corollary}

\newtheorem*{6.3.14A}{\ref{6.3.14}A~Lemma}
\newtheorem*{6.3.14C}{\ref{6.3.14}C~Theorem}

\newtheorem*{6.4.1B}{\ref{6.4.1}B~Lemma}
\newtheorem*{6.4.4A}{\ref{6.4.4}A~Lemma}

\theoremstyle{definition}

\newtheorem{rem}[thm]{Remark}
\newtheorem{srem}[sthm]{Remark}

\newtheorem*{note*}{Note}
\newtheorem{shs}[thm]{Standing hypotheses}
\newtheorem{bn}[thm]{Basic notation}
\newtheorem{se}[thm]{}
\newtheorem{sse}[sthm]{}

\newtheorem*{5.4.1}{\ref{5.4}.1~Standing hypothesis}
\newtheorem*{5.4.2}{\ref{5.4}.2}
\newtheorem*{6.3.14B}{\ref{6.3.14}B}
\newtheorem*{5.4.3}{\ref{5.4}.3~Examples}
\newtheorem*{5.4.4}{\ref{5.4}.4}
\newtheorem*{5.4.5}{\ref{5.4}.5~Examples}
\newtheorem*{5.7.2}{\ref{5.7}.2~Remark}
\newtheorem*{5.7.3}{\ref{5.7}.3~Remark}
\newtheorem*{6.3.8C}{\ref{6.3.8}C~Remark}
\newtheorem*{5.8.1}{\ref{5.8}.1}
\newtheorem*{6.2.7F}{\ref{6.2.7}F}
\newtheorem*{cau}{Caution}
\newtheorem*{6.14e}{\ref{6.3.4}e}
\newtheorem*{6.14g}{\ref{6.3.4}g}

\newcommand{\z}{\boldsymbol{z}}

\newcommand{\f}{\boldsymbol{f}}

\numberwithin{equation}{section}

\begin{document}


\baselineskip=17pt


\title{Veech's Theorem of $G$ acting freely on $G^\textsl{\texttt{LUC}}$ and Structure Theorem of a.a. flows}

\author{{Xiongping Dai}\\
{\small Department of Mathematics,
Nanjing University,
Nanjing 210093,
P.R. China}\\
{\small E-mail: xpdai@nju.edu.cn}
\and
Hailan Liang\\
{\small Department of Mathematics,
Fuzhou University,
Fuzhou 350116,
P.R. China}\\
{\small E-mail: lghlan@163.com}
\and
Zhengyu Yin\\
{\small Department of Mathematics,
Nanjing University,
Nanjing 210093,
P.R. China}\\
{\small E-mail: 455001821@qq.com}
}
\date{July 13, 2023}

\maketitle


\renewcommand{\thefootnote}{}


\footnote{\emph{Key words and phrases}: Maximal ideal space, $G^\textsl{\texttt{LUC}}$, Veech's Theorem, a.a. functions.}

\renewcommand{\thefootnote}{\arabic{footnote}}
\setcounter{footnote}{0}


\begin{abstract}
Veech's Theorem claims that if $G$ is a locally compact\,(LC) Hausdorff topological group, then it may act freely on $G^\textsl{\texttt{LUC}}$. We prove Veech's Theorem for $G$ being only locally quasi-totally bounded, not necessarily LC.
And we show that the universal a.a. flow is the maximal almost 1-1 extension of the universal minimal a.p. flow and is unique up to almost 1-1 extensions. In particular, every endomorphism of Veech's hull flow induced by an a.a. function is almost 1-1; for $G=\mathbb{Z}$ or $\mathbb{R}$, $G$ acts freely on its canonical universal a.a. space. Finally, we characterize Bochner a.a. functions on a LC group $G$ in terms of Bohr a.a. function on $G$ (due to Veech 1965 for the special case that $G$ is abelian, LC, $\sigma$-compact, and first countable).
\end{abstract}

\section{Introduction}
The maximal ideal space, also named as the structure space, of a commutative Banach algebra $\mathcal{A}$ is an important tool for the study of ``universal elements'', for example, the universal $\{\textrm{minimal}\}$ $\{\textrm{minimal equicontinuous}\}$ $\{\textrm{point-transitive}\}$ flow, in the abstract topological dynamical systems (cf.~\cite{E69, G76, De}). Now we will present self-closely a revisit to this concept in the special case that $\mathcal{A}$ is a Banach subalgebra of uniformly continuous functions on groups/semigroups in $\S$\ref{s2}, $\S$\ref{s3}, and $\S$\ref{s4}. The new ingredient is that we shall introduce the analysis on Hausdorff left-topological groups/semigroups instead of Hausdorff topological groups and  discrete semigroups. We will give applications in $\S$\ref{s5} and $\S$\ref{s6}.

Let $G^\textsl{\texttt{LUC}}$ be the left-uniformly continuous compactification of a Hausdorff topological group $G$ (see Def.~\ref{2.8}.4). Veech's Theorem claims that if $G$ is locally compact, then $G$ acts freely on $G^\textsl{\texttt{LUC}}$ (\cite[Thm.~2.2.1]{V77}). In $\S$\ref{s5}, we shall extend Veech's Theorem for $G$ being only a ``locally quasi-totally bounded'' group, not necessarily locally compact (see Def.~\ref{5.4}.4 and Thm~\ref{5.7}). This implies that the universal $\{\textrm{point-transitive}\}$$\{\textrm{minimal}\}$ compact flow of any locally quasi-totally bounded group is free.

In $\S$\ref{s6.1} we shall study left almost automorphic (a.a.) functions on Hausdorff left-topological groups. Further, in $\S\S$\ref{s6.2} and \ref{s6.3} we shall simply reprove Veech's Structure Theorem of a.a. flows and a.a. functions (Thm.~\ref{6.2.6} and Thm.~\ref{6.3.11}); and moreover, show that the universal a.a. flow of $G$ is existent and is exactly the maximal almost 1-1 extension of the universal minimal equicontinuous flow of $G$ unique up to almost 1-1 extensions (Thm.~\ref{6.2.7}D). In particular, we will prove that every endomorphism of Veech's hull flow induced by an a.a. function is almost 1-1 (Thm.~\ref{6.3.12}). In addition, for $G=\mathbb{Z}$ or $\mathbb{R}$, $G$ acts freely on its canonical universal a.a. space (Corollary to Thm.~\ref{6.2.8}).

Recall that if $f$ is a (von Neumann) almost periodic function on a non-abelian group $G$ (Def.~\ref{6.3.14}), then the left hull of $f$,
$\overline{Gf}^\textrm{p}=\overline{Gf}^{\|\cdot\|_\infty}$,
need not be a compact topological group. However, we shall show in $\S$\ref{s6.3} that $f$ is (von Neumann) almost periodic iff Veech's bi-hull,
$$
\overline{G\xi_f}^\textrm{p}=\overline{G\xi_f}^{\|\cdot\|_\infty}\quad \textrm{where }\xi_f(s,t)=f(st)\ \forall s,t\in G,
$$
is a compact topological group (see Thm.~\ref{6.3.14}C). In fact, almost automorphic functions may be characterized in terms of almost periodic functions (see Thm.~\ref{C3} in Appendix~\ref{C}).

Finally, using bi-hulls of uniformly continuous functions on $G$, we shall characterize in $\S$\ref{s6.4} Bochner a.a. functions on a locally compact group $G$ in terms of Bohr a.a. function on $G$ (due to Veech 1965 for the special case that $G$ is abelian, locally compact, $\sigma$-compact, and first countable); see Thm.~\ref{6.4.6}.

\begin{shs}\label{1.1}
\item \textbf{a.}  Unless specified otherwise, assume $T$ is a \textsl{noncompact, Hausdorff, left-topological monoid with identity $e$}; that is, $T$ is a noncompact Hausdorff space, which has a semigroup structure: $T\times T\xrightarrow{(x,y)\mapsto xy}T$, with identity $e$, such that the left-translate $L_y\colon T\xrightarrow{x\mapsto yx}T$ is continuous for all $y\in T$.
\item \textbf{b.}  Given $\tau\in T$, $\mathfrak{N}_\tau(T)$ stands for the filter of neighborhoods of $\tau$ in $T$; and LC is short for ``locally compact''.
\end{shs}

For example, if $G$ is a group and $\lambda\colon X\times G\rightarrow X$ is a right-action nonequicontinuous phase mapping on a noncompact Hausdorff space $X$, then
$T=\overline{\{X\xrightarrow{x\mapsto xg}X\,|\,g\in G\}}$,
equipped with the topology of pointwise convergence,
is a noncompact, Hausdorff, left-topological semigroup under the composition of mappings. In this case there need generally not exist left/right-uniformity structures on the left-topological semigroup $T$.

\begin{bn}\label{1.2}
As usual, we say that $(T,X)$ is a \textit{semiflow}, denoted $\mathscr{X}$ or $T\curvearrowright X$, if there exists a continuous mapping $\pi\colon T\times X\xrightarrow{(t,x)\mapsto tx}X$,
called the phase mapping and where $X$ is a Hausdorff space, such that $ex=x$ and $(st)x=s(tx)$ for all $x\in X$ and $s,t\in T$. Let $\Delta_X=\{(x,x)\,|\,x\in X\}$.
\item \textbf{a.} If $X$ is compact, $\mathscr{X}$ will be referred to as a \textit{compact semiflow}. If meanwhile $X\curvearrowleft T$ is a compact right-action semiflow, with phase mapping $(x,t)\mapsto xt$, such that $(sx)t=s(xt)$ for all $s,t\in T$ and $x\in X$, then $(T,X,T)$ will be called a \textit{compact bi-semiflow}.

\item \textbf{b.} If the orbit closure $\overline{Tx}$ is compact for all $x\in X$, then $\mathscr{X}$ will be called \textit{Lagrange stable} (cf., e.g.,~\cite{NS}).
\item \textbf{c.} The pair $(\mathscr{X},x)$ will be called an {\it ambit} of $T$ as in \cite{A88} if $\mathscr{X}$ is a semiflow with $\overline{Tx}=X$; and in this case, $x$ is called a \textit{transitive point} of $\mathscr{X}$. If no point $y\in\overline{Tx}\setminus\{x\}$ such that $\overline{T(x,y)}\cap\Delta_X\not=\emptyset$, then $x$ will be called a \textit{distal point} for $\mathscr{X}$.
\item \textbf{d.} If the phase mapping $\pi\colon (t,x)\mapsto tx$ is only separately continuous here, then $\mathscr{X}$ or $T\curvearrowright X$ will be called a {\it dynamic} instead of a semiflow.
\item \textbf{e.} For a Lagrange stable dynamic $\mathscr{X}$, a point $x\in X$ is {\it almost periodic} (a.p.) iff $\overline{Tx}$ is a \textit{minimal set} under $\mathscr{X}$ (i.e., $\overline{Ty}=\overline{Tx}$ for all $y\in\overline{Tx}$). It is a well-known fact that a distal point is an a.p. point in a Lagrange stable dynamic~\cite{E69, A88}.

\item \textbf{f.} Given $x\in X$ and $U\in\mathfrak{N}_x(X)$, set $N_T(x,U)=\{t\in T\,|\,tx\in U\}$ as in \cite{F81}.

\item If $T$ is a group, a semiflow will be called a \textit{flow} in the later.
\end{bn}

\section{Definition of the maximal ideal spaces}\label{s2}
A collection $\mathcal{F}$ of complex-valued functions on a space $W$ is called an \textit{algebra} if for any $f,g\in\mathcal{F}$ and $c\in\mathbb{C}$ we have $cf+g, f g\in\mathcal{F}$. Next, $\mathcal{F}$ is said to \textit{separate points} of $W$ if for all $w_1\not=w_2$ in $W$, we have $f(w_1)\not=f(w_2)$ for some $f\in\mathcal{F}$. We say that $\mathcal{F}$ \textit{separates points and closed sets in $W$} if for all point $w\in W$ and closed set $A\subset W$ with $w\notin A$, there exists an $f\in\mathcal{F}$ such that $f(w)=0$, $f=1$ on $A$, and $0\le f\le 1$.
$\mathcal{F}$ is called {\it self-adjoint} if $\bar{f}=g-ih\in\mathcal{F}$ whenever $f=g+ih\in\mathcal{F}$ where $g$ and $h$ are real-valued functions on $W$.

By $\mathcal{C}(W)$ we denote the set of all continuous complex-valued functions on $W$; let $\|f\|_\infty=\sup_{x\in W}|f(x)|$ be the supremum norm of $f\in\mathcal{C}(W)$ and write $\mathcal{C}_b(W)$ for $\{f\in \mathcal{C}(W)\,|\,\|f\|_\infty<\infty \}$.

Using the maximal ideal spaces, we shall prove that every uniformly continuous complex-valued bounded function on $T$ may come from a compact ambit of $T$\,(see Thm.~\ref{3.3}), and, construct the universal a.a. flows of groups (see Def.~\ref{6.2.7}a).

\begin{thm}[{Stone-Weierstrass Theorem, cf.~\cite[Thm.~V.8.1]{C}}]\label{2.1}
Let $W$ be a compact Hausdorff space. Let $\mathcal{A}$ be a subalgebra of $\mathcal{C}(W)$, which separates the points of $W$, contains the constants and is self-adjoint. Then $\mathcal{A}$ is norm dense in $\mathcal{C}(W)$.
\end{thm}

\begin{se}[Left-uniformly continuous function on $T$]\label{2.2}
Let $\mathbf{1}\in\mathcal{C}_b(T)$ be the constant function $\mathbf{1}(x)=1\ \forall x\in T$.
For all $f\in\mathcal{C}_b(T)$ and $t\in T$,
\begin{enumerate}
\item[\textbf{1.}] $ft$ is defined by $T\xrightarrow{x\mapsto f(tx)}\mathbb{C}$, called the \textit{right $t$-shift} of $f$. Clearly, $ft\in\mathcal{C}_b(T)$ since $x_n\to x$ in $T$ implies that $tx_n\to tx$;

\item[\textbf{2.}] $tf$ is defined by $T\xrightarrow{x\mapsto f(xt)}\mathbb{C}$, called the \textit{left $t$-shift} of $f$. Notice here that $tf$ need not belong to $\mathcal{C}_b(T)$, since $T$ need generally not be a right-topological semigroup.

\item[\textbf{3.}] Although there is no any uniformity structure on $T$ here, we can introduce left-uniformly continuous function on $T$ as follows: Let
\begin{equation*}
\verb"LUC"(T)=\{f\in\mathcal{C}_b(T)\,|\,ft_n\to ft\textrm{ $\|\cdot\|_\infty$-uniformly as }t_n\to t\textrm{ in }T\},
\end{equation*}
called the set of \textit{left-uniformly continuous} complex-valued bounded functions on $T$.
\begin{enumerate}
\item[\textbf{i.}] If $T$ is a Hausdorff topological group, then $f\in\verb"LUC"(T)$ iff to any $\varepsilon>0$ there exists some $V\in\mathfrak{N}_e(T)$ such that $|f(x)-f(vx)|<\varepsilon$ for all $x\in T$ and all $v\in V$ (cf., e.g.,~\cite{K66, BF}).

\item[\textbf{ii.}] If $T$ has a ``left-uniformity'' $\mathscr{U}^L$ (i.e., a uniformity on $T$ such that $\{R_t\colon T\xrightarrow{x\mapsto xt}T\,|\,t\in T\}$ is $\mathscr{U}^L$-equicontinuous on $T$; for instance, $T=\mathbb{R}_+$), then $f\in\verb"LUC"(T)$ whenever $f\in \mathcal{C}_b(T)$ is $\mathscr{U}^L$-uniformly continuous.

\begin{proof}
Assume $f$ is $\mathscr{U}^L$-uniformly continuous. Let $\varepsilon>0$. Then there is a $V\in\mathscr{U}^L$ such that $|f(x)-f(y)|<\varepsilon$ $\forall (x,y)\in V$. Pick some $U\in\mathscr{U}^L$ such that $(x,y)\in U$ implies $(xt,yt)\in V$ $\forall t\in T$. So, as $t_n\to t$ in $T$, we have that $(t_n,t)\in U$ eventually so that $|ft_n(x)-ft(x)|<\varepsilon$ $\forall x\in X$. Thus, $f\in\verb"LUC"(T)$.
\end{proof}
\end{enumerate}
For example, if $T$ is a topological group, then
$\mathscr{U}^L=\{V_L\,|\,V\in\mathfrak{N}_e(T)\}$, where $V_L=\{(x,y)\,|\,yx^{-1}\in V\}$, is the canonical left-uniformity on $T$.

\item[\textbf{4.}] Clearly, $\verb"LUC"(T)$ is a norm closed subalgebra of the commutative Banach algebra $\mathcal{C}_b(T)$, which is self-adjoint and contains the constants; and moreover, $\verb"LUC"(T)$ is left and right $T$-invariant (i.e., if $f\in\verb"LUC"(T)$, then $sft\in\verb"LUC"(T)$ for all $s,t\in T$) (cf.~\cite{K66} or \cite[Lem.~3.1, p.84]{E69} for $T$ a topological group).

\item[\textbf{5.}] We could similarly define $\texttt{RUC}(T)$---the set of \textit{right-uniformly continuous} complex-valued bounded functions on $T$ as follows:
\begin{equation*}
\verb"RUC"(T)=\{f\in\mathcal{C}_b(T)\,|\,t_nf\to tf\textrm{ $\|\cdot\|_\infty$-uniformly as }t_n\to t\textrm{ in }T\}.
\end{equation*}
\item[\textbf{6.}] Let $\verb"UC"(T)=\verb"LUC"(T)\cap\verb"RUC"(T)$. If $\hat{T}$ is the continuous character group of $T$, then $\hat{T}\subseteq\verb"UC"(T)$.
\end{enumerate}

\begin{cau}
$\verb"LUC"(T)$ is sometimes written as $\texttt{RUC}(T)$ in some literature\,(cf.~\cite{De}). We agree with \cite{K, K66, Py, BF} here.
\end{cau}
\end{se}

\begin{se}\label{2.3}
If $\mathscr{X}$ is a dynamic of $T$ and $x\in X$, then $f\in\mathcal{C}_b(T)$ is called \textit{coming from} $(\mathscr{X},x)$, denoted $f\in\mathbb{F}(\mathscr{X},x)$, whenever there is an $F\in\mathcal{C}(X)$ such that $f(t)=F(tx)$ for all $t\in T$. We write $F^x(t)=F(tx)$. Clearly,
$$
tF^x=F^{tx}\quad\textrm{and}\quad t\mathbb{F}(\mathscr{X},x)=\mathbb{F}(\mathscr{X},tx)\quad \forall x\in X, t\in T, F\in\mathcal{C}(X).
$$
\end{se}

\begin{se}\label{2.4}

\item[\;\;\textbf{1.}] If $\mathscr{X}$ is a Lagrange stable semiflow of $T$, then $\mathbb{F}(\mathscr{X},x)\subseteq\verb"LUC"(T)$ for all $x\in X$ (see \cite[Prop.~IV.5.6]{De} for $\mathscr{X}$ a compact flow with $T$ a topological group).

\begin{proof}
We may assume $X$ is compact with uniformity $\mathscr{U}$ by replacing $X$ by $\overline{Tx}$ if necessary.	
Let $F\in\mathcal{C}(X)$ and set $f=F^x$. Given $\varepsilon>0$ and $s_n\to s$ in $T$, there is an open $\alpha\in\mathscr{U}$ such that $|F(x^\prime)-F(x^{\prime\prime})|<\varepsilon$ whenever $(x^\prime,x^{\prime\prime})\in\alpha$. Further, there is a set $V\in\mathfrak{N}_s(T)$ such that $Vx^\prime\times \{sx^\prime\}\subset\alpha^2$ for all $x^\prime\in X$. Indeed, by continuity of
$$(T\times X)\times(T\times X)\xrightarrow{((t_1,x_1),(t_2,x_2))\mapsto(t_1x_1,t_2x_2)}X\times X,$$
for all $x^\prime\in X$ there exist some $V_{x^\prime}\in\mathfrak{N}_s(T)$ and some $U_{x^\prime}\in\mathfrak{N}_{x^\prime}(X)$ such that $V_{x^\prime}U_{x^\prime}\times V_{x^\prime}U_{x^\prime}\subset\alpha$. Since $X$ is compact, there exists a finite set $\{x_1,\dotsc, x_n\}\subset X$ with $X=\bigcup_{i=1}^n U_{x_i}$. Clearly, $V=\bigcap_{i=1}^nV_{x_i}$ is as desired.
This implies that $fs_n(t)\to fs(t)$ uniformly for $t\in T$.
\end{proof}

\item[\;\;\textbf{2.}] Let $(\mathscr{X},x)$ be a compact ambit of $T$, then we can define
$$F^{\mathscr{X}}=\{F^y\,|\,y\in X\}\quad \textrm{and}\quad \varrho^x\colon\mathcal{C}(X)\xrightarrow{F\mapsto F^x}\verb"LUC"(T).$$
Then $\mathbb{F}(\mathscr{X},x)=\varrho^x[\mathcal{C}(X)]$, and $\varrho^x$ is an isometric algebra homomorphism from $\mathcal{C}(X)$ into $\verb"LUC"(T)$. Thus, $\mathbb{F}(\mathscr{X},x)$ is a norm closed right-invariant subalgebra of $\verb"LUC"(T)$. In fact,
\begin{equation*}
F^xt=(Ft)^x\in\mathbb{F}(\mathscr{X},x)\quad \forall F\in\mathcal{C}(X), t\in T.
\end{equation*}
Moreover, $\mathbb{F}(\mathscr{X},x)$ is self-adjoint and contains the constants.

\begin{2.4.3}
Let $\mathscr{X}$ be a compact semiflow of $T$ and $x_0,x_1\in X$ with $\overline{Tx_0}=\overline{Tx_1}=X$. Let $F\in\mathcal{C}(X)$, $f_0=F^{x_0}$, and $f_1=F^{x_1}$. Then $$\overline{Tf_0}^\textrm{p}=F^\mathscr{X}=\overline{Tf_1}^\textrm{p}\subseteq\texttt{LUC}(T)$$
where $\overline{B}^\textrm{p}$ denotes the pointwise closure of a set $B$ in $\mathbb{C}^T$.
\end{2.4.3}

\begin{proof}
To show the desired equalities, it is sufficient to prove $F^\mathscr{X}={\overline{Tf_0}}^\textrm{p}$. Let $t_nf_0\to_\textrm{p}\phi$ in $\mathbb{C}^T$. Then for $t\in T$, $t_nf_0(t)\to\phi(t)$ and $f_0(tt_n)\to\phi(t)$. Thus $F(tt_nx_0)\to\phi(t)$. We can assume (a subnet of) $t_nx_0\to x\in X$, for $X$ is compact. Then $F^x(t)=\phi(t)$. This shows ${\overline{Tf_0}}^\textrm{p}\subseteq F^\mathscr{X}$. On the other hand, let $x\in X$. There is a net $t_n\in T$ such that $t_nx_0\to x$, for $Tx_0$ is dense in $X$. Clearly, $t_nf_0\to_\textrm{p} F^x$ and ${\overline{Tf_0}}^\textrm{p}\supseteq F^\mathscr{X}$. Thus, ${\overline{Tf_0}}^\textrm{p}=F^\mathscr{X}$. The proof is completed.
\end{proof}

\item[\;\;\textbf{4.}] Even for a minimal compact flow $\mathscr{X}$ with $T$ abelian, there exists no the equality  $\mathbb{F}(\mathscr{X},x_0)=\mathbb{F}(\mathscr{X},x_1)$ for $x_0\not=x_1$ in general, according to Theorem~\ref{3.1} below. If $T$ is abelian and $f\in\mathbb{F}(\mathscr{X},x_0)$, then it is easy to verify $Tf=fT\subset\mathbb{F}(\mathscr{X},x_0)$. However, $\mathbb{F}(\mathscr{X},x_0)$ need not be pointwise or compact-open closed, so $\overline{Tf}^\textrm{p}\nsubseteq\mathbb{F}(\mathscr{X},x_0)\varsubsetneq\bigcup_{F\in\mathcal{C}(X)}F^\mathscr{X}$ generally.

\medskip
Note here that if $\mathscr{X}=X\curvearrowleft T$ is defined by $T$ acting from right-hand side on $X$, then we have that $\mathbb{F}(\mathscr{X},x)\subseteq\texttt{RUC}(T)$. In this case, we have for $F\in\mathcal{C}(X)$ and $x\in X$ that $F_x(t)=F(xt)$.

\begin{cau}
The pointwise topology is weaker than the weak topology on $\mathbb{F}(\mathscr{X},x)$. Indeed, it holds that $\mathbb{F}(\mathscr{X},x)\subsetneq\overline{\mathbb{F}(\mathscr{X},x)}^\textrm{p}$.
\end{cau}
\end{se}

\begin{se}\label{2.5}
	It is easy to verify that
	\begin{enumerate}
		\item[] $\lambda\colon \verb"LUC"(T)\times T\rightarrow\verb"LUC"(T),\ (f,t)\mapsto ft$
	\end{enumerate}
	is a jointly continuous phase mapping; that is, if $f_n\to_{\|\cdot\|_\infty}^{}f$ in $\verb"LUC"(T)$ and $t_n\to t$ in $T$, then $f_nt_n\to_{\|\cdot\|_\infty}^{}ft$ in $\verb"LUC"(T)$.
	
	\begin{cau}
		If we replace $\verb"LUC"(T)$ by $\mathcal{C}_b(T)$, then $\lambda$ need not be jointly continuous unless $T$ is a LC Hausdorff topological group.
	\end{cau}
\end{se}

\begin{se}\label{2.6}
	Let $\verb"LUC"(T)^*$ be the dual space of $\verb"LUC"(T)$---the space of bounded linear functionals on the Banach space $(\verb"LUC"(T),\|\cdot\|_\infty)$, endowed with the weak-$*$ topology. Then
	\begin{enumerate}
		\item[] $\verb"LUC"(T)^*\times \verb"LUC"(T)\rightarrow\mathbb{C},\ (\mu,f)\mapsto\langle\mu,f\rangle=\mu(f)$
	\end{enumerate}
	is jointly continuous.
	
	\begin{proof}
		If $\mu_n\to_{\textrm{wk}^*}\mu$ in $\verb"LUC"(T)^*$ and $f_n\to_{\|\cdot\|_\infty}^{}f$ in $\verb"LUC"(T)$, then by the uniform boundedness principle\,(cf.~\cite[Thm.~III.14.1]{C}) it follows that
	$$
	|\mu_n(f_n)-\mu(f)|\le\|\mu_n\|\cdot\|f_n-f\|_\infty+|\mu_n(f)-\mu(f)|\to0.
	$$
	The proof is complete.
\end{proof}
\end{se}

\begin{se}\label{2.7}
Given $\mu\in\verb"LUC"(T)^*$ and $t\in T$ we define $t\mu$ by
\begin{enumerate}
\item[] $\langle t\mu,f\rangle=\langle\mu,ft\rangle\ \forall f\in\verb"LUC"(T)$.
\end{enumerate}
Then
\begin{enumerate}
\item[] $T\times \verb"LUC"(T)^*\rightarrow\verb"LUC"(T)^*,\ (t, \mu)\mapsto t\mu$
\end{enumerate}
is a jointly continuous phase mapping.
	
\begin{proof}
Let $t_n\to t$ in $T$ and $\mu_n\to_{\textrm{wk}^*}\mu$ in $\verb"LUC"(T)^*$ then for all $f\in\verb"LUC"(T)$ we have $|\langle t_n\mu_n,f\rangle-\langle t\mu,f\rangle|=|\mu_n(ft_n)-\mu(ft)|\to0$
		by \ref{2.6}, for $ft_n\to_{\|\cdot\|_\infty}^{}ft$ by \ref{2.5}. Thus, $t_n\mu_n\to_{\textrm{wk}^*}t\mu$. Clearly, $(st)\mu=s(t\mu)$ for all $s,t\in T$ and $\mu\in\verb"LUC"(T)^*$.
\end{proof}
\end{se}

\begin{se}\label{2.8}
Let $\mathbb{V}$ be a norm closed subalgebra of $\verb"LUC"(T)$ such that $\mathbb{V}$ is right $T$-invariant (i.e., $ft\in\mathbb{V}$ whenever $t\in T$ and $f\in\mathbb{V}$), and that $\mathbb{V}$ is self-adjoint and contains the constants.

\item[\;\;\textbf{1.}] Given $t\in T$, define $t^*\in\mathbb{V}^*$ by $t^*\colon\mathbb{V}\xrightarrow{f\mapsto f(t)}\mathbb{C}$. Clearly, $\|t^*\|=1$ because of  $\mathbf{1}\in \mathbb{V}$. Write
		\begin{enumerate}
		\item[] $|\mathbb{V}|=\overline{\{t^*\,|\,t\in T\}}^{\textrm{wk}^*}$,
        \end{enumerate}
which is a compact Hausdorff space by Alaoglu's theorem\,(cf.~\cite[Thm.~V.3.1]{C}).
Since $st^*(f)=f(st)=(st)^*(f)$ for all $s,t\in T$ and $f\in\mathbb{V}$, then by \ref{2.7}  it follows that $st^*\in|\mathbb{V}|$ and that
	\begin{enumerate}
		\item[] $\rho\colon T\times|\mathbb{V}|\rightarrow|\mathbb{V}|,\ (t,\mu)\mapsto t\mu$
	\end{enumerate}
	is a jointly continuous phase mapping (cf.~\cite{K66} or \cite[Lem.~3.2, p.85]{E69} for $T$ a topological group) so that
	
\item[\;\;\textbf{2.}] $(\mathscr{V},e^*):=(T,|\mathbb{V}|,e^*)$ is a compact ambit. In fact, if $\tau\in T$ such that $\overline{T\tau}=T$, then $(\mathscr{V},\tau^*)$ is a compact ambit of $T$.
		
\item[\;\;\textbf{3.}] Now for every $f\in\mathbb{V}$, define the \textit{Fourier transformation} of $f$ as follows:
		\begin{enumerate}
			\item[] $\hat{f}\colon|\mathbb{V}|\rightarrow\mathbb{C},\ \mu\mapsto\langle\mu,f\rangle$.
		\end{enumerate}
		Clearly,
		\begin{enumerate}
			\item[] $\hat{{}}\,\colon\mathbb{V}\rightarrow\mathcal{C}(|\mathbb{V}|),\ f\mapsto \hat{f}$
		\end{enumerate}
		is a linear mapping such that:
		\begin{enumerate}
			\item[a.] $\|\hat{f}\|=\|f\|_\infty\ \forall f\in\mathbb{V}$ (by \ref{2.6} and \ref{2.8}.1).
			\item[b.] $\widehat{fg}=\hat{f}\hat{g}\ \forall f,g\in\mathbb{V}$; i.e., $\hat{{}}$\ is multiplicative.
			
			Indeed, we have for every $t\in T$ and all $f,g\in\mathbb{V}$ that $$\widehat{fg}(t^*)=t^*(fg)=f(t)g(t)=\hat{f}(t^*)\hat{g}(t^*);$$
			thus, $\widehat{fg}=\hat{f}\hat{g}$.
			\item[c.] If $f\equiv c$ then $\hat{f}\equiv c$; and $\hat{\bar{f}}=\bar{\hat{f}}$. If $\alpha\not=\mu$ in $|\mathbb{V}|\,(\subseteq\mathbb{V}^*)$, then $\alpha(f)\not=\mu(f)$ for some $f\in\mathbb{V}$ so that $\hat{f}(\alpha)\not=\hat{f}(\mu)$.
			\item[d.] $\widehat{ft}=\hat{f}t\ \forall f\in\mathbb{V}$ and $t\in T$.
		\end{enumerate}
This shows that $\mathcal{A}=\{\hat{f}\,|\,f\in\mathbb{V}\}$ is a norm closed subalgebra of $\mathcal{C}(|\mathbb{V}|)$ such that $\mathcal{A}$ separates points of $|\mathbb{V}|$, is self-adjoint and contains the constants.

\item[\;\;\textbf{\textbf{4.}}] In particular, write
\begin{enumerate}
\item[]	$T^\textsl{\texttt{LUC}}=|\verb"LUC"(T)|$ and $\mathscr{T}^\textsl{\texttt{LUC}}=(T,T^\textsl{\texttt{LUC}})$.
\end{enumerate}
Let $\iota\colon T\rightarrow T^\textsl{\texttt{LUC}}$ be defined by $t\mapsto t^*\ \forall t\in T$. Clearly, $\iota$ is continuous.  Indeed, if $t_n\to t$ in $T$, then $f(t_n)\to f(t)$ for all $f\in\mathbb{V}$ so that $t_n^*\to_{\textrm{wk}^*}t^*$.
		As usual, $(\iota,T^\textsl{\texttt{LUC}})$ is called the \textit{$\texttt{LUC}$-compactification of $T$}, such that
		\begin{enumerate}
			\item[i.] $\iota[T]$ is dense in $T^\textsl{\texttt{LUC}}$;
			\item[ii.] if $f\colon T\rightarrow Y$ is a left-uniformly continuous function on $T$ to a compact Hausdorff space $Y$ (i.e., $ft_n\to ft$ uniformly as $t_n\to t$ in $T$), then $f$ may be extended uniquely to a continuous function $f^\textsl{\texttt{LUC}}$ on $T^\textsl{\texttt{LUC}}$ to $Y$.
		\end{enumerate}
		(See Knapp 1966~\cite{K66} and Flor 1967~\cite{Fl} for $T$ a topological group.)

\item[\;\;\textbf{5.}] If $\verb"LUC"(T)$ separates points of $T$, then $\iota\colon T\rightarrow T^\textsl{\texttt{LUC}}$ is obviously 1-1 (i.e, $\iota$ is an imbedding).

\begin{2.8.6}
If $\verb"LUC"(T)$ separates points and closed sets in $T$\,(e.g., $T$ is a topological group; cf.~Cor.~\ref{5.5}.2), then $\iota\colon T\rightarrow\iota[T]\,(\subseteq T^\textsl{\texttt{LUC}})$ is an open 1-1 onto map. Moreover,
		\begin{enumerate}
			\item[i.] $T$ is LC iff $\iota[T]$ is open in $T^\textsl{\texttt{LUC}}$; and
			\item[ii.] if $T$ is discrete and $A, B\subset T$, then $\overline{A\cap B}=\overline{A}\cap \overline{B}$ relative to $T^\textsl{\texttt{LUC}}$.
		\end{enumerate}
Consequently, if $T$ is a LC topological group, then $\mu\in T^\textsl{\texttt{LUC}}$ is a transitive point for $\mathscr{T}^\textsl{\texttt{LUC}}$ iff $\mu=t^*$ for some $t\in T$.
\end{2.8.6}
	
\begin{proof}
Clearly, $\iota$ is a bijection. Now, let $\tau\in T$ and let $U\in\mathfrak{N}_\tau(T)$. To prove openness of $\iota\colon T\rightarrow\iota[T]$, it is sufficient to show that $\iota[U]$ contains the intersection of $\iota[T]$ and a neighborhood of $\tau^*$ in $T^\textsl{\texttt{LUC}}$. By the separating hypothesis, we can choose $f\in\verb"LUC"(T)$ such that $f(\tau)=0$, $f=1$ on $T\setminus U$, and $0\le f\le 1$. By shrinking $U$ if necessary, we may assume $0\le f<1$ on $U$.	
		Let
		$V=\{\mu\in T^\textsl{\texttt{LUC}}\,|\,|\mu(f)-\tau^*(f)|<1\}$.
		Since $V$ is an open neighborhood of $\tau^*$ in $T^\textsl{\texttt{LUC}}$ such that $V\cap\iota[T]\subseteq \iota[U]$, $\iota$ is an open map of $T$ onto $\iota[T]$.
		
		\item  i). If $T$ is additionally LC, then $\iota[T]$ is open in $T^\textsl{\texttt{LUC}}$ by \cite[Problem~5.G]{K}.\footnote{\textbf{Lemma}\,(\cite[5.G]{K}).\ {\it If $X$ is a Hausdorff space and $Y$ is a dense LC subspace of $X$, then $Y$ is open in $X$.}
\begin{proof}
Suppose to the contrary that there is a point $y\in Y$ and a net $x_n\in X\setminus Y$ with $x_n\to y$. Let $U\in\mathfrak{N}_y(Y)$ be compact. Then $\textrm{cls}_XU=U$ and there exists an open $V\in\mathfrak{N}_y(X)$ such that $V\cap Y\subseteq U$. Clearly, $x_n\in V\setminus U$ eventually. Since $V\setminus U$ is open in $X$ and $Y$ is dense in $X$, we may pick $x_n^\prime\in Y$ with $x_n^\prime\in V\setminus U$ such that $x_n^\prime\to y$. However, since $U\in\mathfrak{N}_y(Y)$, $x_n^\prime\in U$ eventually. So we reach a contradiction.
\end{proof}}
		
		\item ii). Assume $T$ is discrete and $A,B\subset T$. First $\overline{A\cap B}\subseteq\overline{A}\cap\overline{B}$. To show the opposite inclusion, let $\mu\in\overline{A}\cap\overline{B}$. Then there are nets $\{a_n\}$ in $A$ and $\{b_n\}$ in $B$ with $a_n^*\to\mu$ and $b_n^*\to\mu$.
		If there is a subnet $\{a_i\}$ from $\{a_n\}$ with $\{a_i\}\subset B$, then $\mu\in\overline{A\cap B}$. However, if there is no such a subnet, then we can find some $n_0$ such that $a_n\not\in B$ for all $n\ge n_0$. Now define $f\in \verb"LUC"(T)$ such that $f=0$ on $B$ and $f(a_n)=1$ for all $n\ge n_0$. Then $1=\lim_nf(a_n)=\mu(f)=\lim_nf(b_n)=0$, a contradiction.
		The proof is completed.
	\end{proof}
\end{se}

\begin{se}\label{2.9}
Under the context of \ref{2.8}, it follows from Theorem~\ref{2.1} that $\mathcal{C}(|\mathbb{V}|)$ is isometric onto $\mathbb{V}$. Now $|\mathbb{V}|$ is called the \textit{maximal ideal space} or \textit{structure space} of $\mathbb{V}$. It is clear that
	\begin{equation*}
	\mathbb{F}(\mathscr{V},e^*)=\mathbb{V}\leqno{(\ref{2.9}\textrm{a})}
	\end{equation*}
	(see \cite[Thm.~IV.5.18]{De} for $T$ a Hausdorff topological group). Indeed, for all $f\in\mathbb{V}$ or $\hat{f}\in\mathcal{C}(|\mathbb{V}|)$, we have that $\hat{f}^{e^*}(t)=\hat{f}(te^*)=\hat{f}(t^*)=f(t)$ for all $t\in T$. In fact, if $T$ is a group, we can expect more as follows:

\begin{2.9B}
If $T$ is a left-topological group, then $\mathbb{F}(\mathscr{T}^\textsl{\texttt{LUC}},\tau^*)=\verb"LUC"(T)$ for all $\tau\in T$.
\end{2.9B}

\begin{proof}
Let $\tau\in T$. Clearly $\mathbb{F}(\mathscr{T}^\textsl{\texttt{LUC}},\tau^*)\subseteq\verb"LUC"(T)$ by \ref{2.4}.1 and \ref{2.8}.2. Now let $f\in\verb"LUC"(T)$. Then $\tau^{-1}f\in\verb"LUC"(T)$ (by Lem.~\ref{4.2} below) so that
$$
\widehat{\tau^{-1}f}^{\tau^*}(t)=\tau^{-1}f(t\tau)=f(t)\quad \forall t\in T.
$$
Thus, $\widehat{\tau^{-1}f}^{\tau^*}=f$ and $\texttt{LUC}(T)\mathbb{F}\subseteq(\mathscr{T}^\textsl{\texttt{LUC}},\tau^*)$.
The proof is completed.
\end{proof}
\end{se}

Therefore, every $f\in\verb"LUC"(T)$ may come from a compact ambit of $T$. See \cite[$\S\S$20\,$\thicksim$\,26]{L} or \cite[Appendix~C]{HR} for the general theory of the maximal ideal spaces of commutative Banach algebras.
\section{The universal compact ambit of $T$}\label{s3}
A compact ambit of $T$ is called \textit{universal} iff any other compact ambit of $T$ must be one of its factors (cf.~\cite[Def.~7.4]{E69} or \cite[p.117]{A88}). We shall prove that the universal compact ambit of $T$ may be realized by $T$ acting on $T^\textsl{\texttt{LUC}}$\,(Thm.~\ref{3.3}).

\begin{thm}[{see \cite[Prop.~IV.5.12]{De} for compact flows}]\label{3.1}
	Let $(\mathscr{Y},y_0)$ and $(\mathscr{Z},z_0)$ be two point-transitive compact dynamics of $T$. Then there exists an extension $\phi\colon(\mathscr{Y},y_0)\rightarrow(\mathscr{Z},z_0)$ iff $\mathbb{F}(\mathscr{Y},y_0)\supseteq\mathbb{F}(\mathscr{Z},z_0)$. Hence $(\mathscr{Y},y_0)$ is isomorphic with $(\mathscr{Z},z_0)$ iff $\mathbb{F}(\mathscr{Y},y_0)=\mathbb{F}(\mathscr{Z},z_0)$.
\end{thm}

\begin{proof}
	Necessity is obvious. Now conversely, suppose $\mathbb{F}(\mathscr{Y},y_0)\supseteq\mathbb{F}(\mathscr{Z},z_0)$. To construct an extension $\phi\colon(\mathscr{Y},y_0)\rightarrow(\mathscr{Z},z_0)$, we first define a mapping
	$\phi\colon Ty_0\rightarrow Tz_0$ by $ty_0\mapsto tz_0$.
	To see that $\phi$ is a well-defined mapping, let $t_1,t_2\in T$ with $t_1y_0=t_2y_0$. If $t_1z_0\not=t_2z_0$, we can select $F\in \mathcal{C}(Z)$ such that $F(t_1z_0)\not=F(t_2z_0)$ and set $F^{z_0}(t)=F(tz_0)$ for all $t\in T$. So $F^{z_0}(t_1)\not=F^{z_0}(t_2)$ and $F^{z_0}\in \mathbb{F}(\mathscr{Z},z_0)$. By $\mathbb{F}(\mathscr{Y},y_0)\supseteq\mathbb{F}(\mathscr{Z},z_0)$, there exists a function $\eta\in \mathcal{C}(Y)$ such that $\eta^{y_0}(t)=\eta(ty_0)=F^{z_0}(t)$ for every $t\in T$. Thus, $\eta^{y_0}(t_1)\not=\eta^{y_0}(t_2)$. However, $\eta^{y_0}(t_1)=\eta(t_1y_0)=\eta(t_2y_0)=\eta^{y_0}(t_2)$, a contradiction. Thus, $\phi\colon Ty_0\rightarrow Tz_0$ is a well-defined mapping.
	
	Next, we shall extend $\phi$ from $Ty_0$ to $Y$ continuously. For this, let $y\in Y$. Select two nets $s_n, t_n\in T$ with $s_ny_0\to y$ and $t_ny_0\to y$. Then $s_nz_0\to z$ and $t_nz_0\to z$, for some $z\in Z$, and set $\phi(y)=z$. Indeed, assume $s_i^\prime, t_j^{\prime\prime}\in T$ are two subnets of $\{s_n\}$ and $\{t_n\}$, respectively, such that $s_i^\prime z_0\to z^\prime$ and $t_j^{\prime\prime} z_0\to z^{\prime\prime}$ in $Z$ with $z^\prime\not=z^{\prime\prime}$. We can pick a function $F\in \mathcal{C}(Z)$ such that $F(z^\prime)\not=F(z^{\prime\prime})$. Since $F^{z_0}\in \mathbb{F}(\mathscr{Z},z_0)\subseteq\mathbb{F}(\mathscr{Y},y_0)$, there exists $\eta\in \mathcal{C}(Y)$ with $\eta^{y_0}=F^{z_0}$. So
	\begin{equation*}
		F^{z_0}(s_i^\prime)=\eta^{y_0}(s_i^\prime)=\eta(s_i^\prime y_0)\to\eta(y)
	\end{equation*}
	and
	\begin{equation*}
	F^{z_0}(t_j^{\prime\prime})=\eta^{y_0}(t_j^{\prime\prime})=\eta(t_j^{\prime\prime} y_0)\to\eta(y).
	\end{equation*}
	Thus, $F(z^\prime)=\eta(y)=F(z^{\prime\prime})$, a contradiction. So $\phi\colon Y\rightarrow Z$ is well defined and continuous.

Since $\phi(ty_0)=tz_0=t\phi(y_0)$ and $\overline{Ty_0}=Y$, hence $\phi$ is $T$-equivariant with $z_0=\phi(y_0)$. The proof is completed.
\end{proof}

Therefore, if $\mathscr{X}$ is a minimal non-distal compact semiflow such that $(x_0,x_1)$ is not a.p. under $\mathscr{X}\times\mathscr{X}$, then we have that $\mathbb{F}(\mathscr{X},x_0)\not=\mathbb{F}(\mathscr{X},x_1)$. Otherwise, there exists $\phi\colon(\mathscr{X},x_0)\rightarrow(\mathscr{X},x_1)$ so that $(x_0,x_1)$ is an a.p. point under $\mathscr{X}\times\mathscr{X}$.

\begin{se}[Regular and coalescent dynamics]\label{3.2}
A dynamic $\mathscr{X}$ is called \textit{regular} as in~\cite{A63} if whenever $(x,y)\in X\times X$ is a.p. under $\mathscr{X}\times \mathscr{X}$, then there exists $\rho\in\texttt{Aut}(\mathscr{X})$ such that $y=\rho x$. If every endomorphism of $\mathscr{X}$ is an automorphism of $\mathscr{X}$, then $\mathscr{X}$ is said to be \textit{coalescent} (cf.~\cite[p.81]{A88}).
	
For a non-LC group, the regularity/coalescence of the universal minimal compact flow is proved in the literature \cite[p.116]{A88} and \cite{GL} by using the maximal a.p. set and transfinite induction, respectively. Here Ellis' algebraic proof is still valid for left-topological monoids; see Theorem~\ref{3.3} below.
\end{se}

According to \ref{2.8}, $(\mathscr{T}^\textsl{\texttt{LUC}},e^*)$ is a compact ambit of $T$. It turns out that $\mathscr{T}^\textsl{\texttt{LUC}}$ is universal in the category of compact ambits of $T$.

\begin{thm}[{cf.~\cite{K66}, \cite[$\S\S$7,9]{E69}, \cite[p.35]{G76}, \cite[p.122]{A88} for $T$ a topological group}]\label{3.3}
Let $T$ be a Hausdorff left-topological monoid. Then:
\begin{enumerate}
\item $(\mathscr{T}^\textsl{\texttt{LUC}},e^*)$ is the unique (up to isomorphisms) universal compact ambit of $T$. Moreover, it holds that $\mathbb{F}(\mathscr{T}^\textsl{\texttt{LUC}},e^*)=\verb"LUC"(T)$.
\item $T^\textsl{\texttt{LUC}}$ is a compact Hausdorff right-topological semigroup with identity $e^*$, where the multiplication is the extension of $(s^*,t^*)\mapsto (st)^*$  for all $s,t\in T$. (Hence, if $\phi\colon\mathscr{T}^\textsl{\texttt{LUC}}\rightarrow\mathscr{T}^\textsl{\texttt{LUC}}$ is a homomorphism, then there exists $q\in T^\textsl{\texttt{LUC}}$ with $\phi=R_q$; cf.~\cite[Prop.~7.9]{E69}.)
\item If $T$ is a LC topological group, then $\mathscr{T}^\textsl{\texttt{LUC}}$ is coalescent. (Thus, if $\mathscr{Z}$ is point transitive and $\varphi\colon\mathscr{Z}\rightarrow\mathscr{T}^\textsl{\texttt{LUC}}$ is an extension, then $\varphi$ is an isomorphism; cf.~\cite[Cor.~7.10]{E69}.)
\end{enumerate}
\noindent
Hence the universal minimal compact semiflow of $T$ is:
\begin{enumerate}
		\item[a)] regular (so unique up to isomorphisms);
		\item[b)] a factor of $\mathscr{T}^\textsl{\texttt{LUC}}$; and
		\item[c)] a subsystem of $\mathscr{T}^\textsl{\texttt{LUC}}$.
\end{enumerate}
\end{thm}

\begin{proof}
\item (1): First, $(\mathscr{T}^\textsl{\texttt{LUC}},e^*)$ is a compact ambit of $T$ by \ref{2.8}.1 and \ref{2.8}.4. By \ref{2.9}, it follows that $\mathbb{F}(\mathscr{T}^\textsl{\texttt{LUC}},e^*)=\verb"LUC"(T)$. Suppose $(\mathscr{Y},y_0)$ is any  compact ambit of $T$. Since $\mathbb{F}(\mathscr{Y},y_0)\subseteq\verb"LUC"(T)$ by \ref{2.4}.1, hence $(\mathscr{Y},y_0)$ is a factor of $(\mathscr{T}^\textsl{\texttt{LUC}},e^*)$ from Theorem~\ref{3.1}. If $(\mathscr{Y},y_0)$ is another universal compact ambit of $T$, then there exists a homomorphism $(\mathscr{Y},y_0)\rightarrow(\mathscr{T}^\textsl{\texttt{LUC}},e^*)$. Whence
	$\verb"LUC"(T)\supseteq\mathbb{F}(\mathscr{Y},y_0)\supseteq\mathbb{F}(\mathscr{T}^\textsl{\texttt{LUC}},e^*)=\verb"LUC"(T)$. Thus, it follows by Theorem~\ref{3.1} that $(\mathscr{Y},y_0)\cong(\mathscr{T}^\textsl{\texttt{LUC}},e^*)$.

\item (2): It follows from 1. that $(\mathscr{T}^\textsl{\texttt{LUC}},e^*)$ is isomorphic to its Ellis enveloping pointed semiflow $(T, E(\mathscr{T}^\textsl{\texttt{LUC}}),\textrm{id}_{T^\textsl{\texttt{LUC}}})$ such that the homeomorphism $E(\mathscr{T}^\textsl{\texttt{LUC}})\xrightarrow{p\mapsto pe^*}T^\textsl{\texttt{LUC}}$ induces the desired multiplication on $T^\textsl{\texttt{LUC}}$.

\item (3): By Lemma~\ref{2.9}B and Lemma~\ref{2.8}.6.
	
	Finally, let $(T,M)$ be a minimal left ideal in $T^\textsl{\texttt{LUC}}$. By 1., $(T,M)$ is a universal minimal compact semiflow of $T$, which is a factor and a subsystem of $\mathscr{T}^\textsl{\texttt{LUC}}$. Moreover,  $M=\bigcup_{v}vM$ and $vM$ are groups with identity $v$ for all idempotents $v\in M$. So, for every $q\in M$, the right translation $R_q\colon M\rightarrow M$ is an automorphism of $(T,M)$. Then $(T,M)$ is regular and it is exactly the unique (up to isomorphisms) universal minimal compact semiflow of $T$.
	
	The proof is thus completed.
\end{proof}

The simple observation Theorem~\ref{3.3}.2 has been useful in proving the regularity/unicity of the universal minimal compact system of $T$.

\begin{rem}
	$\iota\colon T\xrightarrow{t\mapsto t^*} T^\textsl{\texttt{LUC}}$ is left-uniformly continuous in the sense that as $t_n\to t$ in $T$, $\iota t_n=\iota\circ L_{t_n}\to\iota t=\iota\circ L_t$ uniformly; that is, given any $f_1,\dotsc,f_k\in\verb"LUC"(T)$ and $\varepsilon>0$, $\exists\, n_0$ s.t. if $n\ge n_0$ then
	$$
	|f_i(t_nx)-f_i(tx)|<\varepsilon\quad \forall x\in T
	$$
	for $i=1,\dotsc,k$.
\end{rem}

\begin{rem}
	If $T$ is Tychonoff (completely regular $T_1$), then the Stone-\v{C}ech compactification $(\mathfrak{e},\beta T)$ is well defined (\cite[Thm.~5.24]{K}). So every continuous function from $T$ to a compact Hausdorff space may be uniquely continuously extended on $\beta T$. However, $(\iota,T^\textsl{\texttt{LUC}})$ can only extend left-uniformly continuous functions on $T$ (see \ref{2.8}.4-ii). This shows that $T^\textsl{\texttt{LUC}}$ is ``smaller'' than $\beta T$ in general; that is, there is a continuous map $\psi\colon(\beta T,e)\rightarrow(T^\textsl{\texttt{LUC}},e^*)$ with $\iota=\psi\circ\mathfrak{e}$.
\end{rem}

\begin{rem}\label{3.6}
	If $T$ is a LC topological group or if $T$ is a discrete monoid, then $(\beta T,e)\cong(T^\textsl{\texttt{LUC}},e^*)$.
\end{rem}

\begin{proof}
	It follows by Ellis's joint continuity theorem that the canonical multiplication $T\times\beta T\rightarrow\beta T$ is jointly continuous so that $(T,\beta T,e)$ is a universal compact ambit. Then by Theorem~\ref{3.3}, $(T,\beta T,e)\cong(T,T^\textsl{\texttt{LUC}},e^*)$. The proof is complete.
\end{proof}

\begin{rem}\label{3.7}
	By Theorem~\ref{3.3} or \ref{2.8}.1, $\rho\colon T\times T^\textsl{\texttt{LUC}}\rightarrow T^\textsl{\texttt{LUC}}$ is jointly continuous. However, $\lambda\colon T^\textsl{\texttt{LUC}}\times T\rightarrow T^\textsl{\texttt{LUC}}$ need not be jointly continuous here. To see this, let's borrow an example from \cite{R79} (also see \cite[Prop.~1]{BF}) as follows:	
	Let $G$ be a LC Hausdorff topological group, which is not distal in the sense of Rosenblatt. So there is some $g\in G$ with $g\not=e$ and a net $g_n\in G$ such that $g_n^{-1}gg_n\to e$ in $G$. Passing to a subnet if necessary, we can assume $g_n\to x$ in $G^\textsl{\texttt{LUC}}$. If $\lambda$ is jointly continuous, then
	$$
	gg_n\to gx\quad \textrm{and}\quad gg_n=g_n(g_n^{-1}gg_n)\to x,
	$$
	contrary to $gx\not=x$ by Veech's Theorem (see Thm.~\ref{5.7} below).
\end{rem}

\begin{se}[Left-invariant means on $\texttt{LUC}(T)$ and invariant measures on $T^\textsl{\texttt{LUC}}$]
A linear functional $\mathfrak{m}\colon\verb"LUC"(T)\rightarrow\mathbb{C}$ is called a \textit{left-invariant mean} if
	\begin{gather*}
		\overline{\mathfrak{m}(f)}=\mathfrak{m}(\bar{f})\quad \forall f\in\texttt{LUC}(T),\\
		\mathfrak{m}(\mathbf{1})=1,\\
		\mathfrak{m}(f)\ge0\quad \textrm{if }f\in\texttt{LUC}(T)\textrm{ s.t. }f\ge0,\\
		t\mathfrak{m}=\mathfrak{m}\ (\textrm{or equivalently}, \mathfrak{m}(ft)=\mathfrak{m}(f)\ \forall f\in\texttt{LUC}(T))\quad \forall t\in T.
	\end{gather*}
	$T$ is called \textit{left-amenable} iff there exists a left-invariant mean on $\verb"LUC"(T)$.
	
	Then:
	
\begin{thm*}[{see \cite[Thm.~III.3.1, (1)$\Leftrightarrow$(2)]{G76} for $T$ a topological group}]
$\mathscr{T}^\textsl{\texttt{LUC}}$ admits a regular Borel probability measure iff $\verb"LUC"(T)$ have a left-invariant mean. In other words, $\mathscr{T}^\textsl{\texttt{LUC}}$ admits regular Borel probability measures iff $T$ is left-amenable.
\end{thm*}
	
\begin{proof}
First from \ref{2.8}.3d the Fourier transformation $\hat{}\colon\verb"LUC"(T)\rightarrow\mathcal{C}(T^\textsl{\texttt{LUC}})$ is a topological isomorphism under $\|\cdot\|_\infty$ such that
		$$
		\widehat{ft}=\hat{f}t\quad \forall f\in\verb"LUC"(T)\textrm{ and }t\in T.
		$$
		If $\mathfrak{m}$ is a left-invariant mean on $\verb"LUC"(T)$, then define a linear functional as follows:
		$$
		\mu\colon\mathcal{C}(T^\textsl{\texttt{LUC}})\rightarrow\mathbb{C}\quad \textrm{by}\ \hat{f}\mapsto\mathfrak{m}(f).
		$$
By the Riesz representation theorem, $\mu$ may be a regular Borel probability measure on $T^\textsl{\texttt{LUC}}$ such that $\mu(\varphi)=\mu(\varphi t)$ for all $\varphi\in\mathcal{C}(T^\textsl{\texttt{LUC}})$ and $t\in T$. Thus, $\mu$ is $T$-invariant on $T^\textsl{\texttt{LUC}}$ and $\mathscr{T}^\textsl{\texttt{LUC}}$ admits $\mu$.

Conversely, suppose $\mu$ is an invariant regular Borel probability measure on $T^\textsl{\texttt{LUC}}$. Define $\mathfrak{m}\colon\verb"LUC"(T)\rightarrow\mathbb{C}$ by $f\mapsto\mathfrak{m}(f)=\mu(\hat{f})$. Clearly, $\mathfrak{m}$ is a left-invariant mean on $\verb"LUC"(T)$. The proof is complete.
	\end{proof}
\end{se}
\section{Hulls of left-uniformly continuous functions}\label{s4}
Given $f\in\verb"LUC"(T)$, we have seen that $f$ may come from a compact ambit of $T$ (Thm.~\ref{3.3}). In fact, $(\mathscr{T}^\textsl{\texttt{LUC}},e^*)$ is the maximal such ambit of $T$. Now we will find the ``smallest'' one\,(see Thm.~\ref{4.4}).

\begin{se}[Left compact-open/pointwise hulls of functions]\label{4.1}
	Given $f\in\mathcal{C}_b(T)$, define
	\begin{enumerate}
		\item[(1)] $\verb"H"_{\textrm{co}}^L[f]={\overline{Tf}}^{\textrm{co}}$ and
		
		\item[(2)] $\verb"H"_{\textrm{p}}^L[f]={\overline{Tf}}^{\textrm{p}}$.
	\end{enumerate}
	Here ${}_\textrm{co}$ and ${}_\textrm{p}$ stand for the compact-open and pointwise topologies on $\verb"LUC"(T)$, respectively.
	\begin{enumerate}
		\item[(3)] Let $\mathcal{F}\subseteq\mathcal{C}_b(T)$. We say $\mathcal{F}$ is \textit{left-equicontinuous} on $T$ if given $\varepsilon>0$ and $s_n\to s$ in $T$, there exists $n_0$ such that as $n\ge n_0$,
		$ \|fs_n-fs\|_\infty<\varepsilon$ for all $f\in\mathcal{F}$. In this case, $\mathcal{F}\subseteq\verb"LUC"(T)$ by Def.~\ref{2.2}.3.
	\end{enumerate}
	Then it is easy to verify that
	\begin{enumerate}
		\item[(4)] If $\mathcal{F}\subseteq\mathcal{C}_b(T)$ is left-equicontinuous on $T$, then $\mathcal{F}$ is equicontinuous on $T$ (i.e., $\forall x\in T$ and $\forall \varepsilon>0$, $\exists U\in\mathfrak{N}_x(T)$ s.t. $|f(y)-f(x)|<\varepsilon$ $\forall y\in U$ and $\forall f\in\mathcal{F}$).
	\end{enumerate}
	
	Indeed, let $x\in T$. If $\mathcal{F}$ is not equicontinuous at $x$, then there is some $\varepsilon>0$ such that for every $U\in\mathfrak{N}_x(T)$ we can pick some $f\in\mathcal{F}$ and $u\in U$ with $|f(u)-f(x)|\ge\varepsilon$. Further, there is a net $u_n\to x$ in $T$ and $f_n\in\mathcal{F}$ such that $|f_nu_n(e)-f_nx(e)|\ge\varepsilon$, contrary to (3).
\end{se}

\begin{lem}[{cf.~\cite[Lem.~IV.5.2, Cor.~IV.5.3]{De} for $T$ a topological group}]\label{4.2}
	Let $f\in\mathcal{C}_b(T)$; then the following conditions are pairwise equivalent:
	\begin{enumerate}
		\item[(1)] $f\in\verb"LUC"(T)$;
		\item[(2)] $Tf$ is left-equicontinuous on $T$;
		\item[(3)] $\verb"H"_{\textrm{co}}^L[f]$ is left-equicontinuous on $T$;
		\item[(4)] $\verb"H"_{\textrm{co}}^L[f]\subseteq\verb"LUC"(T)$.
	\end{enumerate}
	If one of (1) $\sim$ (4) holds, then $\verb"H"_{\textrm{co}}^L[f]=\verb"H"_{\textrm{p}}^L[f]$ is a compact left-invariant subset of $\verb"LUC"(T)$ and
	\begin{enumerate}
    \item[] $\rho\colon T\times\verb"H"_{\textrm{co}}^L[f]\rightarrow\verb"H"_{\textrm{co}}^L[f],\quad (t,\phi)\mapsto t\phi$
    \end{enumerate}
	is a jointly continuous phase mapping, under the compact-open topology on $\verb"H"_{\textrm{co}}^L[f]$. Consequently, $(\mathscr{H}_{\textrm{co}}^L[f],f):=(T,\verb"H"_{\textrm{co}}^L[f],f)$ is a compact ambit and $f$ comes from $(\mathscr{H}_{\textrm{co}}^L[f],f)$, for all $f\in\verb"LUC"(T)$.
\end{lem}

\begin{proof}
	\item (1)$\Rightarrow$(2): Let $\varepsilon>0$ and $s_n\to s$ in $T$ be arbitrary. By
	$$
	|tfs_n(x)-tfs(x)|=|fs_n(xt)-fs(xt)|<\varepsilon\quad \forall x\in T\textrm{ and }t\in T
	$$
	as $n$ sufficiently big, it follows that $Tf$ is left-equicontinuous on $T$.
	
	\item (2)$\Rightarrow$(3): Let $\varepsilon>0$ and $s_n\to s$ in $T$. By condition (2) and \ref{4.1}-(3), there is $n_0$ such that as $n\ge n_0$,
	$$
	|tfs_n(x)-tfs(x)|=|fs_n(xt)-fs(xt)|<\varepsilon\quad \forall x\in T\textrm{ and }t\in T.
	$$
	Now let $\phi\in\verb"H"_{\textrm{co}}^L[f]$ and select a net $t_i\in T$ with $t_if\to_{\textrm{co}}\phi$. Then, as $n\ge n_0$,
	$$
	|\phi s_n(x)-\phi s(x)|={\lim}_i|t_ifs_n(x)-t_ifs(x)|\le\varepsilon\quad \forall x\in X.
	$$
	Thus, $\verb"H"_{\textrm{co}}^L[f]$ is left-equicontinuous on $T$.
	
	\item (3)$\Rightarrow$(4): By \ref{4.1}-(3).
	
	\item (4)$\Rightarrow$(1): Obvious.
	
	Now suppose one of (1) $\sim$ (4) is fulfilled. By \ref{2.2}.4 and condition (3), $\verb"H"_{\textrm{co}}^L[f]$ is a left-invariant subset of $\verb"LUC"(T)$. Moreover, it follows from \ref{4.1}-(4) and Arzel\`{a}-Ascoli theorem (cf.~\cite[A.2.5]{De}) that $\verb"H"_{\textrm{co}}^L[f]$ is compact under the compact-open topology on $\verb"LUC"(T)$. Thus, $\verb"H"_{\textrm{co}}^L[f]=\verb"H"_{\textrm{p}}^L[f]$ and the compact-open topology coincides with the pointwise topology on it (cf.~\cite[3A]{L}).
	
	It remains to show that $\rho\colon(t,\phi)\rightarrow t\phi$ is continuous from $T\times\verb"H"_{\textrm{co}}^L[f]$ to $\verb"H"_{\textrm{co}}^L[f]$. For this, let $t_n\to t$ in $T$ and $f_n\to\phi$ in $\verb"H"_{\textrm{co}}^L[f]$; we need prove that $t_nf_n\to_\textrm{p}t\phi$. For this, let $x\in T$ then by $xt_n\to xt$ in $T$ (noting $T$ is a left-topological semigroup), condition (3) and
	\begin{equation*}\begin{split}
			|t_nf_n(x)-t\phi(x)|&=|f_n(xt_n)-\phi(xt)|\\
			&\le|f_n(xt_n)-f_n(xt)|+|f_n(xt)-\phi(xt)|\\
			&=|f_nxt_n(e)-f_nxt(e)|+|f_n(xt)-\phi(xt)|,
	\end{split}\end{equation*}
	it follows that $\lim_n|t_nf_n(x)-t\phi(x)|=0$. Thus, $\rho$ is continuous.
	
	Let $\mathfrak{e}\colon\verb"H"_{\textrm{co}}^L[f]\rightarrow\mathbb{C}$ be defined by $\phi\mapsto\phi(e)$. Then $\mathfrak{e}\in\mathcal{C}(\verb"H"_{\textrm{co}}^L[f])$ and $f(t)=\mathfrak{e}_f(t)$ for all $t\in T$. So $f\in\mathbb{F}(\mathscr{H}_{\textrm{co}}^L[f],f)$ for all
	$f\in\verb"LUC"(T)$.
	The proof is completed.
\end{proof}

Next we proceed to show that $(\mathscr{H}_{\textrm{co}}^L[f],f)$ is actually the smallest compact ambit generated by $f\in\verb"LUC"(T)$.

\begin{thm}[{cf.~\cite[Prop.~9.8]{E69} and \cite[Thm.~IV.5.8]{De} for $T$ a group}]\label{4.3}
	If $(\mathscr{X},x_0)$ is a compact ambit of $T$, then $(\mathscr{X},x_0)\cong(T, |\mathbb{F}(\mathscr{X},x_0)|,e^*)$.
\end{thm}

\begin{proof}
	By \ref{2.4}.2, \ref{2.9} and Theorem~\ref{3.1}.
\end{proof}

We note here that our proof line of Theorem~\ref{4.3} is different with that available in literature (e.g. \cite{De}). Here Theorem~\ref{4.3} follows from Theorem~\ref{3.1}. In fact, Theorem~\ref{3.1} may in turn follow from Theorem~\ref{4.3} as to be dealt with in \cite[p.348]{De}

\begin{thm}\label{4.4}
	Let $f\in\mathcal{C}_b(T)$ then:
	\begin{enumerate}
		\item[(1)] $f\in\verb"LUC"(T)$ iff $f$ comes from a compact ambit of $T$.
		\item[(2)] If $f\in\verb"LUC"(T)$, then $\mathbb{F}(\mathscr{H}_{\textrm{co}}^L[f],f)$ is the smallest norm closed right-invariant self-adjoint subalgebra of $\verb"LUC"(T)$ containing $f$ and $\mathbf{1}$.
	\end{enumerate}
\end{thm}

\begin{proof}
\item (1): By Lemma~\ref{4.2} and \ref{2.4}.1.
	
\item (2):	By Lemma~\ref{4.2} and \ref{2.4}.2, $(\mathscr{H}_{\textrm{co}}^L[f],f)$ is a compact ambit such that $\mathbb{F}(\mathscr{H}_{\textrm{co}}^L[f],f)$ is a norm closed right-invariant self-adjoint subalgebra of $\verb"LUC"(T)$ containing $f$ and $\mathbf{1}$. Let $\mathbb{V}$ be the smallest closed right-invariant self-adjoint subalgebra of $\verb"LUC"(T)$ containing $f$ and $\mathbf{1}$. Then $\mathbb{V}\subseteq\mathbb{F}(\mathscr{H}_{\textrm{co}}^L[f],f)$. By \ref{2.9}, $(\mathscr{V},e^*)$ is a compact ambit such that $\mathbb{F}(\mathscr{V},e^*)=\mathbb{V}$
	and $\mathcal{C}(|\mathbb{V}|)=\{\hat{g}\,|\,g\in\mathbb{V}\}$. Following \ref{2.4}.2, we define
	$$
	\varrho_{\hat{f}}^{}\colon|\mathbb{V}|\rightarrow\verb"LUC"(T),\quad \mu\mapsto\hat{f}^\mu\quad \textrm{s.t. }te^*\mapsto tf\ \forall t\in T,
	$$
	where $\hat{f}^\mu(t)=\hat{f}(t\mu)=\langle\mu,ft\rangle\ \forall t\in T$, for all $\mu\in|\mathbb{V}|$. Since $e^*$ is a transitive point for $\mathscr{V}$, it follows by Lemma~\ref{4.2} that $\varrho_{\hat{f}}^{}(e^*)=f$ and $\varrho_{\hat{f}}^{}[|\mathbb{V}|]=\verb"H"_{\textrm{co}}^L[f]$. Moreover, it is easy to show that $\varrho_{\hat{f}}^{}\colon(|\mathbb{V}|,e^*)\rightarrow(\verb"H"_{\textrm{co}}^L[f],f)$ is a homomorphism. So by Theorem~\ref{3.1}, $\mathbb{V}\supseteq\mathbb{F}(\mathscr{H}_{\textrm{co}}^L[f],f)$ and $\mathbb{V}=\mathbb{F}(\mathscr{H}_{\textrm{co}}^L[f],f)$.
The proof is completed.
\end{proof}

\begin{thm}[{cf.~\cite[Cor.~IV.5.20]{De} for $T$ a topological group}]\label{4.5}
Let $T$ be a Hausdorff left-topological monoid. Then
$$
T\curvearrowright(T^\texttt{LUC},e^*)\cong T\curvearrowright(X,x_0),
$$
where
$$
X=\overline{Tx_0}\quad \textrm{and}\quad x_0=(f)_{f\in\texttt{LUC}(T)}\in{\prod}_{f\in\texttt{LUC}(T)}\verb"H"_p^L[f].
$$
\end{thm}
\begin{proof}
Using Theorems~\ref{3.1} and \ref{3.3} and the fact $\mathbb{F}(\mathscr{X},x_0)=\texttt{LUC}(T)$.
\end{proof}
\section{On locally quasi-totally bounded groups and Veech's Theorem}\label{s5}
If $T$ is a LC Hausdorff topological group, then $\mathscr{T}^\textsl{\texttt{LUC}}$ is free by Veech \cite[Thm.~2.2.1]{V77} (the discrete case of which is due to Ellis~\cite{E60}, also see \cite[Prop.~8.14]{E69}). In fact, we shall extend Veech's theorem to locally quasi-totally bounded topological groups\,(Def.~\ref{5.4}.4 and Thm.~\ref{5.7}).

\begin{se}[Effective and strongly effective]\label{5.1}
	A dynamic $\mathscr{X}=(T,X)$ is said to be \textit{effective} if we have for all $t\in T$ that $tx=x\ \forall x\in X$ implies $t=e$. If, for all $t\in T$, condition $tx=x$ for some $x\in X$ implies $t=e$, then $\mathscr{X}$ is called \textit{strongly effective} or \textit{free}.
\end{se}

\begin{se}\label{5.2}
	Let $t\in T$ with $t\not=e$ and let $\alpha\in T^\textsl{\texttt{LUC}}$. If $t_n\in T$ with $t_n^*\to\alpha$ in $T^\textsl{\texttt{LUC}}$, then by \ref{2.8}.3c and \ref{2.9},
	$t\alpha=\alpha$ iff $\lim_nf(tt_n)=\hat{f}(t\alpha)=\hat{f}(\alpha)=\lim_nf(t_n)$ for all $f\in\verb"LUC"(T)$.
\end{se}

\begin{thm}\label{5.3}
	If $\verb"LUC"(T)$ separates the points and closed sets in $T$ and $T$ is right cancelable (for instance, $T$ is a Hausdorff topological group), then $\mathscr{T}^\textsl{\texttt{LUC}}$ is effective and $\iota\colon T\rightarrow T^\textsl{\texttt{LUC}}$, $t\mapsto t^*$ is an embedding. 
\end{thm}

\begin{proof}
	By \ref{5.2} and \ref{2.8}.6.
\end{proof}

\begin{se}[Locally quasi-totally bounded group]\label{5.4}
Now we will consider a class of topological groups---locally quasi-totally bounded topological groups that are more general than LC topological groups.
	
\begin{5.4.1}
In the sequel of this section, let $G$ be a Hausdorff topological group with identity $e$.
\end{5.4.1}

\begin{5.4.2}
$G$ is called \textit{locally totally bounded} if there exists a set $U\in\mathfrak{N}_e(G)$ which is totally bounded in the sense that given any $V\in\mathfrak{N}_e(G)$ one can find a finite subset $F$ of $U$ such that $VF\supseteq U$; that is, $U$ can be covered by a finite number of right translates of $V$.
\end{5.4.2}
	
\begin{5.4.3}
A LC Hausdorff topological group is obviously locally totally bounded; however, not vice versa.
\begin{enumerate}
\item[a.] $(\mathbb{Q},+)$, as a subspace of the Euclidean space $\mathbb{R}$, is locally totally bounded but not LC.
\item[b.] $\texttt{SL}(2,\mathbb{Q})$, with the topology inherited from the Euclidean space $\mathbb{R}^{2\times 2}$, is locally totally bounded non-LC.
\end{enumerate}
\end{5.4.3}

\begin{5.4.4}
	$G$ is called \textit{locally quasi-totally bounded} if for every $t\in G$ with $t\not=e$ there exists a symmetric set $V\in\mathfrak{N}_e(G)$ such that $tV^3\cap V^3=\emptyset$ and that
	$V^3tV^3t^{-1}\cup V^3t^{-1}V^3t$ can be covered by a finite number of right translates of $V$.
\end{5.4.4}

A locally totally bounded group must be locally quasi-totally bounded; but not vice versa as shown by the following examples.
\begin{5.4.5}
	\item \begin{enumerate}
		\item First we consider the group $(\mathbb{Z}^\mathbb{Z},+)$. Given any integers $n\ge1$ and $n$ elements $x_1,\dotsc,x_n\in\mathbb{Z}$, let $V[x_1,\dotsc,x_n]$ be the subgroup of $\mathbb{Z}^\mathbb{Z}$ consisting of the elements $f\in\mathbb{Z}^\mathbb{Z}$ such that $f(x_i)=0$ for all $1\le i\le n$. It is easy to see that there exists a Hausdorff topology $\mathfrak{T}$ on $\mathbb{Z}^\mathbb{Z}$ for which $(\mathbb{Z}^\mathbb{Z},\mathfrak{T})$ is a topological group with a neighborhood base
		\begin{gather*}
			\{V[x_1,\dotsc,x_n]\,|\,x_1,\dotsc,x_n\in\mathbb{Z}\textrm{ and }n\ge1\}
		\end{gather*}
		of the zero element $\textbf{o}=(\dotsc,0,0,0,\dotsc)$ in $\mathbb{Z}^\mathbb{Z}$.
		\begin{enumerate}
			\item[$\bullet$] $(\mathbb{Z}^\mathbb{Z},\mathfrak{T})$ is an abelian locally quasi-totally bounded, but not locally totally bounded (further non-LC), Hausdorff topological group.
		\end{enumerate}
		\begin{proof}
			It is obvious that $(\mathbb{Z}^\mathbb{Z},\mathfrak{T})$ is locally quasi-totally bounded for $V[x_1,\dotsc,x_n]$ is a subgroup of $\mathbb{Z}^\mathbb{Z}$. Now we show that it is not locally totally bounded; otherwise, there exists a totally bounded neighborhood $V[x_1,\dotsc,x_n]$ of $\textbf{o}$. Let $y$ be an integer with $y\not\in\{x_1,\dotsc,x_n\}$. Then there exist elements $f_1,\dotsc,f_r\in\mathbb{Z}^\mathbb{Z}$ such that
			$$
			V[x_1,\dotsc,x_n]\subseteq (f_1+V[y])\cup\dotsm\cup(f_r+V[y]).
			$$
			So, for every $\varphi\in V[x_1,\dotsc,x_n]$ we have $\varphi(y)\in\{f_1(y),\dotsc,f_r(y)\}$. Thus, $\mathbb{Z}\subseteq\{f_1(y),\dotsc,f_r(y)\}$. This is a contradiction.
		\end{proof}
		
		\item Next we will construct a non-abelian locally quasi-totally bounded Hausdorff topological group, which is not locally totally bounded as follows.
		Let $\{H_n\}_{n\in\mathbb{N}}$ be a family of Hilbert spaces and $H$ the direct sum of them.
		We consider the non-abelian multiplicative group
		$$G=\left\{L\in B(H)\,|\, L\textrm{ is invertible s.t. }LH_n=H_n\ \forall n\in\mathbb{N}\right\}$$
		where $B(H)$ be the set of bounded operators of $H$ into itself. For every finite subset $\gamma$ of $\mathbb{N}$, define a normal subgroup $N_\gamma$
		of $G$ by
		$L\in N_\gamma$ iff $L|_{H_n}=\textrm{id}_{H_n}\ \forall n\in\gamma$.
		Let $\Gamma$ be the collection of all finite nonempty subsets of $\mathbb{N}$.
		
		Note that $N_A\cap N_B=N_{A\cup B}$ for $A,B\in\Gamma$ and $\bigcap_{\gamma\in\Gamma}N_\gamma=\{\textrm{id}_H\}$. It is easy to see
		that there exists a Hausdorff topology $\mathfrak{T}$ such that $(G, \mathfrak{T})$ becomes a topological
		group with the neighborhood base $\{N_\gamma\,|\,\gamma\in\Gamma\}$ of $e$ where $e=\textrm{id}_H$.
		
		\begin{enumerate}
			\item[$\bullet$] $(G,\mathfrak{T})$ is locally quasi-totally bounded, but
			not locally totally bounded,  non-abelian Hausdorff topological group.
		\end{enumerate}
		
		\begin{proof}
			It is obvious that $G$ is locally quasi-totally bounded. Now we show
			that $G$ is not locally totally bounded. Otherwise, there exists a totally bounded
			neighborhood $N_\gamma$ of $e$. Let $i$ be an integer with $i\not\in\gamma$. Then there exist $L_1,\dotsc, L_r\in
			G$ such that $N_\gamma\subseteq L_1N_{\{i\}}\cup\dotsm\cup L_rN_{\{i\}}$. So for all $L\in N_\gamma$ and $x\in H_i$ we have
			$L(x)\in\{L_1(x), \dotsc, L_r(x)\}$. It is a contradiction (if we pick carefully $H_i$, e.g., $H_i=\mathbb{R}$).
		\end{proof}
	\end{enumerate}
\end{5.4.5}

\begin{lem}[{Urysohn-wise lemma; cf.~Kat\v{e}tov 1959~\cite{Ka}}]\label{5.5}
	Let $(X,\mathscr{U})$ be any uniform space and $A,B\subset X$ nonempty. If there is an $\alpha\in\mathscr{U}$ such that $A\cap\alpha[B]=\emptyset$, then there is a uniformly continuous function $f\colon X\rightarrow [0,1]$ such that $f=1$ on $A$ and $f=0$ on $B$.
\end{lem}

\begin{proof}
First we can take some $\gamma\in\mathscr{U}$ such that $\overline{A}\cap\gamma[\overline{B}]=\emptyset$. Further there is some $\delta\in\mathscr{U}$ such that $\delta[\overline{A}]\cap\delta[\overline{B}]=\emptyset$. Now by \cite[Thm~6.15]{K}, it follows that there is a pseudo-metric $\rho$, which is uniformly continuous on $X\times X$, such that $\{(x,y)\in X\times X\,|\,\rho(x,y)<r\}\subseteq\delta$ for some $r>0$. Thus, relative to $\rho$, the ${r}/{3}$-neighborhoods
$\mathrm{B}_{r/3}(\overline{A})$ and $\mathrm{B}_{r/3}(\overline{B})$ are such that $\mathrm{B}_{r/3}(\overline{A})\cap\mathrm{B}_{r/3}(\overline{B})=\emptyset$.
Since $f_B\colon x\mapsto3r^{-1}\rho(\overline{B},x)$ from $X$ to $[0,\infty)$ is uniformly continuous by the triangle inequality, hence $f\colon X\rightarrow [0,1]$ defined by $x\mapsto\min\{1,f_B(x)\}$ is uniformly continuous. Since $f|B\equiv0$ and $f|A\equiv1$, the proof is completed.
\end{proof}

\begin{5.5.1}[{cf.~\cite[Thm.~3E]{L} using Urysohn's lemma and one-point compactification}]
	If $X$ is a LC Hausdorff space and if $C$ and $U$ are respectively compact and open sets such that $C\subset U$, then there exists a real-valued continuous function $f$ on $X$ such that $f=0$ on $C$, $f=1$ on $X\setminus U$ and $0\le f\le 1$.
\end{5.5.1}

\begin{5.5.2}
Let $T$ have a left-uniformity\,(cf.~\ref{2.2}.3-ii). Then $\verb"LUC"(T)$ separates points and closed sets in $T$, and in particular, if $t\in T$ and $t\not=e$ there exists a function $f\in\verb"LUC"(T)$ such that $tf\not=f$.
\end{5.5.2}

\begin{proof}
Let $x\in T$ and $A\subset T$ be a closed set with $x\notin A$. Then there exists an index $\alpha\in\mathscr{U}^L$ such that $\alpha[x]\cap A=\emptyset$. By Lemma~\ref{5.5}, there is a $\mathscr{U}^L$-uniformly continuous function $f\colon X\rightarrow[0,1]$ with $f(x)=0$ and $f=1$ on $A$. Further, $f\in\verb"LUC"(T)$ by \ref{2.2}.3-ii.

Now for $t\in T$ with $t\not=e$, there exists some $f\in\verb"LUC"(T)$ such that $f(e)\not=f(t)=tf(e)$. So $tf\not=f$. The proof is complete.
\end{proof}

\begin{5.5.3}
The universal point-transitive compact flow of $G$ is effective.
\end{5.5.3}

\begin{proof}
By Theorem~\ref{5.3} or by Theorem~\ref{4.5} and Corollary~\ref{5.5}.2.
\end{proof}

\begin{lem}[uniformly separated finite partition lemma]\label{5.6}
Let $t\in G$ with $t\not=e$ such that there exists a symmetric set $V\in\mathfrak{N}_e(G)$ with the properties:
\begin{enumerate}
	\item[$(\mathrm{a})$] $tV^3\cap V^3=\emptyset$;
	\item[$(\mathrm{b})$] $V^3tV^3t^{-1}\cup V^3t^{-1}V^3t$ can be covered by a finite number of right translates of $V$.
\end{enumerate}
Then there exists a finite partition $G=P_1\cup\dotsm\cup P_n$ such that $tP_k\cap VP_k=\emptyset$ for all $k=1,\dotsc,n$.
\end{lem}

\begin{proof}
By Zorn's Lemma, there is a maximal subset $Y$, with respect the inclusion, of $G$ such that $Vy\cap Vy^\prime=\emptyset$ for all $y\not=y^\prime$ in $Y$. Thus
	\begin{equation*}
		|Vx\cap Y|\le 1\quad\forall x\in G,\quad \textrm{where }|N|\textrm{ is the cardinality of }N.\leqno{(\ref{5.6}.1)}
	\end{equation*}
	Let $U=V^3tV^3t^{-1}\cup V^3t^{-1}V^3t$.
	By hypothesis (b) of the lemma, there are a finite number of elements $x_1, \dotsc,x_m$ of $G$ such that $U\subseteq Vx_1\cup\dotsm\cup Vx_m$ so that $Ux\subseteq Vx_1x\cup\dotsm\cup Vx_mx$ for all $x\in G$. Thus, by (\ref{5.6}.1), there exists a positive integer $\mathfrak{n}$ such that
	\begin{equation*}
		|Ux\cap Y|\le\mathfrak{n}\quad \forall x\in G.\leqno{(\ref{5.6}.2)}
	\end{equation*}
	Then put $K=\{k\in\mathbb{N}\,|\,1\le k\le 2\mathfrak{n}+1\}$ and we can define the disjoint subsets $Y_k$, $k\in K$, of $Y$ inductively as follows:
	
	\item \textit{Step~1.} Let $Y_1$ be the maximal subset of $Y$ such that $tV^3Y_1\cap V^3Y_1=\emptyset$. (Since $tV^3x\cap V^3x=\emptyset$ for all $x\in G$ by the choice of $V$, we can always pick such $Y_1$.)
	\item \textit{Step~2.} If the sets $Y_1, \dotsc, Y_k$, for some $k\in K$ with $k<2\mathfrak{n}+1$, have been chosen, then let $Y_{k+1}$ be the maximal subset of $Y\setminus(\bigcup_{j=1}^kY_j)$ such that $tV^3Y_{k+1}\cap V^3Y_{k+1}=\emptyset$; note that we might have $Y_k=\emptyset$ for some $k\in K$ if and only if $Y=\bigcup_{j=1}^{k-1}Y_j$.
	
	We now can assert that $Y=\bigcup_{k\in K}Y_k$. Suppose to the contrary that there exists an element $y\in Y\setminus(\bigcup_{k\in K}Y_k)$. Then $Y_k\not=\emptyset$ for all $k\in K$. Let $k\in K$ be arbitrary. Since $y\in Y\setminus Y_k$, it follows that
	$$
	tV^3(Y_k\cup\{y\})\cap V^3(Y_k\cup\{y\})\not=\emptyset.
	$$
	Then the choice of the sets $V$ and $Y_k$ implies that $tV^3Y_k\cap V^3y\not=\emptyset$ or $tV^3y\cap V^3Y_k\not=\emptyset$ or both. Put
	$$
	K_1=\left\{k\in K\,|\,tV^3Y_k\cap V^3y\not=\emptyset\right\}\ \textrm{ and }\ K_2=\left\{k\in K\,|\,tV^3y\cap V^3Y_k\not=\emptyset\right\}.
	$$
	Then $K=K_1\cup K_2$, and so either $|K_1|\ge \mathfrak{n}+1$ or $|K_2|\ge\mathfrak{n}+1$. Assume first $|K_1|\ge \mathfrak{n}+1$. If $k\in K_1$, then $Y_k\cap V^3t^{-1}V^3y\not=\emptyset$ so that $Y_k\cap Ut^{-1}y\not=\emptyset$. Since the sets $Y_k$ are disjoint, the set $Ut^{-1}y$ contains at least $\mathfrak{n}+1$ members of $Y$, a contradiction to (\ref{5.6}.2). Similarly, $|K_2|\ge\mathfrak{n}+1$ will lead to a contradiction to (\ref{5.6}.2). Thus have concluded our assertion.
	
	$G=V^2Y$ because of the maximality of $Y$. Therefore $G=\bigcup_{k\in K}V^2Y_k$ such that for all $k\in K$, $tV^2Y_k\cap V(V^2Y_k)\subseteq tV^3Y_k\cap V^3Y_k=\emptyset$. Now let $n=2\mathfrak{n}+1$ and $P_k=V^2Y_k$. Then $G=P_1\cup\dotsm\cup P_n$ with $tP_k\cap VP_k=\emptyset$ for $1\le k\le n$. If $P_1, \dotsc, P_n$ are not pairwise disjoint, then by a standard surgery we can obtain the disjointness. The proof is completed.
\end{proof}

\begin{5.6a}
	Let $G$ be a discrete group and $t\in G$ with $t\not=e$. Then there exists a finite partition $G=P_1\cup\dotsm\cup P_n$ such that $tP_k\cap P_k\not=\emptyset$ for all $k=1,\dotsc,n$.
\end{5.6a}

This corollary is comparable with the following well-known classical result:

\begin{5.6b}[{``three-sets'' lemma~\cite{B,K67,Py}}]
Let $G$ be a discrete group and $t\in G$ with $t\not=e$. Then there exists a partition $G=P_1\cup P_2\cup P_3$ such that $tP_k\cap P_k\not=\emptyset$ for all $k=1,2,3$.
\end{5.6b}

\begin{thm}[{cf.~Veech~\cite[Thm.~2.2.1]{V77} and \cite{BJM,Py,Al} for $G$ to be LC}]\label{5.7}
Let $t\in G$ with $t\not=e$. If there exists a symmetric set $V\in\mathfrak{N}_e(G)$ such that
\begin{enumerate}
\item $tV^3\cap V^3=\emptyset$, and
\item $V^3tV^3t^{-1}\cup V^3t^{-1}V^3t$ can be covered by a finite number of right translates of $V$.
\end{enumerate}
Then the canonical action $\rho\colon G\times G^\textsl{\texttt{LUC}}\rightarrow G^\textsl{\texttt{LUC}}$ is $t$-free; that is, $tx\not=x$ for all $x\in G^\textsl{\texttt{LUC}}$.

Consequently, if $G$ is locally quasi-totally bounded, then $G$ acts freely on $G^\textsl{\texttt{LUC}}$, and, the universal compact $\{\textrm{ambit}\}\{\textrm{minimal flow}\}$ of $G$ is free (\cite[Question, p.64]{E69}).
\end{thm}

\begin{proof}
	In view of Theorem~\ref{5.3} and Corollary~\ref{5.5}.2, without loss of generality, we can identity $G$ with $\iota[G]$; i.e., $G$ is a dense subset of $G^\textsl{\texttt{LUC}}$. Let $t\in G$ with $t\not=e$ and $x\in G^\textsl{\texttt{LUC}}$ be any given. Then by Lemma~\ref{5.6}, there are $V\in\mathfrak{N}_e(G)$ and a partition $G=P_1\cup\dotsm\cup P_n$ such that $tP_k\cap VP_k=\emptyset$ for all $k=1,\dotsc,n$.
	
	By $G^\textsl{\texttt{LUC}}=\bigcup_{k=1}^n\textrm{cls}_{G^\textsl{\texttt{LUC}}}{P_k}$, we can pick some $k$ such that $x\in \textrm{cls}_{G^\textsl{\texttt{LUC}}}{P_k}$.
	Since $tP_k\cap VP_k=\emptyset$, by Lemma~\ref{5.5} there exists an $f\in \texttt{LUC}(G)$ such that $f_{|P_k}\equiv0$ and $f_{|tP_k}\equiv1$. Then by \ref{2.8}.4-ii, $f$ can be uniquely extended to $f^\textsl{\texttt{LUC}}=\hat{f}$ on $G^\textsl{\texttt{LUC}}$ to $[0,1]$. Moreover, because $x\in\textrm{cls}_{G^\textsl{\texttt{LUC}}}{P_k}$, $\hat{f}(x)=0$ and $tx\in t\textrm{cls}_{G^\textsl{\texttt{LUC}}}{P_k}=\textrm{cls}_{G^\textsl{\texttt{LUC}}}{tP_k}$ so that $\hat{f}(tx)=1$. Therefore, $x\not=tx$ and $\mathscr{T}^\textsl{\texttt{LUC}}$ is free. Then by Theorem~\ref{3.3}, the universal minimal compact flow of $G$ is also free. The proof is completed.
\end{proof}

\begin{5.7.1}[Veech's theorem in locally totally bounded groups]
If $G$ is locally totally bounded, then the canonical action $\rho\colon G\times G^\textsl{\texttt{LUC}}\rightarrow G^\textsl{\texttt{LUC}}$ is free.
\end{5.7.1}


\begin{5.7.2}
Since $G^\textsl{\texttt{LUC}}$ is a compact Hausdorff right-topological semigroup and we may identify $G$ with a dense subgroup of $G^\textsl{\texttt{LUC}}$\,(by Thm.~\ref{5.3}), Theorem~\ref{5.7} says that if $G$ is locally quasi-totally bounded, then $gx\not=x$ for all $g\in G\setminus\{e\}$ and all $x\in G^\textsl{\texttt{LUC}}$.
\end{5.7.2}

\begin{5.7.3}
However, even for $G$ to be LC, it is possible that $xg=x$ for some $g\in G\setminus\{e\}$ and some $x\in G^\textsl{\texttt{LUC}}$. For instance, let $g_n^{-1}\to y\in G^\textsl{\texttt{LUC}}$ in Remark~\ref{3.7}. Then
	$$
	g_n^{-1}g\to yg\quad \textrm{and}\quad g_n^{-1}g=(g_n^{-1}gg_n)g_n^{-1}\to y
	$$
so that $yg=g$, as desired.
\end{5.7.3}
\end{se}

\begin{se}[Another application of Theorem~\ref{3.1}]\label{5.8}
Let $(\mathscr{X},x_0)$ be a compact ambit of $T$ with phase mapping $(t,x)\mapsto tx$, where $T$ is a topological group. Given any $a\in T$ we can define a flow
$$
\mathscr{X}^a\colon T\times X\rightarrow X,\quad (t,x)\mapsto t\cdot_ax:=ata^{-1}x.
$$
Clearly $\mathscr{X}^a$ is also point-transitive with $\overline{T\cdot_ax_0}=X$.
Since $t\mapsto ata^{-1}$ is an inner-automorphism of $T$, then $\mathscr{X}^a$ is minimal (resp.~proximal, distal) if so is $\mathscr{X}$. Given any $F\in\mathcal{C}(X)$, $x\in X$, and $a\in T$, write $F^x(t|\mathscr{X})=F(tx)$ and $F^x(t|\mathscr{X}^a)=F(t\cdot_ax)$ for $t\in T$.

\begin{5.8.1}
	Given $F\in\mathcal{C}(X)$ and $a\in T$, we have that
	\begin{align*}
		F^{ax_0}(t|\mathscr{X}^a)&=F(t\cdot_a ax_0)\\
		&=F(ata^{-1}ax_0)=F(atx_0)\\
		&=F^{x_0}(at)=F^{x_0}a(t)=(Fa)^{x_0}(t|\mathscr{X})
	\end{align*}
	and
	\begin{align*}
		F^{x_0}(t|\mathscr{X})&=F(tx_0)\\
		&=(Fa^{-1})(atx_0)=(Fa^{-1})(ata^{-1}ax_0)\\
		&=(Fa^{-1})(t\cdot_aax_0)\\
		&=(Fa^{-1})^{ax_0}(t|\mathscr{X}^a)
	\end{align*}
	Noting $Fa, Fa^{-1}\in \mathcal{C}(X)$, we can conclude that
	\begin{equation*}
		\mathbb{F}(\mathscr{X},x_0)=\mathbb{F}(\mathscr{X}^a,ax_0).\leqno{\textrm{(\ref{5.8}.2)}}
	\end{equation*}
\end{5.8.1}
\noindent
Then by Theorem~\ref{3.1} and (\ref{5.8}.2), we could obtain the following:

\begin{5.8.3}[{cf.~\cite{G76,DaG} for $\mathscr{X}$ the universal minimal (proximal) flows of $T$}]
	Given $a\in T$, there exists an isomorphism $\lambda_a\colon\mathscr{X}\rightarrow\mathscr{X}^a$ with $\lambda_a(x_0)=ax_0$. Hence $\lambda_a(x)=ax$ for all $x\in X$.
\end{5.8.3}
\end{se}

\section{Almost automorphic functions and a.a. flows}\label{s6}
\setcounter{thm}{-1}
In this section we will present an alternative complete simple proofs for Veech's Structure Theorems of a.a. flows and a.a. functions (Thm.~\ref{6.2.6} and Thm.~\ref{6.3.11}). Moreover, we shall prove that every endomorphism of Veech's flow of a.a. function is almost 1-1 (Thm.~\ref{6.3.12}). Theorem~\ref{6.4.6} will give us an equivalent description of Bochner a.a. functions on a LC group.

\begin{shs}\label{6.0}
Unless otherwise specified:
\begin{enumerate}
\item[\textbf{1.}] Let $G$ be a Hausdorff left-topological group and by $G_{\!d}$ we denote the discretization of $G$.
\item[\textbf{2.}] Let $l^\infty(G)$ be the set of bounded complex-valued functions, continuous or not, on $G$.
\item[\textbf{3.}] If $G$ is LC, $\beta G$ stands for the Stone-\v{C}ech compactification of $G$ on which there exists a compact Hausdorff right-topological semigroup structure so that $G$ is an open dense subsemigroup of $\beta G$.
\end{enumerate}
\end{shs}
\subsection{Almost automorphic functions}\label{s6.1}
We shall introduce the fundamental theory of a.a. functions on the Hausdorff left-topological group $G$ in $\S$\ref{s6.1}.

\begin{sse}[a.a. functions on $G$]\label{6.1.1}
\item[\;\;\textbf{a.}] A complex-valued function $f$ on $G$ is said to be {\it left almost automorphic} (l.a.a.), denoted $f\in \verb"A"_{\textrm{c,L}}(G)$, if every net $\{t_n\}$ in $G$ has a subnet $\{t_{n^\prime}\}$ and some $\phi\in\mathcal{C}(G)$ such that both
$$
t_{n^\prime}f\to_\textrm{p}\phi\quad \textrm{and}\quad t_{n^\prime}^{-1}\phi\to_\textrm{p}f.
$$
We could define r.a.a. and $\verb"A"_{\textrm{c,R}}(G)$ analogously.

Clearly,  $\verb"A"_{\textrm{c,L}}(G)\cup\verb"A"_{\textrm{c,R}}(G)\subseteq\mathcal{C}(G)$ by considering a special net $\{t_n\}=\{e\}$ so that $f=\phi\in\mathcal{C}(G)$.

\item[\;\;\textbf{b.}] Write $\texttt{A}_\textrm{c}(G)=\texttt{A}_{\textrm{c,L}}(G)\cap\texttt{A}_{\textrm{c,R}}(G)$. If $f\in\verb"A"_\textrm{c}(G)$, then $f$ is called an (Bochner) \textit{a.a. function} on $G$.

\begin{note*}
An a.a. function in our sense above was called a ``continuous'' a.a. function by Veech 1965~\cite[Def.~4.1.1]{V65}. This kind of a.a. functions was studied in Veech~\cite{V65} only for $G$ to be an abelian, LC, $\sigma$-compact, and first countable Hausdorff topological group.
\end{note*}
\end{sse}

\begin{sse}[{cf.~\cite[Thm.~1.2.1-iv]{V65}}]\label{6.1.2}
If $f\in \verb"A"_{\textrm{c,L}}(G)$, then $f\in\mathcal{C}_b(G)$. For otherwise, we can find a net $\{t_n\}$ in $G$ with $|f(t_n)|\to \infty$. So there is no subnet $\{t_{n^\prime }\}$ of $\{t_n\}$ such that $t_{n^\prime}f\to_\textrm{p}\phi$ for some $\phi\in \mathcal{C}(G)$, contrary to Def.~\ref{6.1.1}a.
\end{sse}

\begin{sse}\label{6.1.3}
Let $\{t_n\}$ be any net in $G$. We can select a subnet $\{t_{n^\prime}\}$ of $\{t_n\}$ such that $t_{n^\prime}\to p$ and $t_{n^\prime}^{-1}\to q$ in $\beta G_{\!d}$ or in $\beta G$ if $G$ is a LC topological group. So $st_{n^\prime}\to sp$ and $st_{n^\prime}^{-1}\to sq$ for all $s\in G$.
Now for $f\in\mathcal{C}_b(G)$, let $f^\beta\colon\beta G_{\!d}\rightarrow\mathbb{C}$ or $f^\beta\colon\beta G\rightarrow\mathbb{C}$ be the canonical extension of $f$.

\item[\;\;\textbf{a.}] Thus, we have for each $f\in \verb"A"_{\textrm{c,L}}(G)$ and each $s\in G$ that
$$
t_{n^\prime}f(s)\to \phi(s)=f^\beta(sp)\ \textrm{ and }\
t_{n^\prime}^{-1}(pf^\beta)(s)=f^\beta(st_{n^\prime}^{-1}p)\to f(s)=f^\beta(sqp).
$$
That is, $qpf^\beta(s)=f(s)$ for all $s\in G$.
\item[\;\;\textbf{b.}] If $G$ is a LC topological group and $t_{n^\prime}f\to_\textrm{p}\phi$, then $\phi\in\mathcal{C}_b(G)$ since $\phi(t)=f^\beta(tp)$ for all $t\in G$.

\begin{note*}
Given $f\in \mathcal{C}_b(G)$, $f^\beta$ is itself independent of the right-topological semigroup structure $\cdot_{\!R}^{}$ on $\beta G$, since so is $pt$ and $tp$, for $p\in\beta G$ and $t\in G$. However, $f^\beta(spqt)$ depends generally upon this structure for $p,q\in\beta G\setminus G$ and $s,t\in G$.
\end{note*}
\end{sse}

\begin{sthm}[{cf.~\cite[Thm.~1.2.1, Thm.~1.3.1]{V65} for $G$ a discrete group}]\label{6.1.4}
The following properties are enjoyed by $\verb"A"_{\textrm{c,L}}(G)$:
\begin{enumerate}
		\item If $f, f_1,f_2\in \verb"A"_{\textrm{c,L}}(G)$ and $\lambda\in \mathbb{C}$, then $\lambda f, \bar{f}, f_1+f_2$, $f_1f_2, \mathbf{1}\in \verb"A"_{\textrm{c,L}}(G)$.
		\item If $f\in \verb"A"_{\textrm{c,L}}(G)$, then $\verb"H"_\textrm{p}^L[f]\subseteq\mathcal{C}_b(G)$ and $G\curvearrowright\verb"H"_\textrm{p}^L[f]$ is a compact dynamic.
		\item  $\verb"A"_{\textrm{c,L}}(G)$ is (left and right) invariant; i.e., if $f\in \verb"A"_{\textrm{c,L}}(G)$ and $s,t\in G$, then $sft\in \verb"A"_{\textrm{c,L}}(G)$.
		\item $\verb"A"_{\textrm{c,L}}(G)$ is norm closed such that $\verb"A"_{\textrm{c,L}}(G)\subseteq \verb"A"_{\textrm{c,L}}(G_{\!d})=\verb"A"_{\textrm{c,R}}(G_{\!d})$.	
\item If $f\in\verb"LUC"(G)$, then $f\in\verb"A"_{\textrm{c,L}}(G)$ iff $f\in\verb"A"_{\textrm{c,L}}(G_{\!d})$ iff $f\in\verb"A"_\textrm{c,R}(G_{\!d})$. (Thus, if $G$ is a topological group and $f\in\verb"UC"(G)$, then $f\in\verb"A"_{\textrm{c,L}}(G)$ iff $f\in\verb"A"_\textrm{c,R}(G)$.)
\item If $G$ is a LC topological group, then
$$
\verb"A"_\textrm{c}(G_{\!d})\cap\mathcal{C}(G)=\verb"A"_{\textrm{c,L}}(G)=\verb"A"_{\textrm{c,R}}(G)=\verb"A"_\textrm{c}(G)\subseteq\verb"UC"(G)
$$
(cf.~\cite[Lem.~4.1.1]{V65} for $G$ to be abelian, LC, $\sigma$-compact, and first countable).
\end{enumerate}
Hence $\verb"A"_{\textrm{c,L}}(G)$ is a norm closed, right $G$-invariant, and self-adjoint subalgebra of $\mathcal{C}_b(G)$ containing the constants.
\end{sthm}

\begin{proof}
	\item 1.: Obvious.
	
	\item 2.: First, $f\in\mathcal{C}_b(G)$ by \ref{6.1.2}. Now let $\phi\in \verb"H"_\textrm{p}^L[f]$. Then there is a net $t_n\in G$ such that $t_nf\to_\textrm{p}\phi$. By Def.~\ref{6.1.1}, there is a subnet $\{t_{n^\prime}\}$ of $\{t_n\}$ such that $t_{n^\prime}f\to_\textrm{p}\phi\in \mathcal{C}_b(G)$. Thus, $\verb"H"_\textrm{p}^L[f]\subseteq\mathcal{C}_b(G)$. On the other hand, since
	$t_nf(xt)\to\phi(xt)=t\phi(x)$ for all $x,t\in G$, hence $t\phi \in\verb"H"_\textrm{p}^L[f]$. Thus, $\rho\colon G\times\verb"H"_\textrm{p}^L[f]\rightarrow\verb"H"_\textrm{p}^L[f]$, $(t,\phi)\mapsto t\phi$, is a well-defined separately continuous phase mapping. That is, $G\curvearrowright\verb"H"_\textrm{p}^L[f]$ is a compact dynamic.
	
	If $G$ is a LC topological group, then by Ellis's Joint Continuity Theorem, $G\curvearrowright\verb"H"_\textrm{p}^L[f]$ is a compact flow. So $f\in\texttt{LUC}(G)$ by \ref{2.4}.1, since $f$ may come from the compact ambit $(\mathscr{H}_{\textrm{p}}^L[f],f)$ of $G$.
	
	\item 3.: Obvious.
	
	\item 4.: Let $f_i\in \verb"A"_{\textrm{c,L}}(G)$ be a sequence with $f_i\to_{\|\cdot\|_\infty}f\in\mathcal{C}_b(G)$ as $i\to\infty$. Let $\{t_n\}$ be any net in $G$
	with a subnet $\{t_{n^\prime}\}$ as in \ref{6.1.3}. Then, as $i\to\infty$,
	$$
	|qpf_i^\beta(x)-qpf^\beta(x)|=|(f_i-f)^\beta(xqp)|\le \|f_i-f\|_\infty\to 0
	$$
	uniformly for $x\in G$. Thus, by the triangle inequality and $qpf_i^\beta=f_i$ on $G$, we have as $i\to\infty$ that
	$$|qpf^\beta(x)-f(x)|\le|qpf^\beta(x)-qpf_i^\beta(x)|+|f_i(x)-f(x)|\to0$$
	uniformly for $x\in G$. Moreover, since
	$$|f_i^\beta(xp)-f^\beta(xp)|=|(f_i-f)^\beta(xp)|\le\|f_i-f\|_\infty\to0$$
	uniformly for $x\in G$, we have that $pf^\beta$ is continuous restricted to $G$. Thus, $f\in \verb"A"_{\textrm{c,L}}(G)$ by Def.~\ref{6.1.1}a.
	
	$\verb"A"_{\textrm{c,L}}(G)\subseteq \verb"A"_{\textrm{c,L}}(G_{\!d})$ is obvious.

$\verb"A"_{\textrm{c,L}}(G_{\!d})=\verb"A"_{\textrm{c,R}}(G_{\!d})$ is exactly \cite[Thm.~1.3.1]{V65}.
We will, however, present an alternative simple proof here.	Let $f\in \verb"A"_{\textrm{c,L}}(G_{\!d})$ then it is enough to prove that $f\in \verb"A"_{\textrm{c,R}}(G_{\!d})$. Suppose there exists a net $\{t_n\}$ in $G$ and $\phi,\psi\in l^\infty(G)$ such that $ft_n\to_\textrm{p}\phi$ and $\phi t_n^{-1}\to_\textrm{p}\psi$. To prove $f=\psi$, we may assume $t_n\to p$ and $t_n^{-1}\to q$ in $\beta G_{\!d}$. Then, for all $x\in G$,
	\begin{equation*}
		\psi(x)={\lim}_n\phi t_n^{-1}(x)={\lim}_n(t_n^{-1}x)\phi(e)=\phi^\beta(qx)
	\end{equation*}
and
	\begin{equation*}
		\phi(x)={\lim}_nf t_n(x)={\lim}_n(t_nx)f(e)=p(xf)(e)=f^\beta p(x).
	\end{equation*}
	So, $\phi^\beta=f^\beta p$ and then $\psi(x)=pqxf(e)$. Since $xf\in \verb"A"_{\textrm{c,L}}(G)$, we have that $pqxf=xf$ by \ref{6.1.3} and $\psi=f$.
	
\item 5.: By 4. and Lemma~\ref{4.2}.

\item 6.: By 4. and \ref{6.1.3}b, it follows that $\verb"A"_{\textrm{c,L}}(G)=\verb"A"_\textrm{c}(G_{\!d})\cap\mathcal{C}(G)=\verb"A"_{\textrm{c,R}}(G)$. Finally by 2. and Ellis's Joint Continuity Theorem, $G\curvearrowright\texttt{H}_\textrm{p}^L[f]$ is a compact flow so that $f\in\texttt{LUC}(G)$ by \ref{2.4}.1.
Thus, $\verb"A"_\textrm{c,R}(G)=\verb"A"_\textrm{c}(G)\subseteq\verb"UC"(G)$.

The proof of Theorem~\ref{6.1.4} is therefore completed.
\end{proof}

\begin{cor*}
Let $f=g+ih\in\mathcal{C}(G)$, where $g,h$ are continuous real functions on $G$. Then $f\in\verb"A"_{\textrm{c,L}}(G)$ iff $g,h\in\verb"A"_{\textrm{c,L}}(G)$.
\end{cor*}

\begin{slem}\label{6.1.5}
	If $f\in l^\infty(G)$ is a.a. on $G_{\!d}$, then for every finite set $\{t_1,\dotsc,t_k\}$ in $G$, $(t_1f,\dotsc,t_kf)$ is an a.p. point for $G_{\!d}\curvearrowright X$, where
	$$X=\stackrel{k\textrm{-times}}{\overbrace{l^\infty(G)\times\dotsm\times l^\infty(G)}}$$
	is under the topology of pointwise convergence. Thus, for all $f\in \verb"A"_{\textrm{c,L}}(G)$, $G\curvearrowright \verb"H"_\textrm{p}^L[f]$ is a minimal compact dynamic.
\end{slem}

\begin{proof}
By Theorem~\ref{6.1.4}, $\overline{G(t_1f,\dotsc,t_kf)}^\textrm{p}$ is a compact set in $X$ so that by Zorn's Lemma we can select an a.p. point, say $(\phi_1,\dotsc,\phi_k)$, in it for $G_{\!d}\curvearrowright X$. Then there is a net $\{t_n\}$ in $G$ with $t_n(t_1f,\dotsc,t_kf)\to_\textrm{p}(\phi_1,\dotsc,\phi_k)$. Since $t_1f,\dotsc,t_kf$ all are l.a.a. on $G_{\!d}$ by Theorem~\ref{6.1.4}, there is a subnet $\{t_{n^\prime}\}$ of $\{t_n\}$ such that $t_{n^\prime}^{-1}(\phi_1,\dotsc,\phi_k)\to_\textrm{p}(t_1f,\dotsc,t_kf)$. Thus, $(t_1f,\dotsc,t_kf)$ is also a.p. under $G_{\!d}\curvearrowright X$. The proof is completed.
\end{proof}

In fact, it is known that a l.a.a. function is a ``point-distal function'' as follows. See \cite{V65} by Structure Theorem of a.a. flows and \cite{F81,D20} by $\Delta^{\!*}$-recurrence. Here we will present a simple short proof.

\begin{slem}\label{6.1.6}
	Let $f\in \verb"A"_{\textrm{c,L}}(G)\cap\verb"LUC"(G)$. Then $G\curvearrowright\verb"H"_\textrm{p}^L[f]$ is a minimal compact flow and
	$(f,\phi)$ is a.p. under $G\curvearrowright\verb"H"_\textrm{p}^L[f]\times\verb"H"_\textrm{p}^L[f]$ for all $\phi\in \verb"H"_\textrm{p}^L[f]$.
\end{slem}

\begin{proof}
First by Lemma~\ref{4.2}, $G\curvearrowright\verb"H"_\textrm{p}^L[f]$ is a compact flow.
Let $\phi\in \verb"H"_\textrm{p}^L[f]$. By Lemmas~\ref{6.1.5}, $G\curvearrowright\verb"H"_\textrm{p}^L[f]$ is a minimal compact flow. Let $\mathbb{I}$ be a minimal left ideal in $\beta G_{\!d}$. Then $\verb"H"_\textrm{p}^L[f]=\mathbb{I}f$. Let $q\in\mathbb{I}$ with $\phi=qf$; then we can choose a net $t_n\in G$ with $t_n^{-1}\to q$ and $t_n\to p$ in $\beta G_{\!d}$. So
	$$
	(f,qf)=(qpf, qf)=q(pf,f)\in \verb"H"_\textrm{p}^L[f]\times\verb"H"_\textrm{p}^L[f]
	$$
	is a.p. for $G_{\!d}\curvearrowright\verb"H"_\textrm{p}^L[f]\times\verb"H"_\textrm{p}^L[f]$ and then for $G\curvearrowright\verb"H"_\textrm{p}^L[f]\times\verb"H"_\textrm{p}^L[f]$. The proof is completed.
\end{proof}

\begin{sse}\label{6.1.7}
Recall that a set $A$ in $G$ is (left-)\textit{syndetic}~\cite{GH} if $G=K^{-1}A$, or equivalently, $Kt\cap A\not=\emptyset$ $\forall t\in G$, for some compact set $K$ in $G$. If there is a finite set $L\subset G$ with $G=LA=AL$, then $A$ is called \textit{relatively dense} in $G$~(cf.~\cite{V65}). Then $A$ is relatively dense in $G$ iff it is left-syndetic and right-syndetic in $G_{\!d}$.

Given any $f\in l^\infty(G)$, a finite nonempty subset $N$ of $G$ and a real number $\epsilon>0$, put
$$
U(f;N,\epsilon)=\{\phi\in l^\infty(G)\,|\, {\max}_{t\in N}|\phi(t)-f(t)|<\epsilon\}\leqno{(\ref{6.1.7}\textrm{a})}
$$
that is an open neighborhood of $f$ in $l^\infty(G)$ under the pointwise topology; and
$$
C(f; N,\epsilon)=\{\tau\in G\,|\, {\max}_{s,t\in N}|f(s\tau t)-f(st)|<\epsilon\}.\leqno{(\ref{6.1.7}\textrm{b})}
$$
Then,
$$
C(f; N,\epsilon)=\{\tau\in G\,|\,fs\tau\in U(fs;N,\epsilon)\ \forall s\in N\}\leqno{(\ref{6.1.7}\textrm{c})}
$$
and
\begin{equation*}
C(f; N,\epsilon)=\{\tau\in G\,|\,\tau tf\in U(tf;N,\epsilon)\ \forall t\in N\}.\leqno{(\ref{6.1.7}\textrm{d})}
\end{equation*}
\end{sse}

\begin{slem}\label{6.1.8}
Let $f\in \verb"A"_{\textrm{c,L}}(G)$. Given a finite set $N\subset G$ and a real number $\epsilon>0$, there exists a finite superset $M$ of $N$ and a positive real number $\delta<\epsilon$ such that
$$N_G(f,U(f;M,\delta))^{-1}\cup N_G(f,U(f;M,\delta))\subseteq N_G(f,U(f;N,\epsilon))$$
where $N_G(f,U)$ is as in Def.~\ref{1.2}f under $G\curvearrowright l^\infty(G)$.
\end{slem}

\begin{proof}
	Otherwise, we can select a net $\{t_n\}$ in $G$ such that $t_nf\to_\textrm{p}f$ but $t_n^{-1}f\notin U(f;N,\epsilon)$.
	Thus, $t_nf\to_\textrm{p}f$ and $t_n^{-1}f\not\to_\textrm{p}f$, contrary to $f\in \verb"A"_{\textrm{c,L}}(G_{\!d})$. The proof is complete.
\end{proof}

\begin{slem}\label{6.1.9}
	Let $f\in l^\infty(G)$ be a.a. on $G_{\!d}$. If $N$ is a finite subset of $G$ and $\epsilon>0$, then $C(f; N,\epsilon)$ contains a symmetric syndetic set in $G_{\!d}$.
\end{slem}

\begin{proof}
	By Lemma~\ref{6.1.5}, Lemma~\ref{6.1.8}, and (\ref{6.1.7}d).
\end{proof}

\begin{cor*}[{cf.~\cite[Lem.~2.1.1]{V65} by a different proof}]
	Let $f\in l^\infty(G)$ be a.a. on $G_{\!d}$. If $N$ is a finite subset of $G$ and $\epsilon>0$, then $C(f; N,\epsilon)$ is relatively dense in $G$.
\end{cor*}

\begin{proof}
By Lemma~\ref{6.1.9} and the fact that a syndetic symmetric set in $G_{\!d}$ is relatively dense.
\end{proof}

\begin{slem}[{cf.~\cite[Lem.~2.1.2]{V65} or \cite[Thm.~1.9]{D20}}]\label{6.1.10}
	Let $f\in l^\infty(G)$ be a.a. on $G_{\!d}$. Given $\epsilon>0$ and a finite set $N\subset G$ there exist $\delta>0$ and a finite superset $M$ of $N$ such that $\sigma^{-1}\tau\in C(f;N,\epsilon)$ for all $\sigma,\tau\in C(f;M,\delta)$.
\end{slem}

\begin{proof}
See Appendix~\ref{A}.
\end{proof}

In fact, we shall show in $\S$\ref{s6.4} that the converse of Lemma~\ref{6.1.10} is also true; see Theorem~\ref{6.4.6} in terms of the so-called Bohr a.a. functions.
\subsection{Almost automorphic flows}\label{s6.2}

\begin{sse}\label{6.2.1}
Let $\mathscr{X}=(G,X)$ be a compact flow and by $\mathscr{U}_X$ we denote the uniformity structure of $X$.
\item[\;\;\textbf{a.}] $\mathscr{X}$ is called {\it equicontinuous} if given $\varepsilon\in\mathscr{U}_X$ there exists $\delta\in\mathscr{U}_X$ such that $G\delta\subseteq\varepsilon$. Later $\mathscr{X}_{\!eq}$ will stand for the maximal a.p. factor of $\mathscr{X}$ (see \cite{E69, G76, A88, De} for the existence of $\mathscr{X}_{\!eq}$).

\item[\;\;\textbf{b.}] Let
\begin{equation*}
	\begin{split}
		&\texttt{P}(\mathscr{X})=\{(x,y)\,|\,\overline{G(x,y)}\cap\Delta_X\not=\emptyset\},\\
		&\texttt{RP}(\mathscr{X})=\{(x,y)\,|\,\overline{G(\varepsilon[x]\times \varepsilon[y])}\cap\Delta_X\not=\emptyset\ \forall \varepsilon\in\mathscr{U}_X\}.
	\end{split}
\end{equation*}
be the \textit{proximal relation} and \textit{regionally proximal relation} of $\mathscr{X}$, respectively.

Then $\mathscr{X}$ is equicontinuous iff $\texttt{RP}(\mathscr{X})=\Delta_X$. If $\overline{Gx}=X$, then $x$ is a distal point for $\mathscr{X}$ iff $\texttt{P}[x]=\{x\}$.

\item[\;\;\textbf{c.}] A point $x\in X$ is said to be \textit{a.a. under $\mathscr{X}$}, denoted $x\in \verb"A"(\mathscr{X})$, if for every net $\{t_n\}$ in $G$, $t_nx\to y$ implies that $t_n^{-1}y\to x$. If $\mathscr{X}$ has an a.a. point that has a dense orbit in $X$, then $\mathscr{X}$ is called an \textit{a.a. flow}. See \cite{V65, D20}.
\end{sse}
Then by Def.~\ref{6.1.1}a, $G\curvearrowright\verb"H"_\textrm{p}^L[f]$ is an a.a. flow with $f$ an a.a. point whenever $f\in\verb"A"_{\textrm{c,L}}(G)\cap\verb"LUC"(G)$ or $f\in\verb"A"_{\textrm{c,L}}(G)$ with $G$ a LC topological group.

\begin{slem}\label{6.2.2}
Let $\mathscr{X}=G\curvearrowright\verb"H"_{\textrm{p}}^L[f]$, where $f\in\verb"A"_{\textrm{c,L}}(G)\cap\verb"LUC"(G)$. Then $\verb"P"(\mathscr{X})=\verb"RP"(\mathscr{X})$; and moreover, $\verb"P"[f]=\{f\}$.
\end{slem}

\begin{proof}
First by Lemma~\ref{6.1.6}, we have $\texttt{P}[f]=\{f\}$ under $\mathscr{X}$. Moreover, it is clear that $\texttt{P}(\mathscr{X})\subseteq\texttt{RP}(\mathscr{X})$. For the other inclusion, let $(x,y)\in\texttt{RP}(\mathscr{X})$. By Lemma~\ref{6.1.5}, we can choose nets $s_n, s_n^\prime, t_n\in G$ such that $s_nf\to_\textrm{p}x$, $s_n^\prime f\to_\textrm{p}y$, $t_n(s_nf,s_n^\prime f)\to_\textrm{p}(f,f)$.

Given $N\subset G$ a finite set and $\epsilon>0$, it follows from Lemmas~\ref{6.1.10} and \ref{6.1.9} that there exists a syndetic set $B$ in $G_{\!d}$ with $BB\subseteq N_G(f,U(f;N,\epsilon))$. Then by the ``two-syndetic sets'' criterion (\cite[Thm.~2.3]{MW72}), there exists a finite superset $M$ of $N$, a positive real number $\delta<\epsilon$, and a syndetic set $A$ in $G_{\!d}$ such that $AU(f;M,\delta)\subseteq U(f;N,\epsilon)$. Furthermore, we can write $ka_n=t_n^{-1}$ with $a_n\in A$ and $ k\in G$. Then
$$
a_n(t_ns_nf)\to_\textrm{p}\phi\in\overline{U(f;N,\epsilon)}^\textrm{p},\quad s_nf\to_\textrm{p}x=k\phi
$$
and
$$
a_n(t_ns_n^\prime f)\to_\textrm{p}\phi^\prime\in\overline{U(f;N,\epsilon)}^\textrm{p},\quad s_n^\prime f\to_\textrm{p}y=k\phi^\prime.
$$
Thus, $(x,y)\in\texttt{P}(\mathscr{X})$ and $\texttt{P}(\mathscr{X})=\texttt{RP}(\mathscr{X})$. The proof is complete.
\end{proof}

\begin{sse}\label{6.2.3}
In fact, if $\mathscr{X}$ is any a.a. flow, then $\texttt{P}(\mathscr{X})=\texttt{RP}(\mathscr{X})$, and moreover, $\texttt{P}[x]=\{x\}$ iff $x\in \texttt{A}(\mathscr{X})$ (cf.~\cite{D20} by improving Lem.~\ref{6.1.10} or \cite{MW72} by using Veech's Structure Theorem of a.a. flow---Thm.~\ref{6.2.6} below).
\end{sse}

\begin{slem}\label{6.2.4}
Let $\mathscr{X}$ be a compact flow and $x_0\in X$. Then $x_0$ is a.a. under $\mathscr{X}$ iff $\mathbb{F}(\mathscr{X},x_0)\subseteq\verb"A"_{\textrm{c,L}}(G)$.
\end{slem}

\begin{proof}
Necessity is obvious. For sufficiency, suppose to the contrary that $x_0$ is not a.a. for $\mathscr{X}$. Then we can choose a net $\{t_n\}$ in $G$ with $t_nx_0\to y$ and $t_n^{-1}y\to x_1\not=x_0$. Next, we can select $F\in\mathcal{C}(X)$ with $F(x_0)\not=F(x_1)$. So $F^{x_0}\in\mathbb{F}(\mathscr{X},x_0)$, $F^{x_0}\not= F^{x_1}$. Then
$t_nF^{x_0}\to_\textrm{p} F^y$ and $t_n^{-1}F^y\to_\textrm{p}F^{x_1}\not=F^{x_0}$, contrary to $F^{x_0}\in\verb"A"_{\textrm{c,L}}(G)$. The proof is completed.
\end{proof}

The following important result is comparable with Veech's Structure Theorem of a.a. functions---Thm.~\ref{6.3.11}.

\begin{slem}\label{6.2.5}
Let $f\in\verb"LUC"(G)$ or let $f\in\mathcal{C}_b(G)$ with $G$ a LC topological group. Then $f\in\verb"A"_{\textrm{c,L}}(G)$ iff $\mathscr{H}_{\textrm{p}}^L[f]$ is an almost 1-1 extension of $\mathscr{H}_{\textrm{p}}^L[f]_{eq}$, where $\mathscr{H}_{\textrm{p}}^L[f]_{eq}=G\curvearrowright \verb"H"_\textrm{p}^L[f]_{eq}$ and $\rho\colon\verb"H"_\textrm{p}^L[f]\rightarrow \verb"H"_\textrm{p}^L[f]_{eq}$ is 1-1 at $f$.

\begin{note*}[{cf.~\cite[Thm.~3.3.1]{V65} for $G$ a discrete group}]
So, $f\in\mathcal{C}_b(G)$ is left a.p. (von Neumann 1934; i.e., $Gf$ is relatively compact in $(\mathcal{C}_b(G),\|\cdot\|_\infty)$) iff $\verb"H"_\textrm{p}^L[f]\subseteq\verb"A"_{\textrm{c,L}}(G)$.
\end{note*}
\end{slem}

\begin{proof}
\item \textit{Sufficiency:} By Lemma~\ref{6.2.4} and $f\in\mathbb{F}(\mathscr{H}_{\textrm{p}}^L[f],f)$.
\item \textit{Necessity:} Assume $f\in\verb"A"_{\textrm{c,L}}(G)$. Since $f\in\verb"LUC"(G)$, it follows by Lemmas~\ref{4.2} and \ref{6.1.5} that $\mathscr{H}_{\textrm{p}}^L[f]=G\curvearrowright\verb"H"_\textrm{p}^L[f]$ is a minimal compact flow with $f$ an a.a. point. By Lemma~\ref{6.2.2}, $\verb"P"(\mathscr{H}_{\textrm{p}}^L[f])=\verb"RP"(\mathscr{H}_{\textrm{p}}^L[f])$ is an invariant closed equivalence relation so that $Y=\verb"H"_\textrm{p}^L[f]/\verb"RP"(\mathscr{H}_{\textrm{p}}^L[f])$ is the maximal a.p. factor of $\mathscr{H}_{\textrm{p}}^L[f]$. Let $\rho\colon\verb"H"_\textrm{p}^L[f]\rightarrow Y$ be the canonical mapping. Then by Lemma~\ref{6.2.2} again, it follows that $\rho$ is 1-1 at the point $f$. The proof is completed.
\end{proof}

\begin{proof}[\textbf{Proof of Note}]
If $f\in\mathcal{C}_b(G)$ is left a.p., then $\texttt{H}_\textrm{p}^L[f]=\overline{Gf}^{\|\cdot\|_\infty}$ and $\mathscr{H}_\textrm{p}^L[f]$ is an equicontinuous compact flow so $f\in\texttt{LUC}(G)$ and $\verb"H"_\textrm{p}^L[f]\subseteq\verb"A"_{\textrm{c,L}}(G)$. Now conversely, suppose $\verb"H"_\textrm{p}^L[f]\subseteq\verb"A"_{\textrm{c,L}}(G)$. Then $\verb"H"_\textrm{p}^L[f]\subseteq\verb"A"_{\textrm{c,L}}(G_{\!d})$. Thus, for $G_{\!d}$, $\rho\colon\verb"H"_\textrm{p}^L[f]\rightarrow \verb"H"_\textrm{p}^L[f]_{eq}$ is 1-1 onto so that $G_{\!d}\curvearrowright\verb"H"_\textrm{p}^L[f]$ is equicontinuous. Thus, $G\curvearrowright\verb"H"_\textrm{p}^L[f]$ is an equicontinuous compact flow.
So the pointwise topology coincides with the $\|\cdot\|_\infty$-topology on $\texttt{H}_\textrm{p}^L[f]$ and $f$ is left a.p. of von Neumann. The proof is completed.
\end{proof}

\begin{sthm}[Structure Theorem of a.a. flow]\label{6.2.6}
A minimal compact flow $\mathscr{X}=G\curvearrowright X$ is a.a. iff the canonical mapping $\rho\colon\mathscr{X}\rightarrow\mathscr{X}_{\!eq}$ is almost 1-1 (and, in fact, $\verb"A"(\mathscr{X})=\{x\in X\,|\,\{x\}=\rho^{-1}[\rho(x)]\}$).

Thus, if $\mathscr{X}$ is a minimal pointwise a.a. compact flow, then it is equicontinuous.
\end{sthm}

\begin{note*}
In \cite{V65} Veech only proved his structure theorem of a.a. function that is essentially different with Lemma~\ref{6.2.5} (see Thm.~\ref{6.3.11} and Rem.~\ref{6.3.13} below). This important theorem was first claimed by Veech 1965 in \cite[p.741]{V65} and then 1977 in \cite[Thm.~2.1.2]{V77} without a detailed proof. See Auslander-Markley 2009 \cite[Thm.~19]{AM} for a detailed proof in the special case that $G$ is an abelian group and $X$ is a compact metric space. A complete proof appears in Dai 2020 \cite[Thm.~3.6]{D20} by using a variation \cite[Thm.~1.9]{D20} of Lemma~\ref{6.1.10}. Here we shall give an alternative self-closed simple proof based on Lemma~\ref{6.2.5}.
\end{note*}

\begin{proof}[Proof (without using \ref{6.2.3})]
Sufficiency is obvious. Now conversely, suppose $\mathscr{X}$ is an a.a. flow with $x_0\in\texttt{A}(\mathscr{X})$.

Given $F\in\mathcal{C}(X)$, let $\rho_F^{}\colon\verb"H"_\textrm{p}^L[F^{x_0}]\rightarrow Y_F=\texttt{H}_\textrm{p}^L[F^{x_0}]_{eq}$ be the almost 1-1 extension given by Lemma~\ref{6.2.5} for $f=F^{x_0}$. Let $\varrho_F^{}\colon X\rightarrow\verb"H"_\textrm{p}^L[f]$ be defined by $x\mapsto F^x$. Then $\varrho_F^{}$ is an extension of minimal flows by Lemma~\ref{2.4}.3. Set
$$
y_{0,F}=\rho_F^{}(F^{x_0}), \ y_0=(y_{0,F})_{F\in\mathcal{C}(X)}\in{\prod}_{F\in\mathcal{C}(X)}Y_F,\ Y=\overline{Gy_0}.
$$
Define
$$
\rho\colon X\xrightarrow[\varrho]{x\mapsto(F^x)_{F\in\mathcal{C}(X)}} \prod_{F\in\mathcal{C}(X)}\verb"H"_\textrm{p}^L[F^{x_0}]\xrightarrow[\rho^\prime]{(F^x)_{F\in\mathcal{C}(X)}\mapsto(\rho_F^{}(F^x))_{F\in\mathcal{C}(X)}}\prod_{F\in\mathcal{C}(X)}Y_F.
$$
Clearly, $G\curvearrowright Y$ is a minimal a.p. compact flow and $\rho\colon X\rightarrow Y$ is an extension of flows with $y_0=\rho(x_0)$.

Next we need prove that $\rho$ is almost 1-1 at the point $x_0$. Indeed, for $x\not=x_0$ in $X$, we can choose $\eta\in\mathcal{C}(X)$ with $\eta(x_0)\not=\eta(x)$ so that $$
\varrho(x_0)=(F^{x_0})_{F\in\mathcal{C}(X)}\not=(F^{x})_{F\in\mathcal{C}(X)}=\varrho(x).
$$
Moreover, since $\rho^\prime$ is 1-1 at the point $(F^{x_0})_{F\in\mathcal{C}(X)}$, $\rho$ is 1-1 at the point $x_0$.
Thus, $\mathscr{X}$ is an almost 1-1 extension of $\mathscr{Y}$. This implies that $\mathscr{X}_{\!eq}\cong\mathscr{Y}$. The proof is completed.
\end{proof}

\begin{proof}[\textbf{Proof of \ref{6.2.3}}]
Let $x_0\in\texttt{A}(\mathscr{X})$. Then by Theorem~\ref{6.2.6}, $\mathscr{X}$ is locally a.p. at $x_0$ so $\texttt{P}(\mathscr{X})=\texttt{RP}(\mathscr{X})$ by \cite[Prop.~5.27]{E69}.
\end{proof}

\begin{sse}[Universal a.a. flows of $G$]\label{6.2.7}
By Theorem~\ref{6.1.4} and \ref{2.2}.4, it follows that $\verb"A"_{\textrm{c,L}}(G)\cap\verb"LUC"(G)$ is a norm closed self-adjoint right-invariant subalgebra of $\mathcal{C}_b(G)$ containing the constants.
\item[\;\;\textbf{a$_1$.}] Put $\verb"M"_{a}=|\verb"A"_{\textrm{c,L}}(G)\cap\verb"LUC"(G)|$, which is the maximal ideal space of $\verb"A"_{\textrm{c,L}}(G)\cap\verb"LUC"(G)$ as in \ref{2.9}. Write $\mathscr{M}_{\!a}$ for $G\curvearrowright\verb"M"_{a}$, which will be called the \textit{canonical a.a. flow} of $G$.

\item[\;\;\textbf{a$_2$.}] Let $\mathscr{M}_{\!eq}=G\curvearrowright\verb"M"_{eq}$ be the universal minimal a.p./equicontinuous compact flow of $G$. Although $G$ need not be a topological group under \ref{6.0}.1, it is easily to verify the existence of $\mathscr{M}_{\!eq}$.

In the category of minimal compact flows of $G$, $\mathscr{X}$ is called a \textit{maximal almost 1-1 extension} of $\mathscr{M}_{\!eq}$ if $\mathscr{X}$ is an almost 1-1 extension of $\mathscr{M}_{\!eq}$ such that any other almost 1-1 extension of $\mathscr{M}_{\!eq}$ must be a factor of $\mathscr{X}$.

\begin{6.2.7B}
$\mathscr{M}_{\!a}$ is an a.a. flow with $e^*\in\verb"A"(\mathscr{M}_{a})$.
\end{6.2.7B}

\begin{proof}
First by \ref{2.8}.2, $(\mathscr{M}_{\!a},e^*)$ is a compact ambit of $G$. It follows by (\ref{2.9}a) that $\mathbb{F}(\mathscr{M}_{\!a},e^*)=\verb"A"_{\textrm{c,L}}(G)\cap\verb"LUC"(G)$. Then by Lemma~\ref{6.2.4}, $e^*\in\texttt{A}(\mathscr{M}_{a})$ so that $\mathscr{M}_{\!a}$ is an a.a. flow of $G$.
\end{proof}

Note here that we cannot show that $\phi^{-1}[y]$ must contain a.a. points when $\phi\colon\mathscr{X}\rightarrow\mathscr{Y}$ is an extension of a.a. flows with $y\in\texttt{A}(\mathscr{Y})$.

The proof of the following lemma is an improvement of Auslander's argument of \cite[Lem.~8.2]{A88}.

\begin{6.2.7C}[{cf.~\cite[Lem.~8.2]{A88} for minimal flows with $A$ a maximal a.p. subset}]
Let $\mathscr{X}=(G,X)$ be an a.a. flow and $A=\verb"A"(\mathscr{X})$. Let $\z\in X^A$ with $\z(a)=a\ \forall a\in A$ and $Z=\overline{G\z}$. Suppose $\varphi$ is an endomorphism of $G\curvearrowright Z$ and $A^\prime=\varphi(\z)[A]$. If $z_1$, $z_2\in Z$ with $z_1|_{A^\prime}\not=z_2|_{A^\prime}$, then $\varphi(z_1)\not=\varphi(z_2)$.
\end{6.2.7C}

\begin{proof}
We first note that $\mathscr{Z}=G\curvearrowright Z$ is an a.a. flow with $\z\in\texttt{A}(\mathscr{Z})$. Moreover, for all $z\in Z$ and all $a_1,a_2\in A$ with $a_1\not=a_2$, we have that $z(a_1)\not=z(a_2)$, since $(\z(a_1),\z(a_2))$ is a distal point for $G\curvearrowright X\times X$. So, $z\colon A\rightarrow X$ is 1-1 for all $z\in Z$.

Let $z\in Z$ be an a.a. point under $\mathscr{Z}$. Then $z[A]$ is an a.a. set in $X$ (i.e., given any finite set $\{a_1,\dotsc,a_k\}\subseteq A$, $(z(a_1),\dotsc,z(a_k))\in X^k$ is an a.a. point for $G\curvearrowright X^k$). Clearly, $z\colon A\rightarrow z[A]\subseteq A$ is 1-1.

Since $(\z,\varphi(\z))$ is an a.a. point for $G\curvearrowright Z\times Z$ by Def.~\ref{6.2.1}c, hence $\z[A]\cup\varphi(\z)[A]$ is an a.a. set in $X$. But $\z[A]$ is a maximal a.a. set in $X$, so $\varphi(\z)[A]\subseteq A$.

If $\gamma\colon A\rightarrow A$ is an injection, let $\gamma^*\colon X^A\rightarrow X^A$ be a continuous mapping defined by $z\mapsto\gamma^*(z)$ where $\gamma^*(z)(a)=z(\gamma(a))$ for all $a\in A$. Clearly, $\gamma^*$ is $G$-equivariant so $\gamma^*$ is a self-homomorphism of $G\curvearrowright X^A$.

We now define $\gamma\colon A\rightarrow A$ by $\gamma(a)=\varphi(\z)(a)$ for all $a\in A$. Since $\varphi(\z)\colon A\rightarrow X$ is 1-1 and $A^\prime:=\varphi(\z)[A]\subseteq A$, $\gamma$ is an injection of $A$ to itself with $\gamma[A]=A^\prime$ and $\gamma^*(\z)=\varphi(\z)$.

Since $\varphi(\z)\in Z$, $\gamma^*[Z]\cap Z\not=\emptyset$, so $\gamma^*[Z]=Z$. Now $\gamma^*(\z)=\varphi(\z)$, so $\varphi=\gamma^*|_Z$ and thus, $\gamma^*|_Z$ is an endomorphism of $G\curvearrowright Z$.

Let $z_1$, $z_2\in Z$ with $z_1|_{A^\prime}\not=z_2|_{A^\prime}$. By $\varphi=\gamma^*|_Z$, it follows easily that $\varphi(z_1)\not=\varphi(z_2)$.
The proof is completed.
\end{proof}

\begin{6.2.7D}
Let $A=\verb"A"(\mathscr{M}_{\!a})$ and let $\mathscr{Z}=(G,Z)$ be defined as in Lemma~\ref{6.2.7}C with $\mathscr{X}=\mathscr{M}_{\!a}$. Then:
\begin{enumerate}
\item[(1)] $\mathscr{M}_{\!a}$ is a universal a.a. flow of $G$.

\item[(2)] $\mathscr{Y}$ is a universal a.a. flow of $G$ iff it is a maximal almost 1-1 extension of $\mathscr{M}_{\!eq}$.

\item[(3)] If $\phi$ is an endomorphism of $\mathscr{M}_{\!a}$ such that $a\in\eta(\phi(a))[A]$ for some $a\in A$ and some $\eta\colon\mathscr{M}_{\!a}\rightarrow\mathscr{Z}$, then $\phi$ is an automorphism of $\mathscr{M}_{\!a}$.
\end{enumerate}
In particular, the universal a.a. flow of $G$ is unique up to almost 1-1 extensions.
\end{6.2.7D}

\begin{proof}
\item (1): By Lemma~\ref{6.2.7}B, $\mathscr{M}_{\!a}$ is an a.a. flow of $G$. To prove the university, let $\mathscr{Y}$ is an a.a. flow of $G$ with $y_0\in\texttt{A}(\mathscr{Y})$. Then we have by Lemma~\ref{6.2.4} and Def.~\ref{6.2.7}a$_1$ that $\mathbb{F}(\mathscr{Y},y_0)\subseteq\verb"A"_{\textrm{c,L}}(G)\cap\verb"LUC"(G)=\mathbb{F}(\mathscr{M}_{\!a},e^*)$. Thus, $(\mathscr{Y},y_0)$ is a factor of $(\mathscr{M}_{\!a},e^*)$.

\item (2): Suppose $\mathscr{Y}$ is a universal a.a. flow of $G$ with $y_0\in\verb"A"(\mathscr{Y})$.
Then by Structure Theorem of a.a. flow (Thm.~\ref{6.2.6}), there exists an almost 1-1 extension $\rho\colon(\mathscr{Y},y_0)\rightarrow(\mathscr{Y}_{\!eq}, y_{\!eq})$. By the universality of $\mathscr{M}_{\!eq}$, there is an extension
    $\phi\colon(\texttt{M}_{eq},e_{\!eq}^*)\rightarrow(Y_{eq},y_{\!eq})$. Since $\mathscr{M}_{\!eq}$ is an a.a. regular flow of $G$ with $e_{\!eq}^*$ an a.a. point, there exists an extension $\psi\colon(Y,y_0)\rightarrow(\texttt{M}_{eq},e_{\!eq}^*)$. Then $\rho=\phi\circ\psi$ so that $\phi$ is almost 1-1. However, since $\mathscr{M}_{\!eq}$ is equicontinuous, $\phi$ is open and 1-1. Thus, $\phi$ is an isomorphism from $\mathscr{M}_{\!eq}$ onto $\mathscr{Y}_{\!eq}$. If $\mathscr{Y}^\prime$ is an almost 1-1 extension of $\mathscr{M}_{\!eq}$, then $\mathscr{Y}^\prime$ is a.a. so that it is a factor of $\mathscr{Y}$. Thus, $\mathscr{Y}$ is a maximal almost 1-1 extension of $\mathscr{M}_{\!eq}$.

Conversely, suppose $\mathscr{Y}$ is a maximal almost 1-1 extension of $\mathscr{M}_{\!eq}$. Then $\mathscr{Y}$ is an a.a. flow. To prove that $\mathscr{Y}$ is a universal a.a. flow, it is enough to prove that $\mathscr{M}_{\!a}$ is a factor of $\mathscr{Y}$. Since $\mathscr{M}_{\!a}$ is an almost 1-1 extension of $\mathscr{M}_{\!eq}$ by (1) and the proved necessity of (2), hence $\mathscr{M}_{\!a}$ is a factor of $\mathscr{Y}$ by the maximality.

\item (3): Let $\phi$ is an endomorphism of $\mathscr{M}_{\!a}$ and $\eta\colon\mathscr{M}_{\!a}\rightarrow\mathscr{Z}$. Let $a^\prime\in A$ such that $a^\prime\in\eta(\phi(a^\prime))[A]$. Define $\pi^\prime\colon Z\rightarrow\verb"M"_{a}$ by $z\mapsto z(a^\prime)$, and $\varphi=\eta\circ\phi\circ\pi^\prime$. Then
    $\varphi$ is an endomorphism of $\mathscr{Z}$. To prove that $\phi$ is 1-1, suppose to the contrary that there are two points $w_1\not=w_2$ in $\verb"M"_{a}$ with $\phi(w_1)=\phi(w_2)$.
    We can choose two points $z_1$, $z_2\in Z$ such that $z_1(a^\prime)=w_1$ and $z_2(a^\prime)=w_2$. By $\varphi=\eta\circ\phi\circ\pi^\prime$, it follows that $\varphi(z_1)=\varphi(z_2)$. However, from $\z(a^\prime)=a^\prime$ we have that $\varphi(\z)=\eta(\phi(a^\prime))$ and then $a^\prime\in A^\prime=\varphi(\z)[A]$. Thus, $z_1|_{A^\prime}\not=z_2|_{A^\prime}$ and $\varphi(z_1)\not=\varphi(z_2)$ by Lemma~\ref{6.2.7}C, contrary to $\varphi(z_1)=\varphi(z_2)$.

\item \quad\quad Finally, let $\mathscr{X}_1$ and $\mathscr{X}_2$ be two universal a.a. flows and let $x_1\in\texttt{A}(\mathscr{X}_1)$. Then there exists an almost 1-1 extension $\rho_1\colon(\mathscr{X}_1,x_1)\rightarrow(\texttt{M}_{eq},e_{\!eq}^*)$ by (2). By university of $\mathscr{X}_1$, there is an extension $\phi\colon(\mathscr{X}_1,x_1)\rightarrow(\mathscr{X}_2,x_2)$. Moreover, by (2) and regularity of $\mathscr{M}_{\!eq}$, there also exists an almost 1-1 extension
$\rho_1\colon(\mathscr{X}_2,x_2)\rightarrow(\texttt{M}_{eq},e_{\!eq}^*)$. Since $\rho_1$ is 1-1 at $x_1$ and $\rho_2$ is 1-1 at $x_2$, $\phi$ is 1-1 at $x_1$ and $\phi$ is almost 1-1.
The proof is completed.
\end{proof}

\begin{6.2.7E}
Let $G$ be an abelian group and $\mathscr{X}$ an a.a. flow of $G$. If $\phi$ is an endomorphism of $\mathscr{X}$, then $\phi$ is almost 1-1 with $\phi^{-1}\phi(x)=\{x\}$ for all $x\in\verb"A"(\mathscr{X})$.
\end{6.2.7E}

\begin{proof}
Since $G$ is abelian, $X_{\!eq}$ is a compact Hausdorff topological group and $G<X_{\!eq}$. Then the statement follows from Theorem~\ref{6.2.6}.
The proof is complete.
\end{proof}

\begin{6.2.7F}[An alternative construction of $\texttt{M}_{a}$]
Let
$$
X^\prime=\prod_{f\in\texttt{A}_\textrm{c,L}(G)\cap\texttt{LUC}(G)}\texttt{H}_\textrm{p}^L[f], x_0=(f)_{f\in\texttt{A}_\textrm{c,L}(G)\cap\texttt{LUC}(G)}\in X^\prime \textrm{ and } X=\overline{Gx_0}\subseteq X^\prime.
$$
Then $\mathscr{X}=G\curvearrowright X$ is an a.a. flow with $x_0\in\texttt{A}(\mathscr{X})$.
\end{6.2.7F}

\begin{6.2.7G}
Let $(\mathscr{X},x_0)=G\curvearrowright(X,x_0)$ be as in Def.~\ref{6.2.7}F. Then $G\curvearrowright(\verb"M"_{a},e^*)\cong G\curvearrowright(X,x_0)$.
\end{6.2.7G}

\begin{proof}
By $\mathbb{F}(\mathscr{X},x_0)=\texttt{A}_\textrm{c,L}(G)\cap\texttt{LUC}(G)$ and Theorem~\ref{3.1}.
\end{proof}
\end{sse}

\begin{sthm}\label{6.2.8}
Let $G$ be an abelian group such that for each $t\in G$ with $t\not=e$, there exists some $f\in\verb"A"_\textrm{c}(G)\cap\verb"LUC"(G)$ with $tf\not=f$. Then $G\curvearrowright\verb"M"_{a}$ is free.
\end{sthm}

\begin{proof}
Suppose to the contrary that $G\curvearrowright\verb"M"_{a}$ is not free. Then there exists an element $t\in G$ with $t\not=e$ such that $tx_0=x_0$ for some point $x_0\in\verb"M"_{a}$. Since $G$ is abelian and $\overline{Gx_0}=\verb"M"_{a}$, it follows that $tx=x$ for all $x\in\verb"M"_{a}$. Then by Theorem~\ref{6.2.7}G, we can conclude a contradiction that $tf=f$ for all $f\in\verb"A"_\textrm{c}(G)\cap\verb"LUC"(G)$. The proof is complete.
\end{proof}

\begin{cor*}
If $G=\mathbb{Z}$ or $G=\mathbb{R}$, then $G\curvearrowright\verb"M"_{a}$ is free.
\end{cor*}

\begin{proof}
First note that given $t(\not=0)\in G$ we can choose some $f\in\verb"LUC"(G)$ such that $tf=f$ and $f>0$. Then $f$ is a.a., for it is periodic. This implies by Theorem~\ref{6.1.4} that for every $t(\not=0)\in G$ there exists some function $f\in\verb"A"_\textrm{c}(G)\cap\verb"LUC"(G)$ with $tf\not=f$. Thus, by Theorem~\ref{6.2.8}, $G\curvearrowright\verb"M"_{a}$ is free. The proof is completed.
\end{proof}

\subsection{Veech's hulls of a.a. functions}\label{s6.3}
Let $G$ be a Hausdorff topological group in $\S$\ref{s6.3}, unless otherwise specified.

\begin{sse}[Two canonical bi-flows]\label{6.3.1}
There are two canonical non-compact bi-flows (bi-transformation group) as follows:

\item[\;\;\textbf{a.}] $(G_{\!d},l^\infty(G),G_{\!d})$ with continuous phase mapping
\begin{enumerate}
\item[] $G_{\!d}\times l^\infty(G)\times G_{\!d}\xrightarrow{(\tau_1,f,\tau_2)\mapsto \tau_1f\tau_2}l^\infty(G)$
\end{enumerate}
where
\begin{enumerate}
\item[] $\tau_1f\tau_2(s)=f(\tau_2s\tau_1)\ \forall s\in G$;
\end{enumerate}
\item[\;\;\textbf{b.}] $(G_{\!d},l^\infty(G\times G),G_{\!d})$ with continuous phase mapping
\begin{enumerate}
\item[] $G_{\!d}\times l^\infty(G\times G)\times G_{\!d}\xrightarrow{(\tau_1,x,\tau_2)\mapsto \tau_1x\tau_2}l^\infty(G\times G)$
\end{enumerate}
where
\begin{enumerate}
\item[] $\tau_1x\tau_2(s,t)=x(s\tau_1,\tau_2t)\ \forall (s,t)\in G\times G$.
\end{enumerate}
\end{sse}

\begin{sse}[Veech's hulls of $\texttt{UC}$-functions on $G$]\label{6.3.2}
\item[\;\;\textbf{a.}] Following Veech \cite[p.730]{V65} we now define a continuous mapping
\begin{enumerate}
\item[] $\xi_\centerdot\colon l^\infty(G)\rightarrow l^\infty(G\times G)$
\end{enumerate}
by
\begin{enumerate}
\item[] $f\mapsto\xi_{f}\colon G\times G\xrightarrow{(s,t)\mapsto f(st)}\mathbb{C}\quad \forall f\in l^\infty(G)$.
\end{enumerate}
Clearly, $\xi_\centerdot\colon f\mapsto\xi_f$ is 1-1.

However, if $G$ is not abelian, then $\xi_{\tau_1f\tau_2}\not=\tau_1\xi_f\tau_2$ for $\tau_1,\tau_2\in G$ and $f\in l^\infty(G)$ in general.
So, $\xi$ is not a homomorphism from $(G_{\!d},l^\infty(G),G_{\!d})$ into $(G_{\!d},l^\infty(G\times G),G_{\!d})$.

\item[\;\;\textbf{b.}] Let $\texttt{UC}(G)=\texttt{LUC}(G)\cap\texttt{RUC}(G)$ as in \ref{2.2}.6.

\item[\;\;\textbf{c.}] Following Veech \cite{V65}, for all $f\in\texttt{UC}(G)$, \textit{Veech's hull} of $f$ may be defined as follows:
\begin{enumerate}
\item[] $X_{\!f}=\overline{G\xi_{f}}^\textrm{p}\ \left(=\overline{\xi_{f}G}^\textrm{p}\quad \because\tau\xi_{f}=\xi_{f}\tau\ \forall \tau\in G\right)$.
\end{enumerate}
By Lemma~\ref{4.2}, it follows that
\begin{enumerate}
\item[] $(G,X_{\!f},G)$: $G\times X_{\!f}\times G\xrightarrow{(\tau_1, x, \tau_2)\mapsto\tau_1x\tau_2} X_{\!f}$
\end{enumerate}
is a compact bi-flow.

Note that $\texttt{H}_\textrm{p}^L[f]\xrightarrow{\xi}X_{\!f}$ is continuous 1-1 onto, yet $\xi$ is not $G$-equivariant if $G$ is not abelian.
\end{sse}

\begin{slem}[{cf.~\cite[Lem.~4.1.1]{V65} for $G$ to be abelian, LC, $\sigma$-compact and first countable}]\label{6.3.3}
If $G$ is LC and $f$ is Bochner a.a. on $G$ (i.e., $f\in\verb"A"_\textrm{c}(G)$), then $f\in\verb"UC"(G)$ and $(G,X_{\!f},G)$ is a compact bi-flow.
\end{slem}

\begin{proof}
By \ref{6.3.2}c and 6. of Theorem~\ref{6.1.4}.
\end{proof}

\begin{sse}\label{6.3.4}
Let $(G,X,G)$ be a compact bi-flow. A point $x\in X$ is \textit{left a.a.} if $x$ is a.a. under $G\curvearrowright X$ (cf.~Def.~\ref{6.2.1}c), denoted by $x\in\verb"A"(G\curvearrowright X)$.
If $x$ is a.a. under $X\curvearrowleft G$, then it is \textit{right a.a.} and denoted $x\in\verb"A"(X\curvearrowleft G)$. Points which are both left and right a.a. are called \textit{a.a. points} under $(G,X,G)$, denoted $\verb"A"(X)$.

\item[\;\;\textbf{a.}] Clearly, if $G$ is LC and $\f\in\texttt{A}_\textrm{c}(G)$, then $G\curvearrowright X_{\!\f}$ and $X_{\!\f}\curvearrowleft G$ are both minimal by Lemma~\ref{6.1.5}, and $G\xi_{\f}=\xi_{\f}G\subseteq\verb"A"(X_{\!\f})$ by Lemma~\ref{6.3.3}.

\item[\;\;\textbf{b.}] If $f\in\texttt{UC}(G)$ such that $\xi_f\in\texttt{A}(G\curvearrowright X_{\!f})$, then $f\in\texttt{A}_\textrm{c,L}(G)$.
\end{sse}

Then by Structure Theorem of a.a. flows (Thm.~\ref{6.2.6}), $G\curvearrowright X_{\!\f}$ and $X_{\!\f}\curvearrowleft G$ are both a.a. flows so that they are almost 1-1 extensions of their maximal a.p. factors. In fact, Veech's Structure Theorem of a.a. function will give us more (see Thm.~\ref{6.3.11}).

\begin{slem}\label{6.3.5}
Let $G$ be LC and $f\in \verb"A"_\textrm{c}(G)$. Let $I=\{z\in\mathbb{C}\,|\,|z|\le\|f\|_\infty\}$.
Let $f^\beta\colon\beta G\rightarrow I$ be the canonical extension of $f$ and let $p,q\in\beta G$. Suppose $f^\beta(spqt)=f(st)$ for all $s,t\in G$. Then $f^\beta(sqpt)=f(st)$ for all $s,t\in G$; and moreover, $f^\beta(sq\!\cdot_{\!L}^{}\!pt)=f(st)=f^\beta(sp\!\cdot_{\!L}^{}\!qt)$ for all $s,t\in G$ under the left-topological semigroup structure $\cdot_{\!L}^{}$ on $\beta G$.
\end{slem}

\begin{proof}
First, $f\in\texttt{A}_\textrm{c,L}(G_{\!d})$ by Theorem~\ref{6.1.4}. Given a finite set $N\subseteq G$ and $\epsilon>0$ there exists by Lemma~\ref{6.1.10} a finite superset $M$ of $N$ and a positive real number $\delta<\epsilon$ such that
$$
C(f;M,\delta)\cdot C(f;M,\delta)\cdot C(f;M,\delta)^{-1}\subseteq C(f;N,\epsilon).
$$
Let $\{s_i\,|\,i\in\Lambda_1\}$, $\{t_j\,|\,j\in\Lambda_2\}$ be nets in $G$ with $s_i\to p$ and $t_j\to q$ in $\beta G$. By Lemma~\ref{6.1.9}, $C(f;M,\delta)$ is right-syndetic in $G_{\!d}$, and so for each $j\in\Lambda_2$ we can write $t_j=\tau_jr_j$ where $\tau_j\in C(f;M,\delta)$ and $r_j\in R$ for some finite subset $R$ of $G$. By choosing a subnet if necessary, we may assume $r_j=r\in R$ independent of $j$. Since $f^\beta(srpqr^{-1}t)=f(st)$ for all $s,t\in G$, we can choose $i\in\Lambda_1$ and $j\in\Lambda_2$ such that
$\tau_{ij}:=rs_it_jr^{-1}\in C(f;M,\delta)$. Then $t_js_i=\tau_j\tau_{ij}\tau_j^{-1}$. Since $\tau_{ij}, \tau_j\in C(f;M,\delta)$, it follows that $\max_{s,t\in N}|f(st_js_it)-f(st)|<\epsilon$. Passing to the limit in $i$ and then in $j$ or in $j$ and then in $i$ gives us that
$$
\max_{s,t\in N}|f^\beta(sqpt)-f(st)|\le\epsilon\quad\textrm{ and }\quad\max_{s,t\in N}|f^\beta(sq\!\cdot_{\!L}^{}\!pt)-f(st)|\le\epsilon.
$$
Since $\epsilon$, and then $N$ were arbitrary, hence $f^\beta(sq\!\cdot_{\!R/L}^{}\!pt)=f(st)$ for all $s,t\in G$. By symmetry, $f^\beta(sp\!\cdot_{\!L}^{}\!qt)=f(st)$ for all $s,t\in G$. The proof is completed.
\end{proof}

\begin{cor*}[{cf.~\cite[Lem.~3.1.1]{V65} for $G$ a discrete group}]
Let $G$ be LC and $f\in \verb"A"_\textrm{c}(G)$. If $x\in X_{\!f}$ and $\{t_n\}$ is a net in $G$ such that $\lim t_nx=\xi_f$, then $\lim xt_n=\xi_f$. (By symmetry if $\lim xt_n=\xi_f$, then $\lim t_nx=\xi_f$.)
\end{cor*}

\begin{proof}
Choosing $q\in\beta G$ with $x=q\xi_f$, we have that $x(s,t)=f^\beta(sqt)$ for all $s,t\in G$. Let (a subnet of) $t_n\to p$ in $\beta G$, then we have for all $s,t\in G$ that
$$f(st)=\xi_f(s,t)=px(s,t)=\lim x(st_n,t)=\lim f^\beta(st_nqt)=f^\beta(spqt).$$
Now by Lemma~\ref{6.3.5}, it follows that $f^\beta(sq\cdot_L^{}\!pt)=f(st)$ for all $s,t\in G$. So we have for $s,t\in G$ that
$$\lim xt_n(s,t)=\lim x(s,t_nt)=\lim f^\beta(sqt_nt)=f^\beta(sq\!\cdot_L^{}\!pt)=f(st).$$
Thus, $\lim xt_n=\xi_f$. The proof is completed.
\end{proof}

\begin{sse}[{Veech's relation~\cite[Def.~3.1.2]{V65} and \cite[Def.~5.1]{D20}}]\label{6.3.6}
Let $(G,X,G)$ be a compact bi-flow. For $x,y\in X$ we say that $(x,y)\in\verb"V"(G\curvearrowright X)$ if there exists a net $\{t_i\}$ in $G$ such that $z=\lim t_ix$ and $\lim t_i^{-1}z=y$. Similarly we could define $\verb"V"(X\curvearrowleft G)$. Write
\begin{enumerate}
\item[] $\verb"V"(X)=\verb"V"(G\curvearrowright X)\cap\verb"V"(X\curvearrowleft G)$.
\end{enumerate}
Clearly, $\texttt{V}(G\curvearrowright X)\subseteq\texttt{RP}(G\curvearrowright X)$ and $\verb"V"(X)$ need not be a symmetric relation on $X$ in general. However, it is symmetric in the a.a. case as follows:
\end{sse}

\begin{slem}\label{6.3.7}
Let $G$ be LC and $f\in \verb"A"_\textrm{c}(G)$. Let $x,y\in X_{\!f}$ such that $(x,y)\in\verb"P"(G\curvearrowright X_{\!f})$. If $\{t_n\}$ is a net in $G$ with $t_n\to q\in\beta G$ such that $y=q\xi_f$, then $\lim t_n^{-1}x=\lim xt_n^{-1}=\xi_f$. Thus, $\verb"P"(G\curvearrowright X_{\!f})\subseteq \verb"V"(G\curvearrowright X_{\!f})$.
\end{slem}

\begin{proof}
First, since $G\curvearrowright X_{\!f}$ is a compact minimal flow, we can choose $p\in\beta G$ with $px=py=\xi_f$. Let (a subnet of) $t_n^{-1}\to\alpha\in\beta G$ (independent of $\cdot_{\!R}$ and $\cdot_{\!L}$ on $\beta G$ so that $yt_n^{-1}\to y\alpha$) and $xt_n^{-1}\to z\in X_{\!f}$ (i.e., $x\alpha=z$) based on $X_{\!f}\curvearrowleft G$. By $\xi_f\in\texttt{A}(X_{\!f})$ and by Corollary to Lemma~\ref{6.3.5}, $\alpha y=\xi_f$ and $yt_n^{-1}\to\xi_f$ so that $\xi_f=y\alpha$.
By $py\alpha=\xi_f\alpha=px\alpha$, it follows that $(\xi_f,z)=(y\alpha,x\alpha)\in\texttt{P}(G\curvearrowright X_{\!f})$. By Lemma~\ref{6.1.6}, $z=\xi_f$. Thus, $xt_n^{-1}\to\xi_f$; and then $t_n^{-1}x\to\xi_f$ by Corollary to Lemma~\ref{6.3.5} again.
The proof is complete.
\end{proof}

\begin{cor*}[{cf.~\cite[Lem.~3.1.2, Cor.~3.1.2]{V65} by different approaches}]
Let $x,y\in X_{\!f}$ with $(x,y)\in\verb"V"(G\curvearrowright X_{\!f})$.
If $t_n\in G\to q\in\beta G$ with $x=q\xi_f$, then $\lim yt_n^{-1}=\lim t_n^{-1}y=\xi_f$ and $(y,x)\in\verb"V"(G\curvearrowright X_{\!f})$.
\end{cor*}

\begin{proof}
By \ref{6.2.3} and Lemma~\ref{6.3.7}.
\end{proof}

\begin{sthm}\label{6.3.8}
Let $G$ be LC and $f\in\verb"A"_\textrm{c}(G)$. Then
$$\verb"P"(G\curvearrowright X_{\!f})=\verb"RP"(G\curvearrowright X_{\!f})=\verb"V"(G\curvearrowright X_{\!f})=\verb"V"(X_{\!f}\curvearrowleft G)=\verb"V"(X_{\!f}).$$
\end{sthm}

\begin{proof}
Since $G\curvearrowright X_{\!f}$ is an a.a. flow, by Theorem~\ref{6.2.6} or \ref{6.2.3} it follows that
$\verb"P"=\verb"RP"\supseteq\verb"V"$ for $G\curvearrowright X_{\!f}$.
Then by Lemma~\ref{6.3.7}, $\verb"P"\subseteq\verb"V"$ for $G\curvearrowright X_{\!f}$.
Thus,
$$\verb"P"(G\curvearrowright X_{\!f})=\verb"RP"(G\curvearrowright X_{\!f})=\verb"V"(G\curvearrowright X_{\!f})$$
and
$$
\verb"P"(X_{\!f}\curvearrowleft G)=\verb"RP"(X_{\!f}\curvearrowleft G)=\verb"V"(X_{\!f}\curvearrowleft G).
$$
Finally, we have by Lemma~\ref{6.3.5} that $\verb"P"(G\curvearrowright X_{\!f})=\verb"P"(X_{\!f}\curvearrowleft G)=\verb"V"(X_{\!f})$. The proof is completed.
\end{proof}

\begin{6.3.8A}[{cf.~\cite[Thm.~3.1.1, Lem.~3.1.5]{V65}}]
Let $G$ be LC and $f\in\verb"A"_\textrm{c}(G)$. Then $\verb"V"(X_{\!f})$ is a bi-invariant closed equivalence relation on $X_{\!f}$.
\end{6.3.8A}

\begin{6.3.8B}
If $G$ is LC, then
$$\verb"A"(G\curvearrowright X_{\!f})=\verb"A"(X_{\!f}\curvearrowleft G)=\verb"A"(X_{\!f})$$
for all $f\in\verb"A"_\textrm{c}(G)$.
\end{6.3.8B}

\begin{6.3.8C}
In fact, $\texttt{P}(G\curvearrowright X)\subseteq\texttt{V}(G\curvearrowright X)$ for any minimal compact \textsl{metric} flow $G\curvearrowright X$, not necessarily a.a.; see \cite[Prop.~9.14]{F81} for $G=\mathbb{Z}$ and \cite[Thm.~5.2]{D20} for $G$ any semigroup. So, if $G$ is countable (or more generally, $\sigma$-compact) and if $f\in\verb"LUC"(G)$ such that $G\curvearrowright\verb"H"_\textrm{p}^L[f]$ is a minimal flow, then $\texttt{P}(G\curvearrowright \verb"H"_\textrm{p}^L[f])\subseteq\texttt{V}(G\curvearrowright \verb"H"_\textrm{p}^L[f])$.
\end{6.3.8C}

\begin{sse}[A semigroup structure on $X_{\!f}/\texttt{V}(X_{\!f})$]\label{6.3.9}
Let $G$ be LC and $f\in\verb"A"_\textrm{c}(G)$.
Write $[x]=\verb"V"[x]$, the $\verb"V"(X_{\!f})$-equivalence class of $x\in X_{\!f}$. If $x,y,z\in X_{\!f}$ with $x=\alpha\xi_{f}$, $y=\gamma\xi_f$, and $z=\alpha\gamma\xi_f$ for some $\alpha,\gamma\in\beta G$, we then define a binary operation:
$$
[x]\circ[y]=[z].
$$
\item[\;\;\textbf{a.}] If $x^\prime\in\verb"V"[x]$, $x^\prime=\alpha^\prime\xi_f$, $y^\prime\in\verb"V"[y]$, $y^\prime=\gamma^\prime\xi_f$, and if $z^\prime=\alpha^\prime\gamma^\prime\xi_f$, then $z^\prime\in\verb"V"[z]$ and therefore, $[x]\circ[y]=[z]$ depends only on the $\verb"V"$-classes $[x]$, $[y]$ of $x$ and $y$.
\begin{proof}
Since $(y,y^\prime)\in\verb"V"(X_{\!f})$ and $\verb"V"(X_{\!f})$ is a closed invariant relation, hence $(\alpha^\prime\gamma\xi_f,z^\prime)=(\alpha^\prime y,\alpha^\prime y^\prime)\in\verb"V"(X_{\!f})$. Since $(x,x^\prime)\in\verb"P"(G\curvearrowright X_{\!f})$ by Theorem~\ref{6.3.8}, there exists
$p\in\beta G$ with $p\alpha\xi_f=\xi_f=p\alpha^\prime\xi_f$. By $\xi_f\in\texttt{A}(X_{\!f})$ and letting $\gamma^{-1}=\lim t_n^{-1}$ for a net $t_n\in G\to\gamma$, we have that $\gamma\gamma^{-1}\xi_f=\xi_f$ and $\gamma\gamma^{-1}p\alpha\xi_f=\xi_f=\gamma\gamma^{-1}p\alpha^\prime\xi_f$. Now by Lemma~\ref{6.3.5}, it follows that $\gamma^{-1}p(\alpha\gamma)\xi_f=\gamma^{-1}p(\alpha^\prime\gamma)\xi_f$. Thus, $(\alpha\gamma\xi_f,\alpha^\prime\gamma\xi_f)=(z,\alpha^\prime\gamma\xi_f)\in\verb"P"(X_{\!f})$ and $(z,z^\prime)\in\verb"V"(X_{\!f})$ by Theorem~\ref{6.3.8}.
\end{proof}

\item[\;\;\textbf{b.}] Let $X_0=X_{\!f}/\verb"V"(X_{\!f})$, endowed with the quotient topology. Then
with phase mapping $(s,[x],t)\mapsto s[x]t=[sxt]$ for all $s,t\in G$ and $[x]\in X_0$, $(G,X_0,G)$ is a compact bi-flow.

\begin{proof}
Clearly, $G\curvearrowright X_0$ and $X_0\curvearrowleft G$ are both compact flows by Corollary~\ref{6.3.8}A. Now let $s,t\in G$ and $[x]\in X_0$. Since $s([x]t)=s[xt]=[sxt]$ and $(s[x])t=[sx]t=[sxt]$, then $s([x]t)=(s[x])t$ so that $G\curvearrowright X_0\curvearrowleft G$ is a compact bi-flow.
\end{proof}

\item[\;\;\textbf{c.}] Let $(G,X,G)$ and $(G,Y,G)$ be two compact bi-flows. Then we say that $\rho\colon(G,X,G)\rightarrow(G,Y,G)$ is an extension if $\rho\colon X\rightarrow Y$ is a continuous surjective map such that $\rho(sxt)=s\rho(x)t$ for all $s,t\in G$ and $x\in X$.
\end{sse}

\begin{sthm}[{cf.~\cite[Thm.~3.2.1]{V65}}]\label{6.3.10}
Let $G$ be LC and $f\in\verb"A"_\textrm{c}(G)$.
With the operation $[x]\circ[y]$ as in Def.~\ref{6.3.9}, $X_0$ is a compact Hausdorff topological group with identity $[\xi_f]$. That is,
\begin{enumerate}
\item[1)] $([x]\circ[y])\circ[z]=[x]\circ([y]\circ[z])$ for all $[x], [y], [z]\in X_0$.
\item[2)] $[x]\circ[\xi_f]=[\xi_f]\circ[x]=[x]$ for all $[x]\in X_0$.
\item[3)] Given $[x]\in X_0$ there is a point $[y]=[x]^{-1}\in X_0$ such that
$$
[x]\circ[y]=[y]\circ[x]=[\xi_f].
$$
\item[4)] The mapping $([x],[y])\mapsto[x]\circ[y]^{-1}$ is continuous from $X_0\times X_0$ to $X_0$.
\end{enumerate}
\end{sthm}

\begin{proof}[Proof (simpler than Veech's)]
\item[\;\;1):] Obvious by \ref{6.3.9} and $(\alpha\beta)\gamma=\alpha(\beta\gamma)$ for all $\alpha, \beta,\gamma\in\beta G$.

\item[\;\;2):] Obvious by \ref{6.3.9}, $\xi_f=e\xi_f$ and $e\alpha=\alpha e$ for all $\alpha\in\beta G$.

\item[\;\;3):] Let $x=\alpha\xi_f$ for some $\alpha\in\beta G$ and $t_n\in G\to\alpha$. Let $\alpha^{-1}=\lim t_n^{-1}$ in $\beta G$. Set $y=\alpha^{-1}\xi_f$. Then $\alpha\alpha^{-1}\xi_f=\alpha^{-1}\alpha\xi_f=\xi_f$ so that $[x]\circ[y]=[y]\circ[x]=[\xi_f]$.

Thus, by 1), 2) and 3) above, it follows that $X_0$ is a compact Hausdorff group.

\item[\;\;4):] To prove that $X_0$ is a topological group, according to Ellis's Joint Continuity Theorem it is enough to prove that $([x],[y])\mapsto[x]\circ[y]$ is separately continuous. For this, let $[x_i]\to [x]$ in $X_0$ and $[y]\in X_0$. We may assume $x_i\to x$ in $X_{\!f}$ and select $\gamma,\alpha,\alpha_i\in\beta G$ such that $x_i=\alpha_i\xi_f$, $\alpha_i\to\alpha$, $x=\alpha\xi_f$, and $y=\gamma\xi_f$. Then by $\alpha_i\gamma\to\alpha\gamma$ in $\beta G$ (noting $\beta G$ is a right-topological semigroup), we have that
    $[x_i]\circ[y]=[\alpha_i\gamma\xi_f]\to[\alpha\gamma\xi_f]=[x]\circ[y]$.
    This shows that $X_0$ is a right-topological group. Using $\cdot_{\!L}$ on $\beta G$, we can show that $X_0$ is a left-topological group.
    Thus, $([x],[y])\mapsto[x]\circ[y]$ is separately continuous and the proof is completed.
\end{proof}

\begin{sthm}[{Veech's Structure Theorem of a.a. function~\cite{V65}}]\label{6.3.11}
Let $G$ be LC and $f\in\verb"A"_\textrm{c}(G)$.
Let $\rho\colon X_{\!f}\rightarrow X_0$ be the canonical mapping. Then $(G,X_{\!f},G)$ is an almost 1-1 extension of $(G,X_0,G)$ such that $(G,X_0,G)$ is a compact-group bi-flow and
$\xi_f\in\verb"A"(X_{\!f})=\{x\in X_{\!f}\,|\,\{x\}=\rho^{-1}\rho(x)\}$.

If in addition $t\xi_f\not=\xi_f$ for all $t\in G$ and $t\not=e$, then $G\xrightarrow{t\mapsto[t\xi_f]=[\xi_f t]}X_0$ is an injection and $G\curvearrowright X_0$ and $X_0\curvearrowleft G$ are both free.
\end{sthm}

\begin{proof}
\item First part: By Theorems~\ref{6.3.8} and \ref{6.3.10}.

\item Second part: Let $tx=x$ for some $t\in G$ and some point $x\in X_0$. Then $txs=xs\ \forall s\in G$. Further $tx=x$ for all $x\in X$ so that
$t\xi_f=\xi_f$ and $t=e$. Thus, $G\curvearrowright X_0$ is free. By symmetry, $X_0\curvearrowleft G$ is also free. The proof is completed.
\end{proof}

\begin{sthm}\label{6.3.12}
Let $G$ be LC and $f\in\verb"A"_\textrm{c}(G)$. Then every endomorphism of $G\curvearrowright X_{\!f}$ (resp. $X_{\!f}\curvearrowleft G$) is almost 1-1.
\end{sthm}
\begin{proof}
See Proof of Theorem~\ref{6.2.7}E.
\end{proof}

\begin{srem}\label{6.3.13}
Lemma~\ref{6.2.5} and Theorem~\ref{6.3.11} indicates that if $G$ is not abelian, then $\verb"H"_\textrm{p}^L[f]$ need not be isomorphic to $X_{\!f}$ for any $f\in\verb"A"_\textrm{c}(G)\cap\verb"UC"(G)$. In fact, there exist a.a. functions $f$ on non-abelian discrete groups such that the maximal a.p. factor of $\verb"H"_\textrm{p}^L[f]$ is not a compact group. So Lemma~\ref{6.2.5} and Theorem~\ref{6.3.11} are two essentially different structure theorems of a.a. functions.
\end{srem}

\begin{sse}[A group structure of a.p. functions]\label{6.3.14}
Recall that a function $f$ on $G$ is \textit{left a.p.} in the sense of von Neumann (1934) if $Gf$ is relatively compact in $(\mathbb{C}^G,\|\cdot\|_\infty)$. Clearly, if $f$ is left a.p., then $f\in l^\infty(G)$. If $f$ is both left a.p. and right a.p., then $f$ is said to be \textit{a.p.} and denoted $f\in\texttt{AP}(G)$. This is independent of the topology on $G$.

\begin{6.3.14A}[{J.~von Neumann; cf.~\cite[Thm.~18.1]{HR}}]
Let $f\in\mathbb{C}^G$. Then the following conditions are pairwise equivalent:
\begin{enumerate}
\item $Gf$ is relatively compact in $(\mathbb{C}^G,\|\cdot\|_\infty)$.

\item $fG$ is relatively compact in $(\mathbb{C}^G,\|\cdot\|_\infty)$.

\item Given $\varepsilon>0$, there are two finite sets $\{a_1,\dotsc,a_n\}$ and $\{b_1,\dotsc,b_m\}$ in $G$ such that for all $a, b\in G$ there are $i\in\{1,\dotsc,n\}$ and $j\in\{1,\dotsc,m\}$ with $\|bfa-b_jfa_i\|_\infty<\varepsilon$.

\item Given $\varepsilon>0$, there exists a finite set $\{b_1,\dotsc,b_m\}$ in $G$ such that for all $b\in G$ there is some $k\in\{1,\dotsc, m\}$ with $\|b\xi_f-b_k\xi_f\|_\infty<\varepsilon$.
\end{enumerate}
Consequently, $f$ is  left a.p. iff it is right a.p. iff $\overline{Gf}^{\|\cdot\|_\infty}$ is compact iff $\overline{fG}^{\|\cdot\|_\infty}$ is compact.
\end{6.3.14A}

\begin{proof}
See Appendix~\ref{B}
\end{proof}

\begin{6.3.14B}
Let $G$ be LC and $f\in\texttt{UC}(G)$. If $x=\alpha\xi_f$, $y=\gamma\xi_f$, and $z=\alpha\gamma\xi_f$ in $X_{\!f}$ where $\alpha,\gamma\in\beta G$, then define
$$
x\circ y=z.
$$
Note that $(X_{\!f},\circ)$ is not a group in general. However, we can obtain the following for a.p. functions:
\end{6.3.14B}

\begin{6.3.14C}
Let $G$ be LC and let $f\in\texttt{UC}(G)$. Then $f\in\verb"AP"(G)$ iff $(X_{\!f},\circ)$ is a compact Hausdorff topological group with $e=\xi_f$.
\end{6.3.14C}

\begin{proof}
\item \textit{Necessity:} Suppose $f\in\verb"AP"(G)$. Then $f\in\texttt{A}_\textrm{c}(G)$ by Note of Lemma~\ref{6.2.5}. By Lemma~\ref{6.3.14}A, $G\curvearrowright X_{\!f}$ and $X_{\!f}\curvearrowleft G$ are both compact a.p./equicontinuous minimal flows. Thus, $\rho\colon X_{\!f}\rightarrow X_0$, as in Theorem~\ref{6.3.11}, is continuous 1-1 onto so that $(X_{\!f},\circ)$ is a compact Hausdorff topological group with $e=\xi_f$.

\item \textit{Sufficiency:} Suppose $(X_{\!f},\circ)$ is a compact Hausdorff topological group with $e=\xi_f$. Then that $s\xi_f\circ t\xi_f=st\xi_f\ \forall s,t\in G$ implies that $\alpha\gamma\xi_f=\alpha\xi_f\circ\gamma\xi_f$ for all $\alpha,\gamma\in\beta G$. Now given $t\in G$, $tx=t\xi_f\circ x$ for all $x\in X_{\!f}$. Thus, $G\curvearrowright X_{\!f}$ is equicontinuous. Further, $G\curvearrowright X_{\!f}$ is an a.p. flow so that $G\curvearrowright\texttt{H}_\textrm{p}^L[f]$ is an a.p. flow. Thus, $G\curvearrowright\texttt{H}_\textrm{p}^L[f]$ is equicontinuous so $f\in\texttt{AP}(G)$ by Note of Lemma~\ref{6.2.5}. The proof is completed.
\end{proof}

Note that if $G$ is non-abelian, then there need not exist a natural group structure on $\texttt{H}_\textrm{p}^L[f]=\overline{Gf}^{\|\cdot\|_\infty}$ for all $f\in\texttt{AP}(G)$. In view of this, Theorem~\ref{6.3.14}C is of interest.
\end{sse}

\subsection{Bohr a.a. functions}\label{s6.4}
Let $G$ be a topological group in $\S$\ref{s6.4}, unless otherwise specified. We shall characterize $\verb"A"_\textrm{c}(G)$ in terms of the so-called Bohr a.a. functions on $G$ (Thm.~\ref{6.4.6}).

\begin{sse}[$\texttt{bUC}$-functions on $G\times G$]\label{6.4.1}
\item[\;\;\textbf{a.}] A bounded function $\Phi\colon G\times G\rightarrow\mathbb{C}$ is said to be \textit{bi-uniformly continuous}, denoted $\Phi\in\verb"bUC"(G\times G)$, if for every $\epsilon>0$ there exists some $V\in\mathfrak{N}_e(G)$ such that ${\sup}_{s,t\in G}|\Phi(v_1s,tv_2)-\Phi(s,t)|<\epsilon$ for all $v_1,v_2\in V$.

\begin{6.4.1B}
Let $(G,X,G)$ be a compact bi-flow.
Given $x\in X$ and $F\in\mathcal{C}(X)$, let
$$
F_{\!x}^b\colon G\times G\rightarrow\mathbb{C},\quad (s,t)\mapsto F_{\!x}^b(s,t)=F(sxt).
$$
Then $F_{\!x}^b\in\verb"bUC"(G\times G)$. If $x\in\verb"A"(X)$, then $F_{\!x}^b$ is an a.a. point under $(G_{\!d},l^\infty(G\times G),G_{\!d})$.
\end{6.4.1B}

\begin{proof}
Let $\epsilon>0$. There exists some $V\in\mathfrak{N}_e(G)$ such that
$$
{\max}_{y\in X}|F(v_1yv_2)-F(y)|<\epsilon\quad \forall v_1,v_2\in V.
$$
Thus,
$$
{\sup}_{s,t\in G}|F(v_1sxtv_2)-F(sxt)|<\epsilon\quad \forall v_1,v_2\in V.
$$
This implies that
$$
{\sup}_{s,t\in G}|F_{\!x}^b(v_1s,tv_2)-F_{\!x}^b(s,t)|<\epsilon\quad \forall v_1,v_2\in V.
$$
Finally, we note that if $t_nx\to y$ then $t_nF_{\!x}^b\to_\textrm{p}F_{\!y}^b$, and if $xt_n\to y$ then $F_{\!x}^bt_n\to_\textrm{p} F_{\!y}^b$.
The proof is completed.
\end{proof}
\end{sse}

\begin{sse}[{bi-hulls of $\texttt{UC}$-functions on $G$}]\label{6.4.2}
\item[\;\;\textbf{a.}] Define a continuous mapping
\begin{enumerate}
\item[] $\theta\colon l^\infty(G)\rightarrow l^\infty(G\times G)$
\end{enumerate}
by
\begin{enumerate}
\item[] $f\mapsto\theta_{\!f}\colon G\times G\xrightarrow{(s,t)\mapsto f(ts)}\mathbb{C}$.
\end{enumerate}
It is different with $\xi_f$ in Def.~\ref{6.3.2}a.

\item[\;\;\textbf{b.}] As in Def.~\ref{6.3.1}b, for all $x\in l^\infty(G\times G)$ and $\tau_1,\tau_2\in G$, let
\begin{enumerate}
\item[] $\tau_1x\tau_2\colon G\times G\xrightarrow{(s,t)\mapsto x(s\tau_1,\tau_2t)}\mathbb{C}$.
\end{enumerate}
Then $\tau_1x\tau_2\in l^\infty(G\times G)$ such that $(\tau_1x)\tau_2=\tau_1(x\tau_2)$.

\item[\;\;\textbf{c.}] Given $f\in\texttt{UC}(G)$, we can define a bi-hull of $f$ as follows:
\begin{enumerate}
\item[] $W_{\!f}=\overline{G\theta_{\!f}G}^\textrm{p}$.
\end{enumerate}
By Lemma~\ref{4.2}, it follows that
\begin{enumerate}
\item[] $G\times W_{\!f}\times G\xrightarrow{(\tau_1, x, \tau_2)\mapsto\tau_1x\tau_2} W_{\!f}$
\end{enumerate}
is compact bi-flow.

It should be noted that if $G$ is non-abelian, then $t\theta_{\!f}\not=\theta_{\!f}t$ for $t\not=e$ in general. So $G\curvearrowright W_{\!f}$ and $W_{\!f}\curvearrowleft G$ need not be minimal.
\end{sse}

\begin{slem}\label{6.4.3}
If $f\in\verb"UC"(G)$, then $\theta_{\!f}\in\verb"bUC"(G\times G)$; that is, for every $\varepsilon>0$ there exists some $V\in\mathfrak{N}_e(G)$ such that
$$
{\sup}_{s,t\in G}|f(tvs)-f(ts)|<\varepsilon\quad \forall v\in V.
$$
\end{slem}

\begin{proof}
Define $F\colon W_{\!f}\rightarrow\mathbb{C}$ by $w\mapsto F(w)=w(e,e)$. Then $F\in\mathcal{C}(W_{\!f})$. Then by \ref{6.4.2}c and Lemma~\ref{6.4.1}B, it follows that given $\varepsilon>0$ there exists $V\in\mathfrak{N}_e(G)$ such that
$$
\sup_{s,t\in G}|F_{\!\theta_{\!f}}^b(v_1s,tv_2)-F_{\!\theta_{\!f}}^b(s,t)|<\varepsilon\quad\forall v_1,v_2\in V.
$$
Then
$$
\sup_{s,t\in G}|F(v_1s\theta_{\!f}tv_2)-F(s\theta_{\!f}t)|<\varepsilon\quad\forall v_1,v_2\in V.
$$
Thus,
$$
\sup_{s,t\in G}|\theta_{\!f}(v_1s, tv_2)-\theta_{\!f}(s,t)|<\varepsilon\quad\forall v_1,v_2\in V.\leqno{(\ref{6.4.3}\textrm{a})}
$$
So
$$
\sup_{s,t\in G}|f(tv_2v_1s)-f(ts)|<\varepsilon\quad\forall v_1,v_2\in V.\leqno{(\ref{6.4.3}\textrm{b})}
$$
The proof is completed.
\end{proof}

\begin{sse}\label{6.4.4}
We say that a function $f$ on $G$ is a \textit{Bohr a.a. function} on $G$ if for each finite set $N\subset G$ and all real number $\varepsilon>0$ there exists a set
$B=B(N,\varepsilon)\subseteq G$ such that:
\begin{enumerate}
\item[(a)] $B$ is syndetic in $G$ (\textbf{Caution}: not necessarily in $G_{\!d}$);
\item[(b)] $B$ is symmetric (i.e., $B=B^{-1}$);
\item[(c)] If $b\in B$, then $\max_{s,t\in N}|f(sbt)-f(st)|<\varepsilon$;
\item[(d)] If $b_1,b_2\in B$, then $\max_{s,t\in N}|f(sb_1^{-1}b_2t)-f(st)|<2\varepsilon$;
\item[(e)] $f\in\verb"UC"(G)$.
\end{enumerate}

\begin{note*}
A function $f$ on $G$ is called a \textit{continuous Bohr a.a. function} on $G$ in the sense of Veech (cf.~\cite[Def.~4.1.2]{V65}) if for each finite set $N\subset G$ and all real number $\varepsilon>0$ there exists a set
$B=B(N,\varepsilon)\subseteq G$ such that
\begin{enumerate}
\item[(a)$^\prime$] $B$ is relatively dense in $G$
\end{enumerate}
and such that (b), (c), (d), and (e) above.
\end{note*}
If $G$ is discrete, then $l^\infty(G)=\verb"UC"(G)$ so that Def.~\ref{6.4.1} coincides with Veech \cite[Def.~2.1.1]{V65} in this case.
\end{sse}

Clearly, a continuous Bohr a.a. function on $G$ in the sense of Veech is a Bohr a.a. function on $G$ in our sense. It turns out that vice versa (Thm.~\ref{6.4.7}).
For this, we need a lemma.

\begin{slem}\label{6.4.5}
Let $K$ be a compact subset of $G$ and $\{k_j\,|\,j\in\Lambda\}$ a net in $K$ with $k_j\to k\in K$. Given $j\in\Lambda$ let $j^\prime\in\Lambda$ such that $j^\prime\ge j$. Then $\{k_{j^\prime}\,|\,j^\prime\in\Lambda\}$ is a subnet of $\{k_j\}$ and $k_{j^\prime}\to k$ such that $k_j^{-1}k_{j^\prime}\to e$.
\end{slem}

\begin{proof}
For $j_1^\prime, j_2^\prime\in\Lambda$, there exists some $j\in\Lambda$ with $j_1^\prime\le j$ and $j_2^\prime\le j$. So $j_1^\prime\le j^\prime$ and $j_2^\prime\le j^\prime$. This shows that $\{k_{j^\prime}\,|\,j^\prime\in\Lambda\}$ is a subnet of $\{k_j\}$. Thus, $k_{j^\prime}\to k$. Since $G$ is a topological group, hence $k_j^{-1}\to k^{-1}$ and $k_j^{-1}k_{j^\prime}\to e$. The proof is completed.
\end{proof}

\begin{sthm}\label{6.4.6}
Let $G$ be LC and $f$ a function on $G$. Then $f$ is a Bohr a.a. function on $G$ (cf.~Def.~\ref{6.4.4}) iff $f\in\verb"A"_\textrm{c}(G)$.
\end{sthm}

\begin{note*}
See Veech \cite[Thm.~2.2.1]{V65} for $G$ to be a discrete group, and, \cite[Thm.~4.1.1]{V65} for $G$ to be an abelian, LC, $\sigma$-compact and first countable group. Theorem~\ref{6.4.6} gives us a positive solution to \cite[Question~2.7]{D20} for a.a. functions instead of a.a. points of flows.
\end{note*}

\begin{proof}
\item \textit{Necessity:} Assume $f$ is a Bohr a.a. function on $G$ (cf.~Def.~\ref{6.4.4}). Let $\{t_j\}_{j\in\Lambda}$ be a net in $G$ such that
$t_jf\to_\textrm{p}\phi\in l^\infty(G)$ and $t_j^{-1}\phi\to_\textrm{p}h\in l^\infty(G)$. To prove that $f\in\verb"A"_{\textrm{c,L}}(G)$, we need only prove that $\phi\in\mathcal{C}(G)$ and $h=f$. Indeed, by Def.~\ref{6.4.4}(e) and Lemma~\ref{4.2}, $\phi\in\verb"LUC"(G)\subseteq\mathcal{C}(G)$. It remains to show $h=f$.
Suppose to the contrary that $h\not=f$. Then there exist some $t\in G$ and $\varepsilon>0$ such that $|f(t)-h(t)|\ge 5\varepsilon$. Let $B=B(N,\varepsilon)$, where $N=\{e,t\}$, be given by Def.~\ref{6.4.4}. Passing to a subnet if necessary, we may assume
$t_j=k_js_j$ such that $s_j\in B$ and $k_j\to k\in G$.
Let $j$ and then $j^\prime\ge j$ be chosen from $\Lambda$ so large that
$|t_j^{-1}t_{j^\prime}f(t)-h(t)|<\varepsilon$.
However, by Lemma~\ref{6.4.5} and Lemma~\ref{6.4.3}, it follows that
\begin{equation*}\begin{split}
|t_j^{-1}t_{j^\prime}f(t)-f(t)|&=|s_j^{-1}k_j^{-1}k_{j^\prime}s_{j^\prime}f(t)-f(t)|\\
&=|f(ts_j^{-1}k_j^{-1}k_{j^\prime}s_{j^\prime})-f(t)|\\
&\le|f(ts_j^{-1}k_j^{-1}k_{j^\prime}s_{j^\prime})-f(ts_j^{-1}s_{j^\prime})|+|f(ts_j^{-1}s_{j^\prime})-f(t)|\\
&\le 3\varepsilon
\end{split}\end{equation*}
eventually. So $|f(t)-h(t)|\le 4\varepsilon$, contrary to $|f(t)-h(t)|\ge 5\varepsilon$.
Thus, $f\in\verb"A"_{\textrm{c,L}}(G)$. By symmetry, $f\in\verb"A"_{\textrm{c,R}}(G)$, and, $f\in\verb"A"_\textrm{c}(G)$.

\item \textit{Sufficiency:} Suppose $f\in\verb"A"_\textrm{c}(G)$. By Lemma~\ref{6.3.3}, $f\in\verb"UC"(G)$; and so Def.~\ref{6.4.4}(e) holds. Finally, it is easy to check that conditions (a)$\sim$(d) in Def.~\ref{6.4.4} are all fulfilled by Lemmas~\ref{6.1.9} and \ref{6.1.10}.

The proof of Theorem~\ref{6.4.6} is therefore completed.
\end{proof}

\begin{sthm}\label{6.4.7}
Let $G$ be LC and $f$ a function on $G$. Then $f$ is a Bohr a.a. function on $G$ (Def.~\ref{6.4.4}) iff it is a continuous Bohr a.a. function on $G$ in the sense of Veech.
\end{sthm}

\begin{proof}
Sufficiency is obvious. Necessity follows easily from Theorems~\ref{6.4.6} and \ref{6.1.4}.
\end{proof}
\begin{appendix}
\section{Appendix: Proof of Lemma~\ref{6.1.10}}\label{A}
Lemma~\ref{6.1.10} is the most technical tool for proving the structure theorem of a.a. functions. We now present its proof here for reader's convenience.

\begin{6.1.10}[{cf.~\cite[Lem.~2.1.2]{V65} or \cite[Thm.~1.9]{D20}}]
	Let $f\in l^\infty(G)$ be a.a. on $G_{\!d}$. Given $\epsilon>0$ and a finite set $N\subset G$, there exist $\delta>0$ and a finite superset $M$ of $N$ such that $\sigma^{-1}\tau\in C(f;N,\epsilon)$ for all $\sigma,\tau\in C(f;M,\delta)$.
\end{6.1.10}

\begin{proof}[Proof (Veech 1965~\cite{V65})]
	Suppose to the contrary that, for some finite set $N\subset G$ and some $\epsilon>0$, for every finite superset $M$ of $N$ and all $\delta>0$ with $\delta<\epsilon$ there must exist $\sigma,\tau\in C(f;M,\delta)$ with $\sigma^{-1}\tau\notin C(f;N,\epsilon)$.
	
	Choose a sequence of positive real numbers $\{\delta_n\}_{n=1}^\infty$ with $\epsilon>\delta_n\searrow0$ and $\sum_n\delta_n<\infty$. Let a sequence $\{M_n\}$ of finite supersets of $N$ together with a sequence $(\sigma_n,\tau_n)$ of pairs of elements of $G$ be chosen as follows:
	
	Let $M_1^\prime=N\cup\{e\}$ and set $M_1=M_1^\prime\cup(M_1^\prime)^{-1}$. $M_1$ is a finite superset of $N$, so by assumption there exists $(\sigma,\tau)=(\sigma_1,\tau_1)$ such that $\sigma,\tau\in C(f;M_1,\delta_1)$ but $\sigma^{-1}\tau\notin C(f;N,\epsilon)$. Having chosen $M_1,\dotsc, M_k$ and $(\sigma_1,\tau_1), \dotsc,(\sigma_k,\tau_k)$, we set $M_{k+1}^\prime=M_k\cdot M_k\cdot N_k$ where $N_k=\{e, \sigma_k,\tau_k,\sigma_k^{-1}\tau_k\}$, and we define $M_{k+1}=M_{k+1}^\prime\cup(M_{k+1}^\prime)^{-1}$. Then $M_{k+1}$ is a finite superset of $N$ and there exists a pair $(\sigma,\tau)=(\sigma_{k+1},\tau_{k+1})$ such that $\sigma,\tau\in C(f;M_{k+1},\delta_{k+1})$ but $\sigma^{-1}\tau\notin C(f;N,\epsilon)$. The construction then proceeds by induction.
	
	Set $G_0=\bigcup_{k=1}^\infty M_k$. Clearly, $G_0$ is a subgroup of $G_{\!d}$ and $f$ when restricted to $G_0$ is again an a.a. function on $G_0$.
	
	A sequence $\{\alpha_k\}$ of elements of $G_0$ is now defined as follows:
	$\alpha_1=\tau_1$, $\alpha_2=\sigma_1$,  $\alpha_{2k+1}=\tau_1\dotsm\tau_{k+1}$, and $\alpha_{2k+2}=\alpha_{2k-1}\sigma_{k+1}$ for all $k\ge1$.
	Then, for all $s\in G_0$ and $t\in N$, by $t\in M_{j+1}$ and $s\alpha_{2j-1}\in M_{j+1}$ as $j$ sufficiently large, we have
	\begin{equation*}\begin{split}
			&\left|f(s\alpha_{2j+1}t) -f(s\alpha_{2j+2}t) \right|\\
			&\quad\le\left|f(s\alpha_{2j+1}t) -f(s\alpha_{2j-1}t) \right|+\left|f(s\alpha_{2j-1}t) -f(s\alpha_{2j+2}t) \right|\\
			&\quad\le\delta_{j+1}+\delta_{j+1}\ (\to0\ \textrm{ as } j\to\infty);
	\end{split}\end{equation*}
	and, for $k>j$
	\begin{equation*}\begin{split}
			&\left|f(s\alpha_{2j+1}t) -f(s\alpha_{2k+1}t) \right|\\
			&\quad\le\sum_{i=0}^{k-j-1}\left|f(s\alpha_{2(j+i)+1}t) -f(s\alpha_{2(j+i+1)+1}t) \right|\\
			&\quad\le\sum_{i=0}^{k-j-1}\delta_{j+i+2}\ (\to0\ \textrm{ as } j\to\infty).
	\end{split}\end{equation*}
	Thus, for all $s\in G_0$ and $t\in N$ the limit $\lim_{k\to\infty}f(s\alpha_kt)=g(s,t)$ exists.
	Since $f$ is a.a. on $G_0$, $\alpha_kt\in G_0$, and $\alpha_ktf(s)\to g(s,t)$ for all $s\in G_0$, we have for all $s\in G_0$ and $t\in N$ that
	$\lim_{j\to\infty}g(s(\alpha_jt)^{-1},t)=f(s)$. Since $st\in G_0$ for all $s,t\in N$, we can choose $j$ so large that
	$$
	\left|f(st)-\lim_{k\to\infty}f(s\alpha_j^{-1}\alpha_kt))\right|=\left|f(st)-\lim_{k\to\infty}f(st(\alpha_jt)^{-1}\alpha_kt)\right|\le\frac{\epsilon}{4}.
	$$
	Further, for all $s\in G_0$ and $t\in N$, as $j$ sufficiently large we can choose $k>j$ so large that
	$$
	\left|f(st)-f(s\alpha_j^{-1}\alpha_kt))\right|<\frac{\epsilon}{2};
	$$
	moreover,
	\begin{equation*}\begin{split}
			&\max_{s,t\in N}\left|f(s\sigma_{j+1}^{-1}\tau_{j+1}t) -f(s\sigma_{j+1}^{-1}\tau_{j+1}\dotsm \tau_{k}t) \right|\\
			&\quad\le\max_{s,t\in N}\sum_{i=0}^{k-j-1}\left|f(s\sigma_{j+1}^{-1}\tau_{j+1}\dotsm\tau_{j+1+i}t) -f(s\sigma_{j+1}^{-1}\tau_{j+1}\dotsm \tau_{j+1+i+1}t) \right|\\
			&\quad\le\sum_{i=0}^{k-j-1}\delta_{j+1+i+1}\\
			&\quad<\frac{\epsilon}{2}.
	\end{split}\end{equation*}
	Noting that $\sigma_{j+1}^{-1}\tau_{j+1}\dotsm\tau_k=\sigma_{j+1}^{-1}\tau_{j}^{-1}\dotsm\tau_{1}^{-1}\tau_1\dotsm\tau_j\tau_{j+1}\dotsm\tau_k=\alpha_{2(j+1)}^{-1}\alpha_{2k-1}$, as $2(j+1)<2k-1$ sufficiently large we have that
	$$\max_{s,t\in N}\left|f(st) -f(s\sigma_{j+1}^{-1}\tau_{j+1}t) \right|<\epsilon,$$
	contrary to that $\sigma_{j+1}^{-1}\tau_{j+1}\notin C(f;N,\epsilon)$. This proves the lemma.
\end{proof}

\section{Appendix: Proof of Lemma~\ref{6.3.14}A}\label{B}
\begin{6.3.14A}[{J.~von Neumann; cf.~\cite[Thm.~18.1]{HR}}]
Let $f\in\mathbb{C}^G$. Then the following conditions are pairwise equivalent:
\begin{enumerate}
\item $Gf$ is relatively compact in $(\mathbb{C}^G,\|\cdot\|_\infty)$.

\item $fG$ is relatively compact in $(\mathbb{C}^G,\|\cdot\|_\infty)$.

\item Given $\varepsilon>0$, there are two finite sets $\{a_1,\dotsc,a_n\}$ and $\{b_1,\dotsc,b_m\}$ in $G$ such that for all $a, b\in G$ there are $i\in\{1,\dotsc,n\}$ and $j\in\{1,\dotsc,m\}$ with $\sup_{x\in G}|f(axb)-f(a_ixb_j)|<\varepsilon$.

\item Given $\varepsilon>0$, there exists a finite set $\{b_1,\dotsc,b_m\}$ in $G$ such that to any $b\in G$ there is some $1\le k\le m$ with $\sup_{x,y\in G}|f(xby)-f(xb_ky)|<\varepsilon$.
\end{enumerate}
Consequently, $f$ is  left a.p. iff it is right a.p. iff $\overline{Gf}^{\|\cdot\|_\infty}$ is compact iff $\overline{fG}^{\|\cdot\|_\infty}$ is compact.
\end{6.3.14A}

\begin{proof}
(A slight modification of Proof of \cite[Thm.~18.1]{HR}) It is obvious that 3.$\Rightarrow$1. and 2.; and moreover, 4.$\Rightarrow$1. and 2.

\item 1.$\Rightarrow$4.: Assume 1. Given $\varepsilon>0$, there is a finite set $\{a_1,\dotsc,a_m\}$ in $G$ such that $\{a_jf\}_{j=1}^m$ is a $\varepsilon/4$-mesh in $Gf$. Set $A_j=\{t\in G\,|\,\|tf-a_jf\|_\infty<\varepsilon/4$ for $j=1,\dotsc,m$. Let
$B_1, B_2, \dotsc, B_n$ be the non-empty sets in the family
$$
A_{l_1}a_1^{-1}\cap A_{l_2}a_2^{-1}\cap\dotsm\cap A_{l_m}a_m^{-1},\quad (l_1,\dotsc,l_m)\in\{1,\dotsc,m\}^m.
$$
It is obvious that $\bigcup_{k=1}^nB_k=G$. Indeed, for all $t\in G$,
\begin{equation*}\begin{split}
t\in A_{l_1}a_1^{-1}\cap\dotsm\cap A_{l_m}a_m^{-1}&\Leftrightarrow ta_j\in A_{l_j}\  \textrm{ for }j=1,\dotsc,m\\
&\Leftrightarrow\|ta_jf-a_{l_j}f\|_\infty<\varepsilon/4\  \textrm{ for }j=1,\dotsc,m,
\end{split}\end{equation*}
then for $ta_1,\dotsc,ta_m$ we can select indices $l_1,\dotsc,l_m\in\{1,\dotsc,m\}$ such that $t\in\bigcap_{j=1}^mA_{l_j}a_j^{-1}=B_k$ for some $k\in\{1,\dotsc,n\}$.
Now let $b\in G$; then we let $B_{k}\in\{B_1,\dotsc,B_n\}$ with $b\in B_{k}$. Let $(x,y)\in G\times G$. Select $j\in\{1,\dotsc,m\}$ such that $y\in A_{j}$. Then
\begin{equation*}\begin{split}
|f(xby)-f(xb_{k}y)|&\le|f(xby)-f(xba_{j})|+|f(xba_{j})-f(xb_{k}a_{j})|\\
&\qquad +|f(xb_{k}a_{j})-f(xb_{k}y)|\\
&<\varepsilon
\end{split}\end{equation*}
Thus 4. holds.

\item 2.$\Rightarrow$4.: Similar to 1.$\Rightarrow$4.

\item 1.$\Leftrightarrow$2.: Obvious.

\item 1.$\Rightarrow$3.: Assume 1.; and then 2. holds. Thus there are $\varepsilon/2$-meshes $\{fa_i\}_{i=1}^n$ and $\{b_jf\}_{j=1}^m$ in $fG$ and $Gf$, respectively. For all $x,a,b\in G$, we have
$$
|f(axb)-f(a_ixb_j)|\le|f(axb)-f(axb_j)|+|f(axb_j)-f(a_ixb_j)|<\varepsilon
$$
for some $i\in\{1,\dotsc,n\}$ and some $j\in\{1,\dotsc,m\}$. That is, 3. holds. The proof is completed.
\end{proof}

\section{Appendix: Veech's approximation of a.a. functions via a.p. ones}\label{C}
Let $G$ be a discrete group here. Veech's Approximation Theorem says that a function on $G$ is a.a. iff it is the pointwise limit of a ``jointly a.a.'' net of a.p. functions on $G$. This theorem is important for the Fourier analysis of a.a. functions (see \cite[$\S$4.2]{V65}). However, its proof available in \cite{V65} is very complicated (see \cite[Lem.~3.3.1 and Thm.~3.3.2]{V65}). Here we will present a simple weak version (Thm.~\ref{C3}).

Let $(\mathcal{F},\le)$ be the directed system of finite subsets of $G$ with $\le$ = $\subseteq$; and set $N^2=NN$ and $\mathcal{F}_N=\{M\in\mathcal{F}\,|\,M\ge N\}$ for $N\in\mathcal{F}$ in this Appendix.

\begin{se}\label{C1}
\item[\;\;\textbf{a.}] A net $\{f_N\}_{N\in\mathcal{F}}$ of a.a. functions on $G$ shall be called \textit{jointly a.a.} (cf.~\cite[Def.~3.3.1]{V65}) if it is uniformly bounded, and if for each $M\in\mathcal{F}$ and $\varepsilon>0$ there exists a set $B=B(M,\varepsilon)\subset G$ and an index $N_0\in\mathcal{F}$ such that if $N\ge N_0$ then
\begin{enumerate}
\item[i)] $B=B^{-1}$;
\item[ii)] $B$ is relatively dense in $G$;
\item[iii)] If $b\in B$ then $\max_{s,t\in M}|f_N(sbt)-f_N(st)|<\varepsilon$; and
\item[iv)] If $b_1,b_2\in B$ then $\max_{s,t\in M}|f_N(sb_1b_2t)-f_N(st)|<2\varepsilon$.
\end{enumerate}
Then Veech's Approximation Theorem can be stated as follows:
\begin{thm*}[{cf.~\cite[Thm.~3.3.2]{V65}}]
A function $f$ on $G$ is a.a. iff it is the pointwise limit of a jointly a.a. net of a.p. functions on $G$.
\end{thm*}

\item[\;\;\textbf{b.}] A net $\{f_N\}_{N\in\mathcal{F}}$ of a.a. functions on $G$ shall be called \textit{$(M,\varepsilon)$-jointly a.a.}, where $M\in\mathcal{F}$ and $\varepsilon>0$, if there exists a set $B=B(M,\varepsilon)\subset G$ such that if $N\ge M$, then
\begin{enumerate}
\item[i)] $B=B^{-1}$,
\item[ii)] $B$ is syndetic in $G$ (not necessarily relative dense),
\item[iii)] If $b\in B$ then $\max_{s,t\in M}|f_N(sbt)-f_N(st)|<\varepsilon$,
\item[iv)] If $b_1,b_2\in B$ then $\max_{s,t\in M}|f_N(sb_1b_2t)-f_N(st)|<2\varepsilon$, and
\item[v)] $\|f_N\|_\infty\le\|f\|_\infty$,
\end{enumerate}
\end{se}

\begin{lem}[{comparable with~\cite[Lem.~3.3.1]{V65}}]\label{C2}
Let $f\in\verb"A"_\textrm{c}(G)$ and let $\rho\colon X_{\!f}\rightarrow X_0$ be the canonical almost 1-1 extension as in Theorem~\ref{6.3.11}. Then, for all $M\in\mathcal{F}$ and $\varepsilon>0$, there exists a net $\{\Psi_{\!N}\}_{N\in\mathcal{F}}$ of continuous functions on $X_0$ such that:
\begin{enumerate}
\item[i)] $\Psi_{\!N}(t[\xi_f])=f(t)\ \forall t\in N^2$.
\item[ii)] With each $\tau\in M^2$ is associated a closed $U_\tau^M\in\mathfrak{N}_{\tau[\xi_f]}(X_0)$ such that if $N\ge M$ then
$\textmd{Var}(\Psi_{\!N}|U_\tau^M)\le\textmd{Var}(\Phi|_{\rho^{-1}[U_\tau^N]})\le \varepsilon$, where
$$
\Phi\colon X_{\!f}\xrightarrow{x\mapsto x(e,e)}X_0\textrm{ and }\textmd{Var}(h|_V)={\sup}_{v_1,v_2\in V}|h(v_1)-h(v_2)|.
$$
\item[iii)] $\|f\|_\infty=\|\Psi_{\!N}\|_\infty$.
\end{enumerate}
\end{lem}

\begin{proof}
Let $M\in\mathcal{F}$ and $\varepsilon>0$ be any given. Since $\Phi$ is continuous, there exists $U_\tau\in\mathfrak{N}_{\tau\xi_f}(X_{\!f})$, for all $\tau\in M^2$, such that $\textmd{Var}(\Phi|_{U_\tau})\le \varepsilon$. Since $\rho$ is 1-1 at each point of $G\xi_f$ and $\rho(\tau\xi_f)=\tau[\xi_f]$ for all $\tau\in G$, we can choose a closed neighborhood $U_\tau^M$ of $\tau[\xi_f]$ in $X_0$, for each $\tau\in M^2$, such that:
\begin{enumerate}
\item[1)] $\rho^{-1}[U_\tau^M]\subseteq U_\tau$; and
\item[2)] $U_\tau^M\cap U_\sigma^M=\emptyset$ if $\tau[\xi_f]\not=\sigma[\xi_f]$, for all $\tau,\sigma\in M^2$.
\end{enumerate}
Let $N\in\mathcal{F}$ such that $M\le N$.
Now define a function $\psi_N\colon N^2[\xi_f]\rightarrow\mathbb{C}$ by $\tau[\xi_f]\mapsto f(\tau)=\Phi(\tau\xi_f)$.
If $\tau_1\not=\tau_2$ in $G$ with $\tau_1[\xi_f]=\tau_2[\xi_f]$, then $\tau_1\xi_f=\tau_2\xi_f$ so $f(\tau_1)=f(\tau_2)$. Thus, $\psi_N$ is a well-defined continuous function on $N^2[\xi_f]$. Since $U_\tau^M\cap N^2[\xi_f]$ is finite for each $\tau\in M^2$, By Tietze's Theorem we can extend $\psi_N$ from $U_\tau^M\cap N^2[\xi_f]$ to $U_\tau^M$, for each $\tau\in M^2$, such that $\textmd{Var}(\psi_N|_{U_\tau^M})\le\textmd{Var}(\Phi|_{U_\tau})$.
Finally, using Tietze's Theorem again for $\psi_N\colon N^2[\xi_f]\cup\bigcup_{\tau\in M^2}U_\tau^M\rightarrow\{z\in\mathbb{C}\,|\,|z|\le\|f\|_\infty\}$, we can find the desired function $\Psi_{\!N}\colon X_0\rightarrow\mathbb{C}$. The proof is completed.
\end{proof}

\begin{thm}\label{C3}
$f\in\verb"A"_\textrm{c}(G)$ iff for every $M\in\mathcal{F}$ and $\varepsilon>0$, $f$ is the pointwise limit of an $(M,\varepsilon)$-jointly a.a. net $\{f_N\}_{N\in\mathcal{F}}$ of a.p. functions on $G$.
\end{thm}

\begin{proof}
\item \textit{Sufficiency:} Obvious by Theorem~\ref{6.4.6}.

\item \textit{Necessity:} Suppose $f\in \texttt{A}_\textrm{c}(G)$. First let $M\in\mathcal{F}$ and $\varepsilon>0$. Let $\{\Psi_{\!N}\}_{N\in\mathcal{F}}$ be the net of continuous functions on $X_0$ constructed in Lemma~\ref{C2}. Define
$f_N(t)=\Psi_{\!N}^{[\xi_f]}(t)\,(=\Psi_{\!N}(t[\xi_f]))$ for all $t\in G$. Since $G\curvearrowright X_0$ is equicontinuous, $f_N$ is a.p. on $G$; and $\|f_N\|_\infty=\|f\|_\infty$ for each $N\in\mathcal{F}$. Also $f_N(t)\to f(t)$ for all $t\in G$, since $f_N(t)=f(t)$ as soon as $t\in N^2$.

We are to produce a set $B=B(M,\varepsilon)$ in $G$ in accordance with Def.~\ref{C1}b. Let $U_{s^\prime t^\prime}^M$, for $s^\prime, t^\prime\in M$, be given in Lemma~\ref{C1}. Since $\rho^{-1}[U_{s^\prime t^\prime}^M]$ is a neighborhood of ${s^\prime t^\prime}\xi_f$ in $X_{\!f}$, there exists a finite superset $M_0$ of $M$ and $\delta>0$ such that whenever $\tau\in G$ such that
$$
\max_{s,t\in M_0}|s^\prime\tau t^\prime \xi_f(s,t)-s^\prime t^\prime\xi_f(s,t)|<\delta,
$$
then $s^\prime\tau t^\prime[\xi_f]\in U_{s^\prime t^\prime}^M$ for all $s^\prime, t^\prime\in M$. Thus, by applying Lemma~\ref{6.1.10} with $N=M_0M\cup MM_0$ and $\epsilon=\delta$, there exists a set $B\subset G$ with $e\in B$ such that $B=B^{-1}$ is syndetic in $G$ and that $s^\prime BBt^\prime[\xi_f]\in U_{s^\prime t^\prime}^M$ for all $s^\prime,t^\prime\in M$.
Now, if $N\ge M$, and if $b_1,b_2\in B$, then
$$
\max_{s,t\in M}|f_N(sb_1b_2t)-f_N(st)|\le\max_{\tau\in M^2}\textmd{Var}(\Phi|_{\rho^{-1}[U_\tau^M]})\le\varepsilon.
$$
Thus, $B=B(M,\varepsilon)$ satisfies the conditions of Def.~\ref{C1}b, and $\{f_N\}$ is an $(M,\varepsilon)$-jointly a.a. net of a.p. functions.
The proof is completed.
\end{proof}
\end{appendix}

\subsection*{Acknowledgements}
This work was supported by National Natural Science Foundation of China (Grant No. 11790274) and PAPD of Jiangsu Higher Education Institutions.

\end{document}